\documentclass{article}

\usepackage{amsmath}
\usepackage{amssymb}
\usepackage{latexsym}
\usepackage{verbatim}
\usepackage[final]{graphicx}
\usepackage{psfrag}

\usepackage{amsthm}

\newtheorem{theorem}{Theorem}
\newtheorem{proposition}{Proposition}
\newtheorem{lemma}{Lemma}

\newtheorem{remark}{Remark}
\newtheorem{assumption}{Assumption}

\newtheorem{example}{Example}

{\begin{list}               %    with flush left bullets.
    {$\bullet$ \hfill}{
        \setlength{\leftmargin}{\parindent}
        \setlength{\parsep}{0.04\baselineskip}
        \setlength{\itemsep}{0.5\parsep}
        \setlength{\labelwidth}{\leftmargin}
        \setlength{\labelsep}{0em}}
    }
{\end{list}}

\providecommand{\eref}[1]{\eqref{eq:#1}}  % call \eqref from amstex
\providecommand{\cref}[1]{Chapter~\ref{chap:#1}}
\providecommand{\sref}[1]{Section~\ref{sec:#1}}
\providecommand{\fref}[1]{Figure~\ref{fig:#1}}

\providecommand{\R}{\ensuremath{\mathbb{R}}}
\providecommand{\C}{\ensuremath{\mathbb{C}}}
\providecommand{\N}{\ensuremath{\mathbb{N}}}

\providecommand{\abs}[1]{\lvert#1\rvert}
\providecommand{\absb}[1]{\big\vert#1\big\vert}

\providecommand{\norm}[1]{\lVert#1\rVert}
\providecommand{\norms}[1]{\Vert#1\Vert}
\providecommand{\inprod}[1]{\langle#1\rangle}
\providecommand{\set}[1]{\left\{#1\right\}}

\providecommand{\bydef}{\overset{\text{def}}{=}}

\renewcommand{\vec}[1]{\ensuremath{\boldsymbol{#1}}}
\providecommand{\mat}[1]{\ensuremath{\boldsymbol{#1}}}

\providecommand{\mA}{A} 
\providecommand{\mB}{B}

\providecommand{\mI}{I}  
  
\providecommand{\mM}{\mat{M}}

\providecommand{\mY}{Y}
\providecommand{\mZ}{\mat{Z}}

 \providecommand{\ve}{\vec{e}}
\providecommand{\vg}{\vec{g}}

\providecommand{\vu}{\vec{u}} 
\providecommand{\vx}{\vec{x}} \providecommand{\vy}{\vec{y}}
\providecommand{\vz}{\vec{z}} \providecommand{\vi}{\vec{i}}

\providecommand{\vv}{\vec{v}}

\usepackage[margin=1in]{geometry}

\usepackage{subcaption}
\usepackage[square,numbers]{natbib}
\numberwithin{equation}{section}

\usepackage{color}

\usepackage{booktabs,multirow}

\usepackage{enumitem}
\usepackage{pmat}
\usepackage{commath}
\usepackage{mathtools}
\mathtoolsset{showonlyrefs}

\usepackage{algorithm,algpseudocode}

\usepackage{hyperref}
 \hypersetup{
     colorlinks=true,
     linkcolor=blue,
     filecolor=blue,
     citecolor = blue,      
     urlcolor=cyan,
 }

\usepackage{esdiff}

\usepackage{pgf,tikz,psfrag}

\usepackage{authblk}

\usepackage{dsfont}
\newcommand{\charfn}{\mathds{1}}

\renewcommand{\bydef}{\coloneqq}

\providecommand{\co}{\mu}
\providecommand{\coa}{\gamma_a}
\providecommand{\cob}{\gamma_b}
\providecommand{\coc}{\gamma_c}
\providecommand{\coch}{\widehat\gamma_c}
\providecommand{\mi}{^{[i]}}
\providecommand{\mn}{^{[n]}}
\providecommand{\idi}{\alpha}
\providecommand{\idj}{\beta}
\providecommand{\iq}{\alpha}
\providecommand{\jq}{\beta}

\providecommand{\ratio}{\kappa}
\renewcommand{\i}{\mathrm{i}}
\providecommand{\et}{\omega}
\providecommand{\rp}{\phi}

\DeclareMathOperator{\diag}{diag}

\DeclareMathOperator{\sign}{sign}
\DeclareMathOperator{\tr}{tr}

\DeclareMathOperator{\rank}{rank}

\newcommand{\unif}[1]{\mathsf{Unif}(#1)}
\newcommand{\taumeasure}{\mathsf{s}_{d}^1}

\providecommand{\sph}{\mathcal{S}^{d-1}}
\providecommand{\unifsp}{\unif{\sph}}
\providecommand{\qlp}[2]{C(#1, #2)}
\providecommand{\Al}{A_\mathrm{LO}}
\providecommand{\Ah}{A_\mathrm{HI}}
\providecommand{\SH}{S}

\newcommand*{\tran}{^{\mkern-1.5mu\mathsf{T}}}
\providecommand{\EE}{\mathbb{E}}
\renewcommand{\P}{\mathbb{P}}

\providecommand{\Nk}[1][k]{N_{#1}}
\providecommand{\har}[2]{Y_{#1, #2}}
\providecommand{\normp}[2]{\norm{#1}_{L^{#2}}}
\providecommand{\re}{\mathfrak{Re}}
\providecommand{\im}{\mathfrak{Im}}
\providecommand{\Eid}{t}
\providecommand{\rad}{\nu}
\providecommand{\Gm}{J}

\providecommand{\xcoeff}{\beta}

\providecommand{\dxi}{\mathcal{E}}
\providecommand{\qt}{\widetilde{q}}

\title{An Equivalence Principle for the Spectrum of Random Inner-Product Kernel Matrices with Polynomial Scalings}

\author{Yue M. Lu and Horng-Tzer Yau}%
\affil{Harvard University}

\date{\today}

\begin{document}

%\markboth{}{Lu and Yau: Equivalence Principle}

\maketitle

\begin{abstract}
We investigate random matrices whose entries are obtained by applying a nonlinear kernel function to pairwise inner products between $n$ independent data vectors, drawn uniformly from the unit sphere in $\R^d$. This study is motivated by applications in machine learning and statistics, where these kernel random matrices and their spectral properties play significant roles. We establish the weak limit of the empirical spectral distribution of these matrices in a polynomial scaling regime, where $d, n \to \infty$ such that $n / d^\ell \to \ratio$, for some fixed $\ell \in \N$ and $\ratio \in (0, \infty)$. Our findings generalize an earlier result by Cheng and Singer, who examined the same model in the linear scaling regime (with $\ell = 1$).

Our work reveals an equivalence principle: the spectrum of the random kernel matrix is asymptotically equivalent to that of a simpler matrix model, constructed as a linear combination of a (shifted) Wishart matrix and an independent matrix sampled from the Gaussian orthogonal ensemble. The aspect ratio of the Wishart matrix and the coefficients of the linear combination are determined by $\ell$ and the expansion of the kernel function in the orthogonal Hermite polynomial basis. Consequently, the limiting spectrum of the random kernel matrix can be characterized as the free additive convolution between a Marchenko-Pastur law and a semicircle law. We also extend our results to cases with data vectors sampled from isotropic Gaussian distributions instead of spherical distributions.
\end{abstract}

\tableofcontents

% !TEX root = equivalence.tex

\section{Introduction}

Let $\vx_1, \ldots, \vx_n \in \R^d$ be a set of independent data vectors drawn uniformly from the unit sphere $\mathcal{S}^{(d-1)}$. Consider a random matrix $A \in \R^{n \times n}$ with entries
\begin{equation}\label{eq:A_f}
A_{ij} \bydef \begin{cases} \frac{1}{\sqrt{n}} \, f_d(\sqrt{d}\, \vx_i\tran \vx_j) &\text{if } i \ne j\\
0&\text{if } i = j
\end{cases},
\end{equation}
where $f_d: \R \mapsto \R$ is a (nonlinear) ``kernel'' function. Its empirical spectral distribution (ESD) is denoted by
\begin{equation}\label{eq:esd_def}
\varrho_A(x) = \frac{1}{n} \sum_{i = 1}^n \delta_{\lambda_i}(x)
\end{equation}
where $\lambda_1, \ldots, \lambda_n$ are the eigenvalues of $A$, and $\delta_{\lambda}$ denotes the counting measure. In this work, we establish the weak-limit of the ESD $\varrho_A$ when $d, n \to \infty$ such that $n / d^\ell  \to \ratio \in (0, \infty)$, for some fixed $\ratio \in \R$ and $\ell \in \N$.

Our study of this model is motivated by recent problems in machine learning, statistics, and signal processing, where random matrices like \eref{A_f} and their spectral properties play crucial roles. Examples include kernel methods (such as kernel-PCA \cite{scholkopf1998NonlinearComponent} and kernel-SVM \cite{scholkopf2002Learningkernels}), covariance thresholding procedures \cite{bickel2008Covarianceregularization,deshpande2014Sparsepca}, nonlinear dimension reduction \cite{belkin2003LaplacianEigenmaps}, and probabilistic matrix factorization \cite{mnih2007ProbabilisticMatrix}. Moreover, the closely-related non-Hermitian version of \eref{A_f}, where $A_{ij} = \frac{1}{\sqrt n} f_d(\sqrt{d} \vx_i \vy_j)$ for two sets of vectors $\{\vx_i\}_{i \le n}$ and $\{\vy_j\}_{j \le p}$, appears in the random feature model \cite{rahimi2007RandomFeatures,louart2017RandomMatrix,hastie2020SurprisesHighDimensional,penningtonNonlinearRandomMatrix2019}, an interesting theoretical model for large random neural networks.

In this work, we examine a general high-dimensional polynomial scaling regime, where the number of samples $n$ can grow to infinity proportional to $d^\ell$, for some positive integer $\ell$. The linear regime, with $\ell =1$, was previously studied by \citet{cheng2013SpectrumRandom}, who showed that the ESD of $A$ converges to a deterministic limit distribution, with its Stieltjes transform characterized as the solution to a cubic equation (see \sref{related} for further discussions on related work in the literature). Interestingly, \citet{fan2017SpectralNorm} observed that the limiting ESD of $A$ (as determined by the aforementioned cubic equation) is equivalent to that of a simpler matrix model, which consists of the linear combination of a (shifted) Wishart matrix and an independent matrix sampled from the Gaussian orthogonal ensemble (GOE). This observation leads to an intriguing interpretation: the \emph{nonlinear} kernel $f_d$ is asymptotically equivalent to a \emph{linear} and \emph{noisy} transformation.

%\begin{equation}\label{eq:B_f_linear}
%B = \frac{\co \sqrt{d}}{\sqrt{n}}  (X\tran X - I) + \gamma H,
%\end{equation}
%where $X = [\vx_1, \ldots, \vx_n] \in \R^{d \times n}$ is a matrix whose columns are the data vectors, $H$ is an independent matrix sampled from the Gaussian orthogonal ensemble (GOE), and $\co, \gamma$ are two constants that depend on $f_d$. 

%Comparing the models in \eref{A_f} and \eref{B_f_linear}, one may interpret the effect of the \emph{nonlinear} kernel function $f_d$ as a \emph{noisy linear} transformation applied to the shifted Gram matrix $X\tran X - I$.

%The equivalence between the ESDs of $A$ and $B$ has an interesting interpretation, namely, the \emph{nonlinear} kernel $f_d$ is asymptotically equal to a \emph{linear} and \emph{noisy} transformation.

The main contribution of this work is to demonstrate that the aforementioned characterization, derived under the linear scaling regime, represents a special case of a more general \emph{equivalence principle}. This principle holds when $n / d^\ell  \to \ratio \in (0, \infty)$ for any $\ell \in \N$ and under mild conditions on the kernel $f_d$. Specifically, we establish in Theorem~\ref{thm:equivalence_f} that the ESD of $A$ is asymptotically equivalent to that of
\[
B =  \frac{\co_\ell}{\sqrt{n N_\ell}} (W\tran W - N_\ell I) + \gamma_\ell H
\]
where $W \in \R^{N_\ell \times n}$ is an i.i.d. standard Gaussian matrix with aspect ratio $N_\ell / n \to 1/(\ratio \ell!)$, and $H$ is a GOE matrix independent of $W$. Both constants $\co_\ell$ and $\gamma_\ell$ depend on $\ell$ and can be determined by expanding $f_d$ in the orthogonal Hermite polynomial basis. As a direct consequence of this equivalence principle, the limiting ESD of $A$ can be characterized simply as a free additive convolution between a (shifted) Marchenko-Pastur (MP) law and a semicircle law.

The remainder of the paper is organized as follows. In \sref{equivalence}, we begin by examining the special case of polynomial kernel functions, with the equivalence principle formalized in Theorem~\ref{thm:equivalence} and a heuristic explanation provided in \sref{heuristics}. We address the case of more general nonlinear functions in \sref{general_f}, stating the corresponding asymptotic characterizations in Theorem~\ref{thm:equivalence_f}. \sref{numerical} presents several numerical experiments to illustrate our theory, while \sref{related} discusses related work in the literature. We dedicate \sref{proof_equivalence} to the proof of Theorem~\ref{thm:equivalence} and gather auxiliary results in the appendix. In \sref{gaussian}, we extend our findings to cases where data vectors are sampled from the isotropic Gaussian distribution rather than the spherical distribution. Finally, we conclude the paper in \sref{summary}, discussing potential extensions of our results and open problems.

\subsection{Notation}

Before delving into the technical details, we first establish some notations employed throughout this paper.

\emph{Common sets}: $\N$ represents the set of positive integers, and $\N_0 = \N \cup \set{0}$. For each $n \in \N$, $[n] \coloneqq \set{1, 2, \ldots, n}$, and $(n)_k \bydef n(n-1)\ldots (n-k+1)$ denotes the falling factorial. For $i,j \in \N$, $\charfn_{ij}$ represents the Kronecker delta function, \emph{i.e.}, $\charfn_{ij} = 1$ if $i = j$ and $\charfn_{ij} = 0$ otherwise. The data vectors' dimension is denoted by $d$. The unit sphere in $\R^d$ is expressed as $\mathcal{S}^{d-1}$. Throughout the paper, we consistently use $z$ to represent a complex number in the upper-half plane $\C_+$, and we define
\[
\eta \bydef \im(z) > 0.
\]
Additionally, we assume that $\abs{\re(z)} \le \tau^{-1}$ and $\tau \le \eta \le \tau^{-1}$, for some global constant $\tau > 0$. 

%We regard it as the fundamental scaling parameter in this paper. The total number of such vectors is denoted by $n \bydef n_d$.

\emph{Probability distributions}: $\unifsp$ denotes the uniform probability measure on $\mathcal{S}^{d-1}$. For two independent vectors $\vx, \vy \sim \unifsp$, we define
\begin{equation}\label{eq:tau_measure}
\taumeasure \bydef \mathsf{Law}(\sqrt{d} \,\vx\tran \vy).
\end{equation}
Due to the rotational symmetry of $\unifsp$, we also have $\sqrt{d}\, \vx\tran \ve_1 \sim \taumeasure$, where $\ve_1 = (1,  0, \ldots, 0)$. Given a random variable $X$, its $L^p$ norm is denoted by $\normp{X}{p} \bydef (\EE \abs{X}^p)^{1/p}$. We say a symmetric random matrix is drawn from the GOE if its law is equivalent to that of $(F + F\tran)/ \sqrt{2n}$, where $F$ is an $n \times n$ matrix with i.i.d. standard normal entries. 

\emph{Stochastic order notation}: In our proof, we utilize a concept of high-probability bounds known as \emph{stochastic domination}. This notion, first introduced in \cite{erdos2013LocalSemicircle,erdos2017Dynamicalapproach}, provides a convenient way to account for low-probability exceptional events where some bounds may not hold.  Consider two families of nonnegative random variables:
\[
X = \big(X^{(d)}(u) : d \in \N, u \in U^{(d)}\big), \quad Y = \big(Y^{(d)}(u) : d \in \N, u \in U^{(d)}\big),
\]
where $U^{(d)}$ is a possibly $d$-dependent parameter set. We say that $X$ is \emph{stochastically dominated} by $Y$, uniformly in $u$, if for every (small) $\varepsilon > 0$ and (large) $D > 0$ we have
\[
\sup_{u \in U^{(d)}} \P[X^{(d)}(u) > d^\varepsilon Y^{(d)}(u)] \le d^{-D}
\]
for sufficiently large $d \ge d_0(\varepsilon, D)$. If $X$ is stochastically dominated by $Y$, uniformly in $u$, we use the notation $X \prec Y$. Moreover, if for some complex family $X$ we have $\abs{X} \prec Y$, we also write $X = \mathcal{O}_\prec(Y)$. This stochastic order notation should not be mistaken for the conventional big $\mathcal{O}$ notation, which we will also use in this paper: for two \emph{deterministic} sequences $X^{(d)}$ and $Y^{(d)}$, we write $X = \mathcal{O}_\alpha(Y)$, with some parameter $\alpha$, if $\abs{X} \le C(\alpha) Y$ for all sufficiently large $d$. Here, the constant $C(\alpha)$ may depend on $\alpha$.

\emph{Vectors and matrices:} For a vector $\vv \in \R^{n}$, its $\ell_2$ norm is denoted by $\norm{\vv}$. For a matrix $\mA \in \R^{n \times n}$, $\norm{\mA}_{\mathsf{op}}$ and $\norm{\mA}_\mathsf{F}$ denote the operator (spectral) norm and the Frobenius norm of $\mA$, respectively. Additionally, $\norm{\mA}_\infty \bydef \max_{i,j \in [n]} \abs{A(i,j)}$ denotes the entry-wise $\ell_\infty$ norm. We use $\vec{1}$ to denote $(1,1, \ldots, 1)$, and $\mI$ is an identity matrix. Their dimensions can be inferred from the context. The trace of $\mA$ is written as $\tr(\mA)$. Lastly, for an $n \times n$ Hermitian matrix $A$, the Stieltjes transform of its empirical spectral distribution is $s_A(z) \bydef \frac{1}{n} \tr(A - zI)^{-1}$. Here, the matrix $(A - zI)^{-1}$ is the resolvent of $A$.
% !TEX root = equivalence.tex

\section{An Asymptotic Equivalence Principle}
\label{sec:equivalence}

\subsection{Polynomial Kernels}

We begin by examining the special case when the kernel function $f_d$ in \eref{A_f} is a degree-$L$ polynomial, for some $L \in \N$. (The case of more general nonlinear functions is addressed in \sref{general_f}.) For reasons that will be clarified later (see \sref{heuristics}), we expand $f_d$ in the following form:
\begin{equation}\label{eq:f_polynomial}
f_d(x) = \sum_{0 \le k \le L} \co_k q_k^{(d)}(x)
\end{equation}
where $\set{\co_k}_k$ is a set of expansion coefficients, and $q_{k}^{(d)}(x)$ denotes the $k$th Gegenbauer polynomial \cite{dai2013ApproximationTheory, efthimiou2014Sphericalharmonics}. The Gegenbauer polynomials, also known as the ultraspherical polynomials in the literature \cite{ismailClassicalQuantumOrthogonal2005}, form a set of orthogonal polynomial basis with respect to the probability measure $\taumeasure$ defined in \eref{tau_measure}. Specifically, we have $\mathsf{deg}\, q_{k}^{(d)}(x) = k$ and
\begin{equation}\label{eq:ortho_poly}
\EE_{\xi \sim \taumeasure} \big[q_{i}^{(d)}(\xi) q_{j}^{(d)}(\xi)\big] = \charfn_{ij},
\end{equation}
where $\charfn_{ij}$ is the Kronecker delta.

%Let $\unifsp$ denote the uniform distribution on the unit sphere in $\R^d$. For $\vx \sim \unifsp$, we use $\taumeasure$ to represent the probability measure of $\sqrt{d}\, \ve_1\tran \vx$. 
%Let $\vx$ be a vector drawn uniformly from $\mathcal{S}^{d-1}$. Consider a random variable $\xi = \sqrt{d} \, \ve_1\tran\vy$. 

The coefficients of these polynomials can be determined by performing the Gram-Schmidt procedure on the monomial basis $\set{0, 1, x, x^2, \ldots}$ and by using the explicit formula for the moments of $\xi$ [see \eref{xi_moments}]. The first few polynomials in the sequence are
\begin{equation}\label{eq:Gegenbauer_low_degree}
q_{0}^{(d)}(x) = 1, \qquad q_{1}^{(d)}(x) = x, \qquad q_{2}^{(d)}(z) = \frac{1}{\sqrt{2}}\sqrt{\frac{d+2}{d-1}}(x^2 - 1),
\end{equation}
and
\begin{equation}\label{eq:ultra_deg3}
\qquad q_{3}^{(d)}(x) = \frac{1}{\sqrt{6}}\sqrt{\frac{d+4}{d-1}}\left(\frac{d+2}{d}x^3 - 3x\right).
\end{equation}
Note that the coefficients of the Gegenbauer polynomials depend on the dimension $d$. To simplify the notation, we will suppress this dependence by writing throughout the paper
\begin{equation}\label{eq:polynomial_s_d}
q_k(x) \bydef q_{k}^{(d)}(x) \quad\text{and} \quad \widetilde{q}_k(x) \bydef q_{k}^{(d-1)}(x).
\end{equation}
In Appendix~\ref{appendix:ortho_poly}, we compile a list of properties of Gegenbauer polynomials that will be used in our analysis. 

For each $k \in \N_0$, define a matrix $\mA_k \in \R^{n \times n}$ such that
\begin{equation}\label{eq:Ak}
(A_k)_{ij} \bydef \begin{cases} \frac{1}{\sqrt{n}} \, q_{k}(\sqrt{d}\, \vx_i\tran \vx_j) &\text{if } i \ne j\\
0&\text{if } i = j
\end{cases},
\end{equation}
where $\set{\vx_i}_{i \in [n]}$ is a collection of independent vectors drawn from $\unifsp$. Considering \eref{f_polynomial}, we can express the inner-product kernel matrix in \eref{A_f} as
\begin{equation}\label{eq:A}
\mA =  \sum_{k = 0}^L \co_k A_k = \Big(\sum_{k=0}^{\ell-1} + \sum_{k = \ell}^L\Big) \co_k A_k,
\end{equation}
where $\ell \in \N$ and $\ell \le L$. The main result of this paper is to establish an asymptotic \emph{equivalence principle} for the empirical spectral distribution (ESD) of $A$: when $n = \ratio d^\ell + o(d^\ell)$ for some $\ratio > 0$, the ESD of $A$ is asymptotically equivalent to that of a matrix $B \in \R^{n \times n}$, defined as
\begin{equation}\label{eq:B}
 \mB \bydef \sum_{k = \ell}^L \co_k  \mB_k.
\end{equation}
The first component in the sum in \eref{B} is a (shifted) Wishart matrix, constructed as follows:
\begin{equation}\label{eq:Bell}
\mB_\ell = \frac{1}{\sqrt{n N_\ell}} W\tran W - \sqrt{\frac {N_\ell} {n}} \mI,
\end{equation}
where $W \in \R^{N_\ell \times n}$ is an i.i.d. Gaussian matrix with $N_\ell \bydef d^\ell / \ell! + o(d^\ell)$. For each $k > \ell$, $\mB_k$ is an $n\times n$ GOE matrix. Furthermore, $B_\ell, B_{\ell+1}, \ldots, B_L$ are jointly \emph{independent}. 

For $z \in C_+$, let $s_{\mA}(z)$ and $s_{\mB}(z)$ denote the Stieltjes transforms of the ESDs of $\mA$ and $\mB$, respectively. The following theorem formalizes the asymptotic equivalence of the matrix models $A$ and $B$.

\begin{theorem}\label{thm:equivalence}
Fix $\ratio, \tau > 0$, two positive integers $\ell \le L$, and a set of coefficients $\set{\co_k}_{0 \le k \le L}$. Suppose that $n /d^\ell  = \ratio + \mathcal{O}(d^{-1/2})$, and $z = E + \mathrm{i}\eta$ with $\abs{E} \le \tau^{-1}$ and $\tau \le \eta \le \tau^{-1}$. Then, almost surely as $d \to \infty$, we have $s_{\mA}(z) \to m(z)$ and $s_{\mB}(z) \to m(z)$, where $m(z)$ is the unique solution in $C_+$ to the equation:
\begin{equation}\label{eq:m_equation}
m(z)\Big(z + \frac{\coa m(z)}{1 + \cob m(z)} + \coc m(z)\Big) + 1 = 0.
\end{equation}
Here, $\coa \bydef \co_\ell^2$, $\cob \bydef \co_\ell \sqrt{\ell!\ratio}$, and $\coc \bydef \sum_{\ell+1 \le k \le L} \co_k^2$. Moreover, for every $\varepsilon > 0$ and $D > 0$, we have
\begin{equation}\label{eq:sA_m}
\P\Big(\abs{s_A(z) - m(z)} \ge d^{-(1/2-\varepsilon)}\Big) \le d^{-D}
\end{equation}
for all sufficiently large $d$.
\end{theorem}
\begin{remark}
As $B$ is the linear combination of a (shifted) Wishart matrix and independent GOE matrices, its limiting eigenvalue density is given by an additive free convolution of a Marchenko-Pastur law with a semicircle law. Explicit formulas for the limiting density function
\begin{equation}\label{eq:rhofc}
\varrho_\mathrm{fc}(x) \bydef \lim_{\eta \downarrow 0} \frac{1}{\pi} \im[m(x + i \eta)]
\end{equation}
can be found in \cite[Appendix A]{cheng2013SpectrumRandom}. As a consequence of Theorem~\ref{thm:equivalence}, the spectral distribution of $A$ converges weakly almost surely to the same limit $\rho_\mathrm{fc}(E)$ as $d \to \infty$.
\end{remark}
\begin{remark}
It is important to note that the components of $A$ associated with indices $k < \ell$ do not play any role in the equivalent model given by $B$ in \eref{B}. This is due to the fact that the ``low-order'' components $\set{A_k}_{k < \ell}$ are essentially low-rank matrices with ranks on the order of $\mathcal{O}(d^{\ell-1})$. As a result, they do not contribute to the limiting spectrum density of the $n \times n$ matrix when $n \asymp d^\ell$. For a more detailed explanation, see Lemma~\ref{lemma:omit_low_order}.
\end{remark}

\subsection{A Heuristic Explanation of the Equivalence Principle}
\label{sec:heuristics}

The equivalence principle stated above has a simple heuristic explanation. To understand this, let us first revisit a crucial property of Gegenbauer polynomials. Given any $\vx_i, \vx_j \in \sph$,
\begin{equation}\label{eq:linearization_heuristics}
q_{k}(\sqrt{d} \, \vx_i\tran \vx_j) = \frac{1}{\sqrt{\Nk}}\sum_{a \in [\Nk]} \har{k}{a}(\vx_i) \har{k}{a}(\vx_j),
\end{equation}
where $\set{\har{k}{a}(\vx_i)}_{a \in [\Nk]}$ represents a collection of orthonormal degree-$k$ spherical harmonics associated with $\vx_i$, and 
\begin{equation}\label{eq:Nk_heuristics}
\Nk = \frac{d^k}{k!} (1 + \mathcal{O}(d^{-1}))
\end{equation}
denotes the cardinality of the set. For a comprehensive explanation and the exact expression for $\Nk$, we refer the reader to Appendix~\ref{appendix:ortho_poly}. The above identity is valuable because it allows us to ``linearize'' the term $q_{k}(\sqrt{d} \, \vx\tran \vy)$,  transforming it into an inner product of two $\Nk$-dimensional vectors comprised of spherical harmonics.

Using \eref{linearization_heuristics} and another identity, $q_{k}(\sqrt{d}) = \sqrt{\Nk}$, we can express the matrices $\set{A_k}$ in a factorized form:
\begin{equation}\label{eq:Ak_Sk}
\mA_k = \frac{1}{\sqrt{n \Nk}} \SH_k\tran \SH_k - \sqrt{\frac{\Nk}{n}}\, \mI,
\end{equation}
where $\SH_k$ is an $\Nk \times n$ matrix with entries consisting of the spherical harmonics, that is,
\begin{equation}\label{eq:Yk_def_heuristics}
(\SH_k)_{ai} = Y_{k, a}(\vx_i), \qquad \text{ for } a \in [\Nk], i \in [n]
\end{equation}
and $\set{\vx_i}_{i \in [n]}$ are the data vectors in the definition given by \eref{Ak}. By design, the $i$th column of $\SH_k$ depends solely on $\vx_i$. Consequently, due to the independence of $\set{\vx_i}_i$, the columns of $\SH_k$ are entirely independent. While the entries within each column are indeed dependent, the orthonormality of the spherical harmonics (refer to \eref{sh_ortho} in Appendix~\ref{appendix:ortho_poly}) confirms that these entries are uncorrelated random variables with unit variances. Thus, the columns of $\SH_k$ are independent and isotropic random vectors in $\R^{\Nk}$.

Heuristically, if we replace the entries of $\SH_k$ with i.i.d. standard Gaussians, we can expect the limiting ESD of $\mA_k$ to be characterized by the Marchenko-Pastur (MP) law with an aspect ratio parameter
\begin{equation}\label{eq:rp}
\rp_k \bydef \Nk/n.
\end{equation}
For $k = \ell$, this intuition underlies the equivalent model $B_\ell$ constructed in \eref{Bell}. (Refer to Appendix~\ref{appendix:MP_law} for a review of the relevant properties of the MP law.) Note that, since we set the diagonal entries of $A_k$ to zero, there is an additional shift (by $\sqrt{\rp_k}$) of the eigenvalues in \eref{Bell}. 

According to the heuristic arguments above, the limiting ESD of $A_k$ for every $k$ should be characterized by the MP law. However, in our equivalent model given by $B$, the higher order components, \emph{i.e.}, $B_k$ with $k > \ell$, are associated with the semicircle law. This discrepancy can be explained by examining the density function associated with the (shifted) MP law. By setting $\rp = \rp_k$ and $t = 1/\sqrt{\rp_k}$ in \eref{rho_MP}, we obtain
\begin{equation}\label{eq:rhog}
\varrho_k(x) \dif x \bydef \frac{\sqrt{\big[(2+x-1/\sqrt{\rp_k}\,)(2-x+1/\sqrt{\rp_k}\,)\big]_+}}{2\pi(1+x/\sqrt{\rp_k})}\, \dif x + (1-\rp_k)_+ \, \delta(x+ \sqrt{\rp_k}) \dif x.
\end{equation}
This density function is entirely determined by the aspect ratio parameter $\rp_k$ defined in \eref{rp}. Recall the expression for $\Nk$ in \eref{Nk_heuristics} and that $n = \ratio d^\ell + o(d^{\ell})$. For $k = \ell$, we have $\rp_k \to (\ell! \ratio)^{-1}$ as $d \to \infty$. In this case, \eref{rhog} corresponds to an MP law with an aspect ratio of size $\mathcal{O}(1)$. However, for $k > \ell$, we have $\rp_k \asymp d^{k-\ell}$, and thus $1/\sqrt{\rp_k}  = \mathcal{O}(1/\sqrt{d})$. In this situation, the MP law in \eref{rhog} degenerates to the standard semicircle law 
\begin{equation}\label{eq:rhosc}
\varrho_\mathrm{sc}(x) \dif x  := \frac{1}{ 2 \pi } \sqrt{(4 -x^2)_+} \; \dif x.
\end{equation}

Theorem~\ref{thm:equivalence} rigorously establishes the heuristic arguments mentioned above. In addition, it shows that, despite the apparent statistical dependence among $\set{A_k}_{k \ge \ell}$, they can be treated as a collection of asymptotically independent matrices. As a result, the self-consistent equation in \eref{m_equation} corresponds to the additive free convolution of an MP law with a semicircle law. To heuristically derive \eref{m_equation}, we can employ \eref{Ak_Sk} and \eref{Yk_def_heuristics} to express:
\begin{equation}\label{eq:Ak_Sk_full}
\mA_{ij} =  \frac 1 n\sum_{0 \le k \le L}  t_k \sum_{a \in [\Nk]}  \har{k}{a}(\vx_i) \har{k}{a}(\vy_j) - \Big(\frac 1 n \sum_{0 \le k \le L} t_k \Nk\Big) I,  \qquad t_k \bydef  \frac{\mu_k \sqrt n  }{\sqrt{\Nk}}.
\end{equation}
Notice that this expression has the form of the matrix presented in \eref{gMP} (in Appendix~\ref{appendix:MP_law}). By further assuming that the family ${\har{k}{a}(\vx)}$ consists of not just uncorrelated but indeed independent random variables, $\mA$ adheres to the MP law in \eref{m_gMP}, meaning its limiting Stieltjes transform, denoted by $m(z)$, satisfies the equation:
\begin{align}
\frac 1 m  %& = -z + \frac 1 n  \sum_k  \frac { N_k} { t_k^{-1} + m}  \nonumber\\
& = -z - \frac 1 n \Big [  \sum_{k <  \ell} +  \sum_{k = \ell}  +  \sum_{k >  \ell}  \Big ] \frac { N_k t_k^2 m} {1+ t_k m}  \nonumber \\
 &= -z -  \Big [  \sum_{k <  \ell} +  \sum_{k = \ell}  +  \sum_{k >  \ell}  \Big ] \frac { \co_k^2 m} {1+ \co_k \sqrt{n / \Nk} m}  \nonumber \\
& = -z -  \frac{\co_\ell^2 m}{1 + \co_\ell \sqrt{\ell!\ratio} \, m} - \Big(\sum_{k > \ell} \co_k^2\Big) m  + \mathcal{O}( d^{-1/2} ),\label{eq:m_equation_heuristics}
\end{align} 
where in the last step we have used the fact that $n/\Nk[\ell] \to \ratio \ell!$, $n / \Nk \to \infty$ for $k < \ell$ and $n / \Nk \to 0$ for $k > \ell$. Observe that, after dropping the $ \mathcal{O}( d^{-1/2} )$ error term, \eref{m_equation_heuristics} is exactly the self-consistent equation in \eref{m_equation}.

\subsection{General Nonlinear Kernels}
\label{sec:general_f}

In this subsection, we extend the equivalence principle established in Theorem~\ref{thm:equivalence} to encompass cases where the kernel $f_d$ in \eref{A_f} is a general function, going beyond merely polynomials. 

As our analysis will involve the expansion of $f_d$ using normalized Hermite polynomials, we first revisit the definition of these orthogonal polynomials: For $k \in \N_0$, the $k$th Hermite polynomial, denoted by $h_k(x)$, has degree $k$. Furthermore, for $i, j \in \N_0$,
\begin{equation}\label{eq:ortho_poly_Hermite}
\EE \, h_{i}(g) h_{j}(g) = \charfn_{ij},
\end{equation}
where $g \sim \mathcal{N}(0, 1)$ and $\charfn_{ij}$ is the Kronecker delta. The first four (normalized) Hermite polynomials are
\begin{equation}\label{eq:Hermite_low_degree}
h_{0}(x) = 1, \quad h_{1}(x) = x, \quad h_{2}(x) = \frac{1}{\sqrt{2}}(x^2 - 1), \quad h_3(x) = \frac 1{\sqrt 6}(x^3 - 3x).
\end{equation}
Upon comparing \eref{Hermite_low_degree} with \eref{Gegenbauer_low_degree} and \eref{ultra_deg3}, we observe that, as $d \to \infty$, the Hermite polynomials $h_k(x)$ for $k = 0, 1, 2, 3$ are indeed the asymptotic limits of the corresponding Gegenbauer polynomials $q_k(x)$. For a general $k \in \N_0$, it can be shown (see \cite[Lemma 4.1]{cheng2013SpectrumRandom}) that
\begin{equation}\label{eq:qk_hk}
q_k(x) = h_k(x) + \sum_{i = 0}^k c_{i, k}(d) x^i,
\end{equation}
where $\max_i \abs{c_{i, k}(d)} = \mathcal{O}_k(1/d)$.

In the subsequent discussion, we will establish an equivalence principle for the matrix in \eref{A_f}, given the following assumption on the function $f_d$.
\begin{assumption}\label{assumption:f}
Let $g \sim \mathcal{N}(0, 1)$, and let $\set{h_k(x)}_k$ represent the Hermite polynomials defined above. The function $f_d$ in \eref{A_f} satisfies the following conditions:
\begin{enumerate}
\item[(a)] For each $k \in \N_0$, we have
\begin{equation}\label{eq:cok_limit}
\EE [f_d(g) h_k(g)] \xrightarrow{d\to\infty} \co_k,
\end{equation}
for some finite numbers $\set{\co_k}$.

\item[(b)] The sequence $\set{\co_k}$ is square-summable, {i.e.}, 
\begin{equation}\label{eq:cok_summable}
\sigma^2 \bydef \sum_{k = 0}^\infty \co_k^2 < \infty.
\end{equation}
Moreover,
\begin{equation}\label{eq:cosigma}
\EE [f_d^2(g)] \xrightarrow{d\to\infty} \sigma^2.
\end{equation}

\item[(c)] Let $w(x)$ denote the probability density functions of $\mathcal{N}(0, 1)$, and let $w_d(x)$ denote the density function of the probability measure $\taumeasure$, as defined in \eref{tau_measure}. We have
\begin{equation}\label{eq:weight_w_wd}
\int_{\R} f_d^2(x) \abs{w(x) - w_d(x)} \dif x  \xrightarrow{d\to\infty} 0.
\end{equation}
\end{enumerate}
\end{assumption}

When $f_d(x) = f(x)$ is a function that is independent of $d$, the following lemma offers simple sufficient conditions that can be used to verify that Assumption~\ref{assumption:f} holds.
\begin{lemma}\label{lemma:f_const_condition}
A function $f(x)$ meets the conditions in Assumption~\ref{assumption:f} if (a) $\EE[f^2(g)] < \infty$ with $g \sim \mathcal{N}(0, 1)$, and (b) there exist positive constants $c_1, c_2, c_3$ such that $\abs{f(x)} < c_1 e^{c_2 \abs{x}}$ when $\abs{x} \ge c_3$.
\end{lemma}
\begin{proof}
The conditions in \eref{cok_limit}, \eref{cok_summable}, and \eref{cosigma} are immediate consequences of the assumption that $\EE[f^2(g)] < \infty$ and the completeness of the Hermite polynomials. The condition in \eref{weight_w_wd} can be verified by using Lemma~\ref{lemma:w_wd_wdb} found in Appendix~\ref{appendix:concentration}.
\end{proof}

%It is possible to further relax the conditions in Assumption~\ref{assumption:f} and consider a larger class of functions that can be  approximated by a finite number of Gegenbauer polynomials uniformly over $d$. See \cite[Sec. 3.2]{cheng2013SpectrumRandom} for details. On the other hand, 

\begin{theorem}\label{thm:equivalence_f}
Let $A$ be the matrix in \eref{A_f} with the function $f_d$ satisfying the conditions in Assumption~\ref{assumption:f}. Denote its Stieltjes transform by $s_A(z)$. Fix $\ratio, \tau > 0$, and assume that $n/d^\ell = \ratio + \mathcal{O}(d^{-1/2})$, and $z = E + i\eta$ with $\abs{E} \le \tau^{-1}$ and $\tau \le \eta \le \tau^{-1}$. Then, $s_{\mA}(z) \to m(z)$ almost surely as $d \to \infty$, where $m(z)$ is the unique solution in $C_+$ to the equation:
\begin{equation}\label{eq:m_equation_f}
m(z)\Big(z + \frac{\coa m(z)}{1 + \cob m(z)} + \coch m(z)\Big) + 1 = 0.
\end{equation}
Here, $\coa = \co_\ell^2$, $\cob = \co_\ell \sqrt{\ell!\ratio}$, $\coch = \sigma^2 - \sum_{0 \le k \le \ell} \co_k^2$, and $\set{\co_k}, \sigma^2$ are the constants defined in Assumption~\ref{assumption:f}.
\end{theorem}
\begin{remark}
Recall from \sref{heuristics} that the self-consistent equation \eref{m_equation_f} characterizes the free additive convolution of a (shifted) MP law and the semicircle law. Consequently, the statement of Theorem~\ref{thm:equivalence_f} can still be interpreted in terms of an equivalence principle: the limiting ESD of $A$ is equivalent to that of
\[
B = \co_\ell B_\ell + (\sigma^2 - \textstyle\sum_{0 \le k \le \ell} \co_k^2)^{1/2} H,
\]
where $B_\ell$ is the matrix defined in \eref{Bell} and $H$ is a GOE matrix independent of $B_\ell$.
\end{remark}
\begin{proof}
Let $\xi \sim \taumeasure$, with $\taumeasure$ being the probability measure defined in \eref{tau_measure}, and let $\set{q_k(x)}$ be the Gegenbauer polynomials. For any two functions $f(x), \bar f(x)$, we write $\inprod{f, \bar f} \bydef \EE[f(\xi) \bar f(\xi)]$ and $\norm{f} \bydef [\EE f^2(\xi)]^{1/2}$. As a straightforward consequence of the conditions \eref{cosigma} and \eref{weight_w_wd} presented in Assumption~\ref{assumption:f}, we have
\begin{equation}\label{eq:cosigma_q}
\norm{f_d}^2 \xrightarrow{d\to\infty} \sigma^2.
\end{equation}
Additionally, by using \cite[Lemma C.1]{cheng2013SpectrumRandom}, we can conclude from the conditions \eref{cok_limit} and \eref{weight_w_wd} that
\begin{equation}\label{eq:cok_limit_q}
\inprod{f_d, q_k} \xrightarrow{d\to\infty} \co_k
\end{equation}
for every $k \in \N_0$.

By the square-summable condition \eref{cok_summable} in Assumption~\ref{assumption:f}, for any fixed (small) $c > 0$, there exists some integer $L \ge \ell+1$ such that
\begin{equation}\label{eq:co_conv}
\sigma^2 - (\eta^4 c^2/24) \le \textstyle\sum_{0 \le k \le L - 1} \co_k^2 \le \sigma^2.
\end{equation}
Given the $L$ found above, we define two functions
\[
P_{f_{d}, L}(x) = f_d(x) - \textstyle\sum_{0 \le k \le L} \inprod{f_d, q_k} q_k(x)
\]
and
\[
\hat f_d(x) = \textstyle\sum_{0 \le k \le L-1} \co_k q_k(x) + \hat \co_{L} q_{L}(x),
\]
where $\hat \co_{L} \bydef (\sigma^2 - \sum_{0 \le k \le L-1} \co_k^2)^{1/2}$. By using the property that the Gengenbauer polynomials are orthonormal [see \eref{ortho_poly}], we can write
\begin{align}
\norms{f_d - \hat f_d}^2 &= \norms{\textstyle\sum_{0 \le k \le L-1} (\inprod{f_d, q_k} - \co_k) q_k(x) + (\inprod{f_d, q_L} - \hat \co_L) q_L(x) + P_{f_{d}, L}(x)}^2\\
&= \textstyle\sum_{0 \le k \le L-1} (\inprod{f_d, q_k} - \co_k)^2 + (\inprod{f_d, q_L} - \hat \co_L)^2 + (\,\norm{f_d}^2 - \sum_{0 \le k \le L} \inprod{f_d, q_k}^2)\\
&\le (\co_L - \hat \co_L)^2 + (\sigma^2 - \textstyle\sum_{0 \le k \le L} \co_k^2) + \eta^4 c^2/24, \qquad \text{for all sufficiently large } d.\label{eq:fhf_0}
\end{align}
To reach the inequality in \eref{fhf_0}, we have employed the characterizations presented in \eref{cosigma_q} and \eref{cok_limit_q}, which guarantee that the inequality holds for all $d \ge d(L, c)$, where $d(L, c)$ is some integer that may depend on $L$ and $c$. By \eref{co_conv}, $\hat \co_L^2 \le \eta^4 c^2/24$ and $\co_L^2 \le \eta^4 c^2/24$. We can then further bound the right-hand side of \eref{fhf_0} as
\begin{equation}\label{eq:f_approx}
\norms{f_d - \hat f_d}^2 \le 2 \co_L^2 + 2 \hat \co_L^2 + (\sigma^2 - \textstyle\sum_{0 \le k \le L} \co_k^2) + \eta^4 c^2/24 \le \eta^4 c^2/4,
\end{equation}
for all sufficiently large $d$.

Let $\hat A$ be a matrix constructed according to \eref{A_f} but with the function $f_d$ replaced by $\hat f_d$. Let $s_{\hat A}$ denote its Stieltjes transform. We can characterize $s_{\hat A}$ in two ways. On the one hand, since $\hat f_d$ is a linear combination of $L+1$ Gegenbauer polynomials, we can apply Theorem~\ref{thm:equivalence} to get
\begin{equation}\label{eq:sfm}
\abs{s_{\hat A}(z) - m(z)} \prec \frac{1}{\sqrt{d}},
\end{equation}
where $m(z) \in \C_+$ is the solution to \eref{m_equation_f}. On the other hand, by viewing $\hat A$ as a perturbation of $A$, we can employ the comparison inequality in Lemma~\ref{lemma:s1s2_general} in Appendix~\ref{appendix:resolvent_identities} to get
\begin{equation}\label{eq:s1s2_sphere}
\abs{s_A(z) - s_{\hat A}(z)} \le \frac{1}{\eta^2} \norms{f_d - \hat f_d} + \mathcal{O}_\prec\Big(\frac{1}{\eta\sqrt{n}}\Big) \le  \frac{c}{2} +  + \mathcal{O}_\prec\Big(\frac{1}{\eta\sqrt{n}}\Big),
\end{equation}
where the second inequality follows from \eref{f_approx}. By the triangular inequality and estimates in \eref{sfm} and \eref{s1s2_sphere}, we have
\begin{align}
\abs{s_A(z) - m(z)} &\le \abs{s_A(z) - s_{\hat A}(z)} + \abs{s_{\hat A}(z) - m(z)}\\
&\le c/2 +  \mathcal{O}_\prec\Big(\frac{1}{\sqrt{d}}\Big).
\end{align}
Thus, for any $D > 0$,
\[
\P(\,\abs{s_A(z) - m(z)} > c) < d^{-D}
\]
for all sufficiently large $d$. Applying the Borel-Cantelli lemma (for a fixed $D > 1$), we can then conclude that $s_A(z)$ converges to $m(z)$ almost surely.
\end{proof}

\subsection{Numerical Experiments}
\label{sec:numerical}

In this subsection, we present numerical experiments that demonstrate the equivalence principle as stated in Theorems~\ref{thm:equivalence} and \ref{thm:equivalence_f}. We first examine the empirical spectral distribution (ESD) of individual polynomial component matrices $A_k$, as defined in \eref{Ak}. In the quadratic scaling regime, where $n$ is asymptotically proportional to $\ratio d^2$ for a fixed $\ratio$, $A_2$ is asymptotically equivalent to $B_2$ in \eref{Bell}. Consequently, the ESD of $A_2$ converges to the MP law, characterized by a density function given in \eref{rho_MP}. Both $A_3$ and $A_4$ are asymptotically equivalent to a GOE matrix, and their ESDs therefore converge to the standard semicircle law, as outlined in \eref{rhosc}. As illustrated in \fref{A2A3A4}, the ESDs of $A_2$, $A_3$, and $A_4$ closely align with their respective limiting spectral densities.

\begin{figure}[t!]
	\centering
	\begin{subfigure}[b]{0.3\textwidth}
		\centering
		\includegraphics[width=\textwidth]{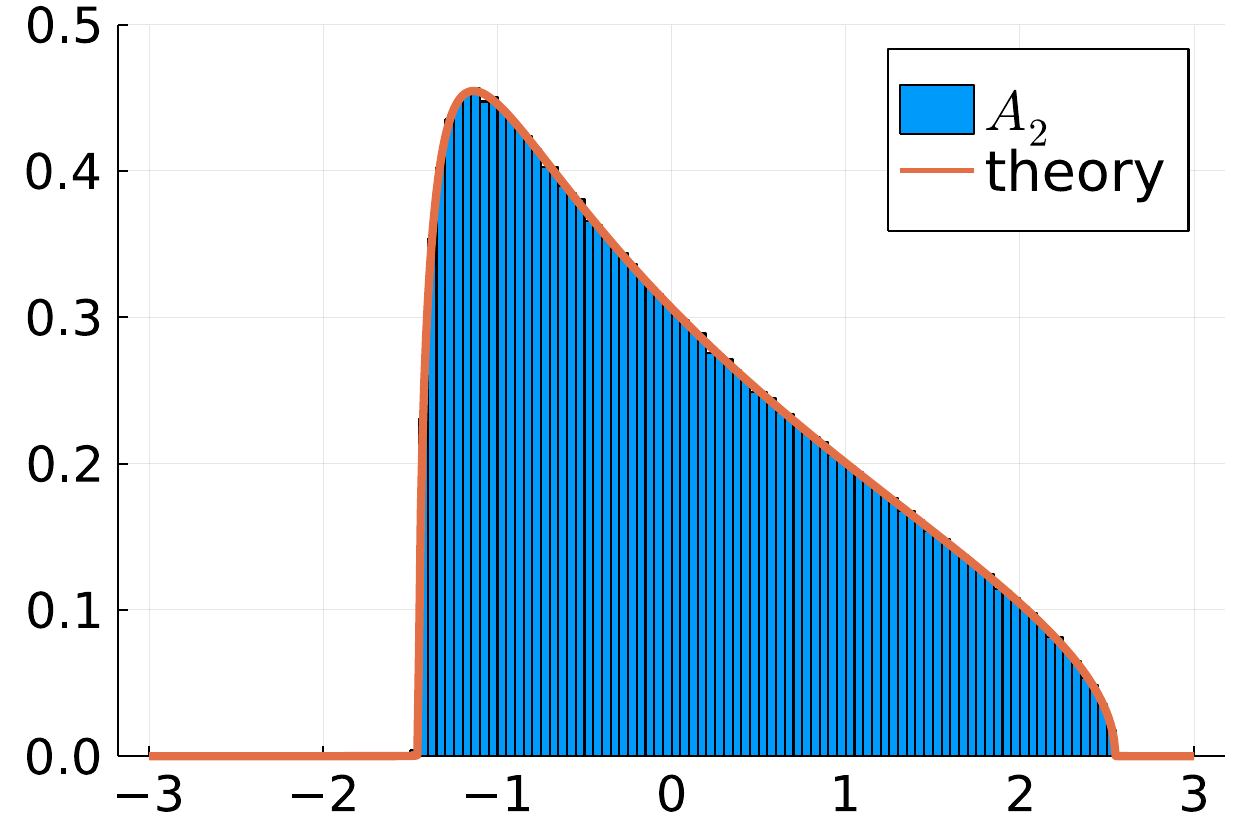}
	\end{subfigure}
	\hspace{5pt}
		\begin{subfigure}[b]{0.3\textwidth}
		\centering
		\includegraphics[width=\textwidth]{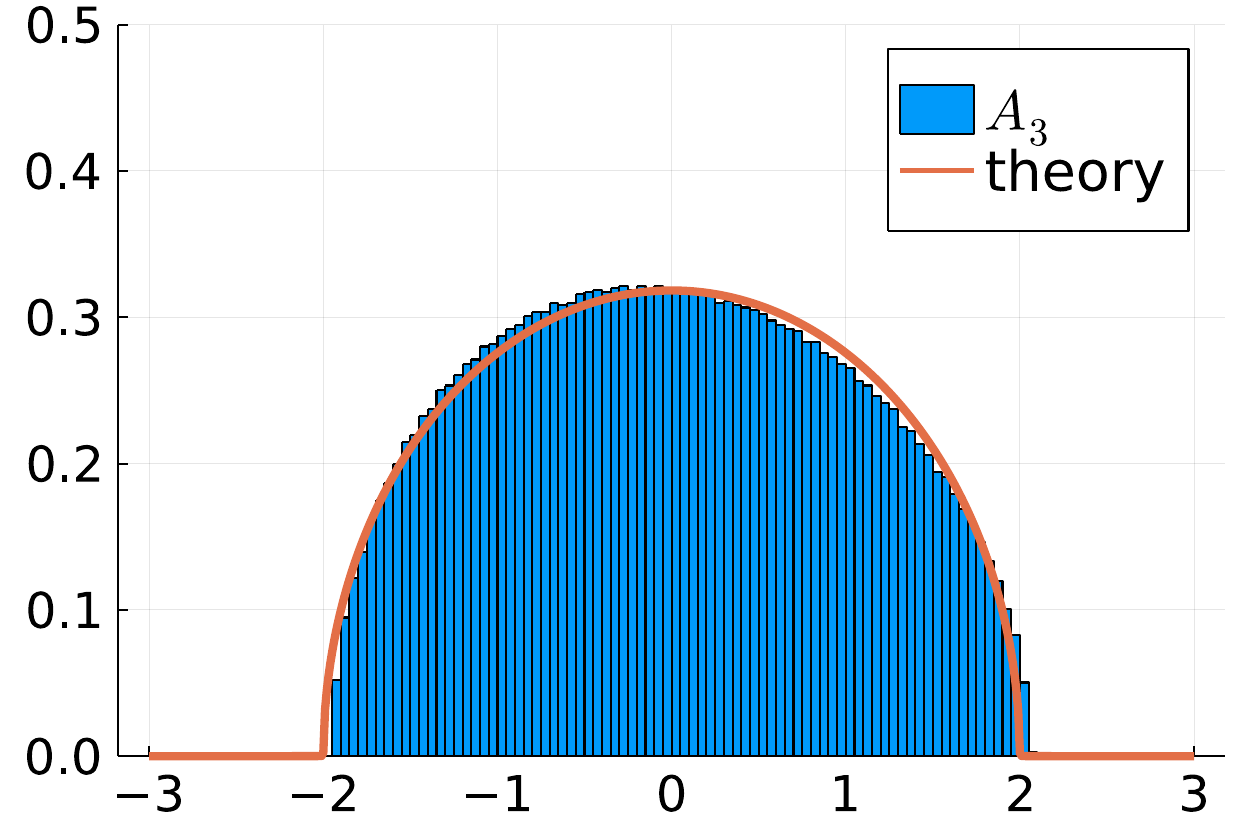}
	\end{subfigure}
	\hspace{5pt}
	\begin{subfigure}[b]{0.3\textwidth}
		\centering
		\includegraphics[width=\textwidth]{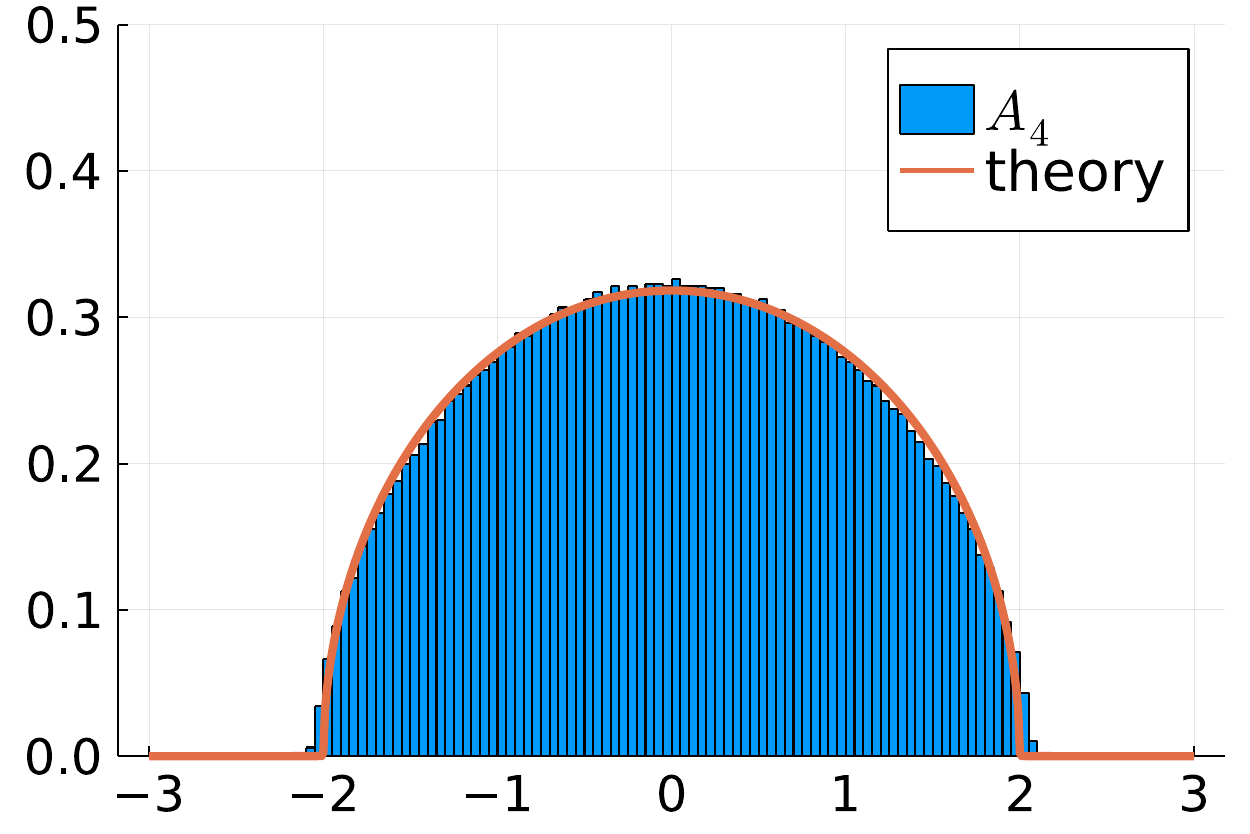}
	\end{subfigure}
	\caption{Demonstration of the equivalence principle in the quadratic scaling regime, where $n \asymp \ratio d^2$. In this experiment, $d = 300$ and $\ratio = 0.15$. The ESDs of $A_2$, $A_3$, and $A_4$ are plotted against their corresponding limiting spectral densities given by Theorem~\ref{thm:equivalence}. Observe that the ESD of $A_2$ is characterized by the MP law in \eref{rho_MP} (with shape parameters $\rp = 1/(2\ratio)$ and $t = 1/\sqrt{\rp}$), whereas the ESDs of $A_3$ and $A_4$ both follow the standard semicircle law given in \eref{rhosc}.}\label{fig:A2A3A4}
\end{figure}

\begin{figure}[t!]
	\centering
	\begin{subfigure}[b]{0.3\textwidth}
		\centering
		\includegraphics[width=\textwidth]{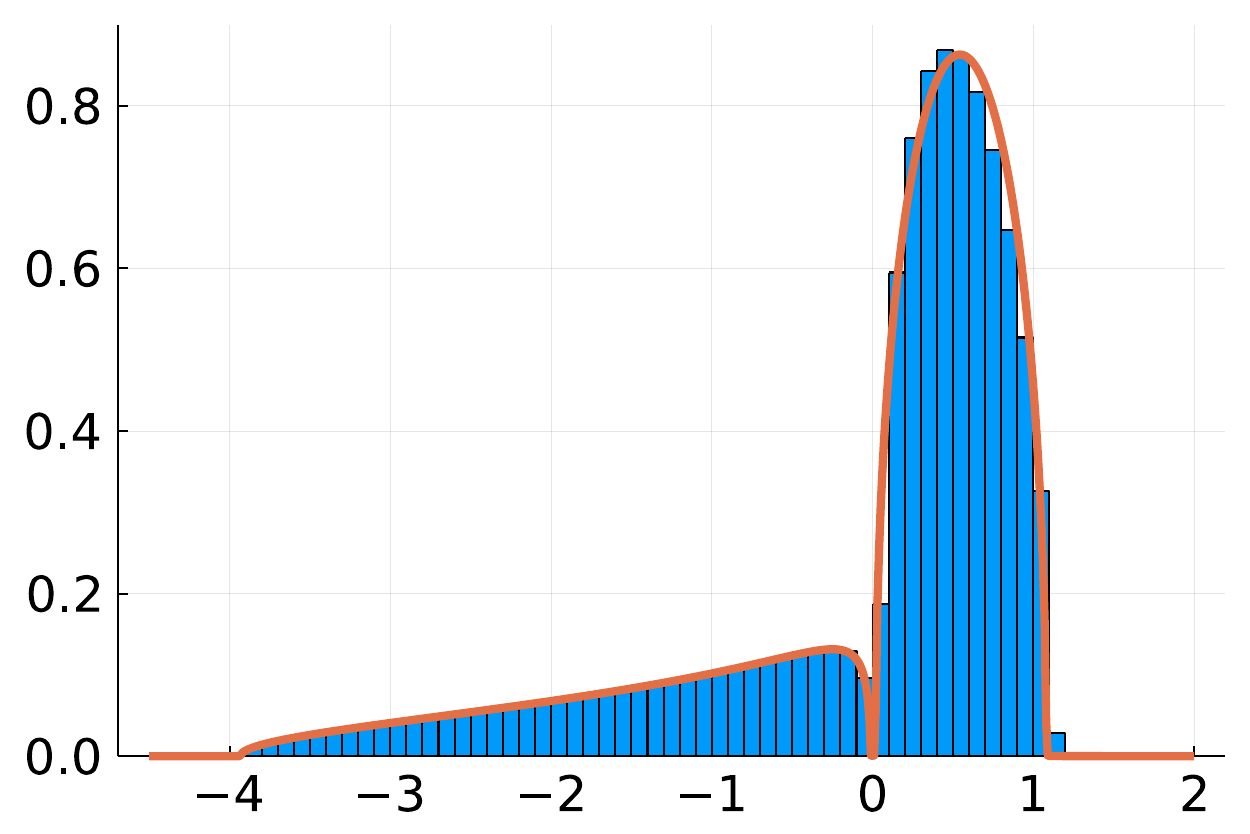}
		\caption{The matrix in \eref{A2A3A4}}\label{fig:quadratic_combination:1}
	\end{subfigure}
	\hspace{20pt}
	\begin{subfigure}[b]{0.3\textwidth}
		\centering
		\includegraphics[width=\textwidth]{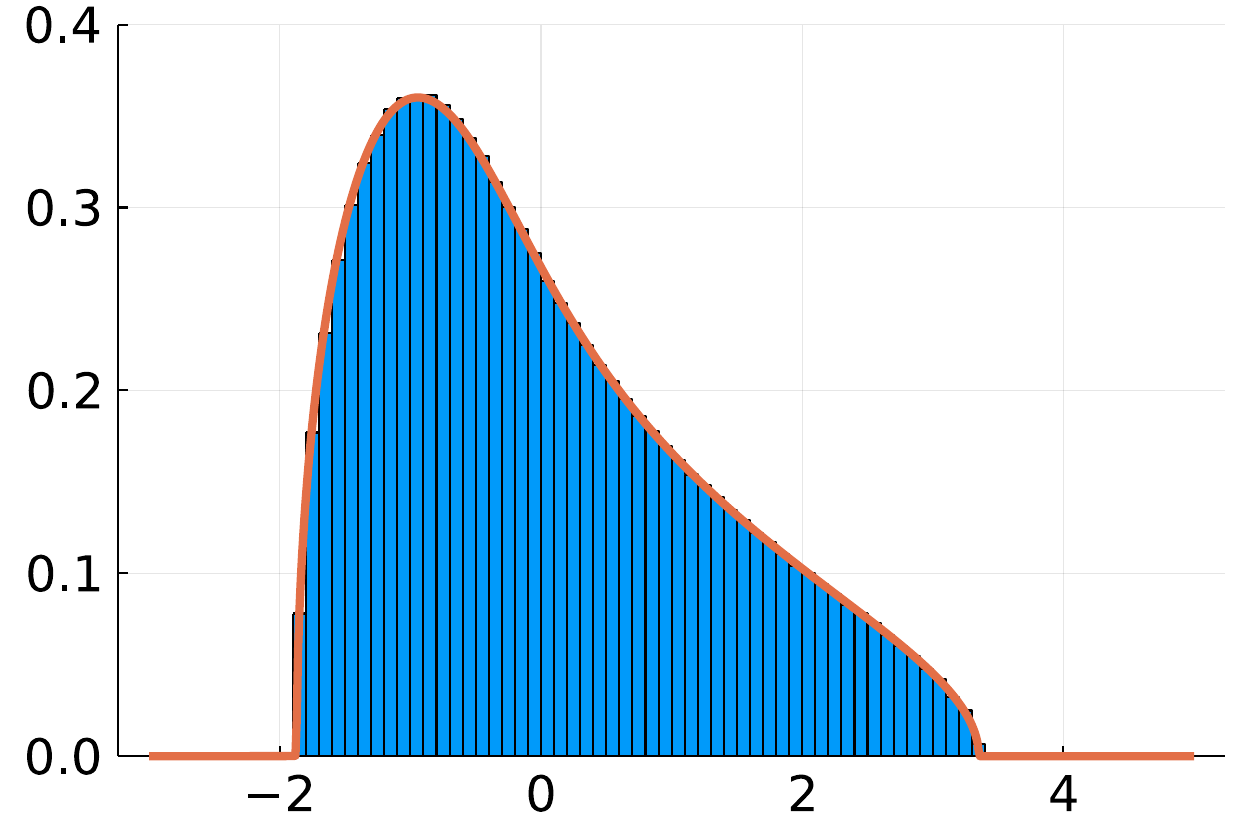}
		\caption{The matrix in \eref{A5A6_trunc}}\label{fig:quadratic_combination:2}
	\end{subfigure}
	\caption{Linear combinations of the Gegenbauer polynomial matrices. (a) The ESD of the matrix defined in \eref{A2A3A4}. In this experiment, $d = 100$ and $n = 1.8\times10^4$, and thus $n /d^2 = 1.8$. (b) The ESD of the matrix defined in \eref{A5A6_trunc}. Here, we set $d = 200$, $n =2\times 10^4$. The truncation threshold in \eref{A5_trunc} is set to $M = 3$.} \label{fig:quadratic_combination}
\end{figure}

In the second example, we consider linear combinations of the polynomial matrices $\set{A_k}$. Theorem~\ref{thm:equivalence} states that the ESD of any fixed linear combination of $A_k$ can be obtained by a free additive convolution between the MP law and the semicircle law. In \fref{quadratic_combination:1}, we plot the ESD of
\begin{equation}\label{eq:A2A3A4}
-A_2 + \frac {1} {\sqrt{20}} A_3 + \frac {1} {\sqrt{20}} A_4
\end{equation}
against the limiting spectral density. The theoretical curve (red solid line in the figure) is computed by solving the self-consistent equation \eref{m_equation} and then evaluating \eref{rhofc} numerically. In \fref{quadratic_combination:2}, we show the ESD of another matrix in the form of
\begin{equation}\label{eq:A5A6_trunc}
A_2 + \frac {1} {4} A_3 + \frac {1} 4 A_4 + \frac{1}4 A_5^M + \frac{1}4 A_6^M.
\end{equation}
Here, $A_5^M$ and $A_6^M$ are the truncated version of $A_5$ and $A_6$, respectively. Specifically, for $M > 0$, we have
\begin{equation}\label{eq:A5_trunc}
(A_5^M)_{ij} = (A_5)_{ij} \cdot \charfn(\, \abs{(A_5)_{ij}} \le M/\sqrt{n}) \qquad i,j \in [n],
\end{equation}
and $A_6^M$ is defined similarly. Note that we apply this truncation to remove a small number of outliers in the entries of $A_5$ and $A_6$ that have very large magnitudes. The presence of these outliers intensifies the finite-size effect, causing the ESD to deviate from the limiting spectral density when the dimension $d$ is not very large. Due to the truncation step, the matrix in \eref{A5A6_trunc} is no longer a finite linear combination of $\set{A_k}$. Consequently, we employ Theorem~\ref{thm:equivalence_f} to compute the theoretical curve.

In the last example, we consider the ESD of the matrix in \eref{A_f}, where the nonlinear function is a soft-thresholding operator, \emph{i.e.},
\begin{equation}\label{eq:thresholding}
f_d(x) = \sign(x) (\,\abs{x} - \tau)_+
\end{equation}
for some threshold $\tau > 0$. \fref{thresholding} compares the ESDs of this matrix against the limiting spectral densities in the linear ($n \asymp \ratio d$), quadratic ($n \asymp \ratio d^2$), and cubic ($n \asymp \ratio d^3$) scaling regimes, respectively. Since $\EE_{\xi \sim \taumeasure} [f_d(\xi) q_1(\xi)] > 0$, the matrix $\mA$ has a nonzero ``projection'' in $A_1$. As we show in Lemma~\ref{lemma:op_spike} in Appendix~\ref{appendix:op_norm}, $A_1$ is essentially a rank-$d$ matrix with $d$ large eigenvalues of size $\mathcal{O}(\sqrt{n/d})$. It follows that, in the quadratic and cubic scaling regimes, the spectrum of $\mA$ contains $d$ outlier eigenvalues that are separated from the bulk.

\begin{figure}[t!]
	\centering
	\begin{subfigure}[b]{0.3\textwidth}
		\centering
		\includegraphics[width=\textwidth]{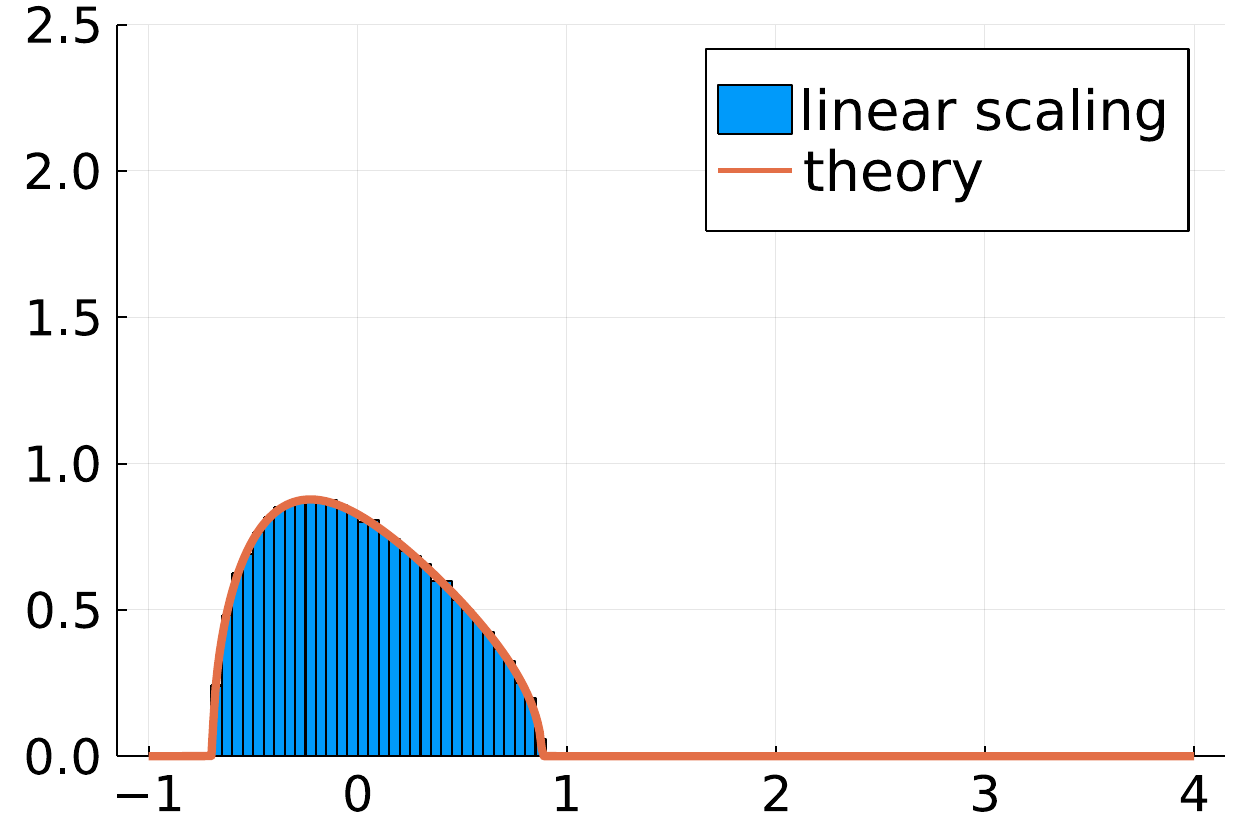}
	\end{subfigure}
	\hspace{5pt}
		\begin{subfigure}[b]{0.3\textwidth}
		\centering
		\includegraphics[width=\textwidth]{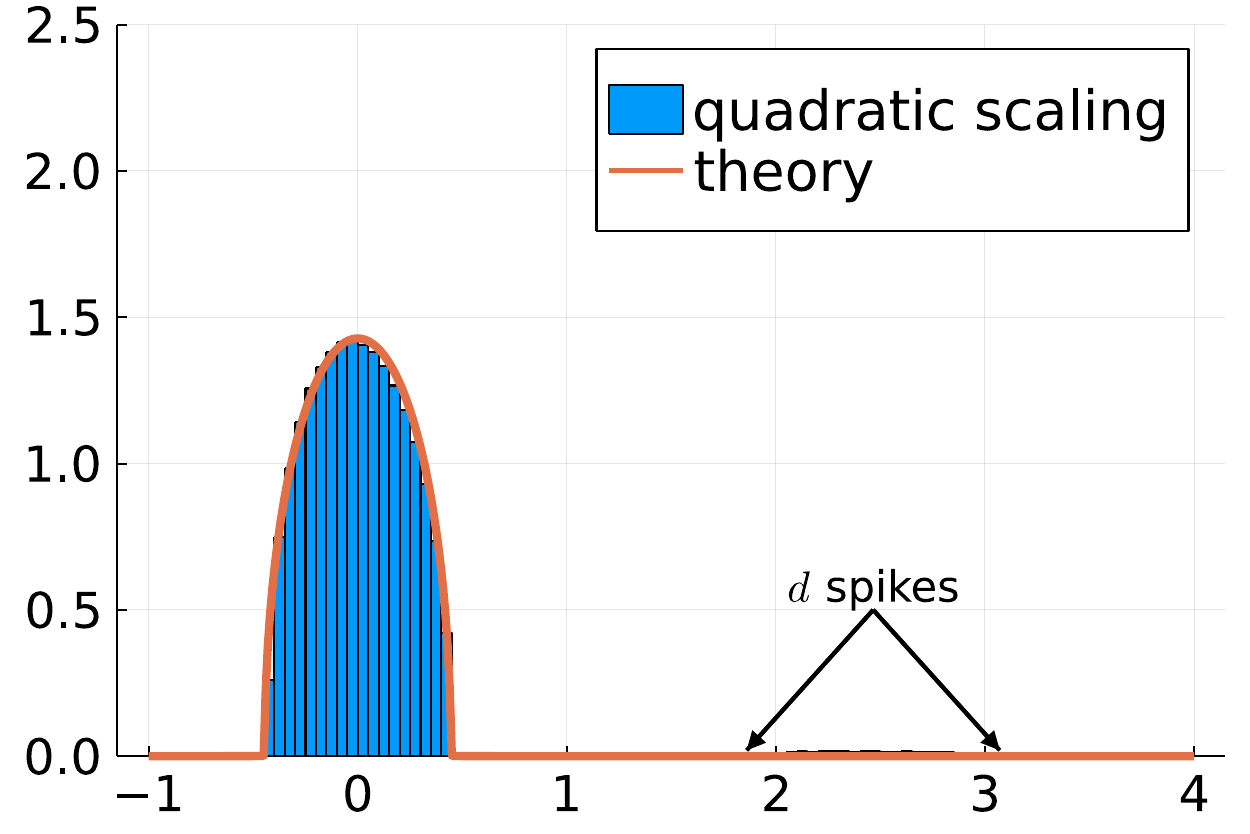}
	\end{subfigure}
	\hspace{5pt}
	\begin{subfigure}[b]{0.3\textwidth}
		\centering
		\includegraphics[width=\textwidth]{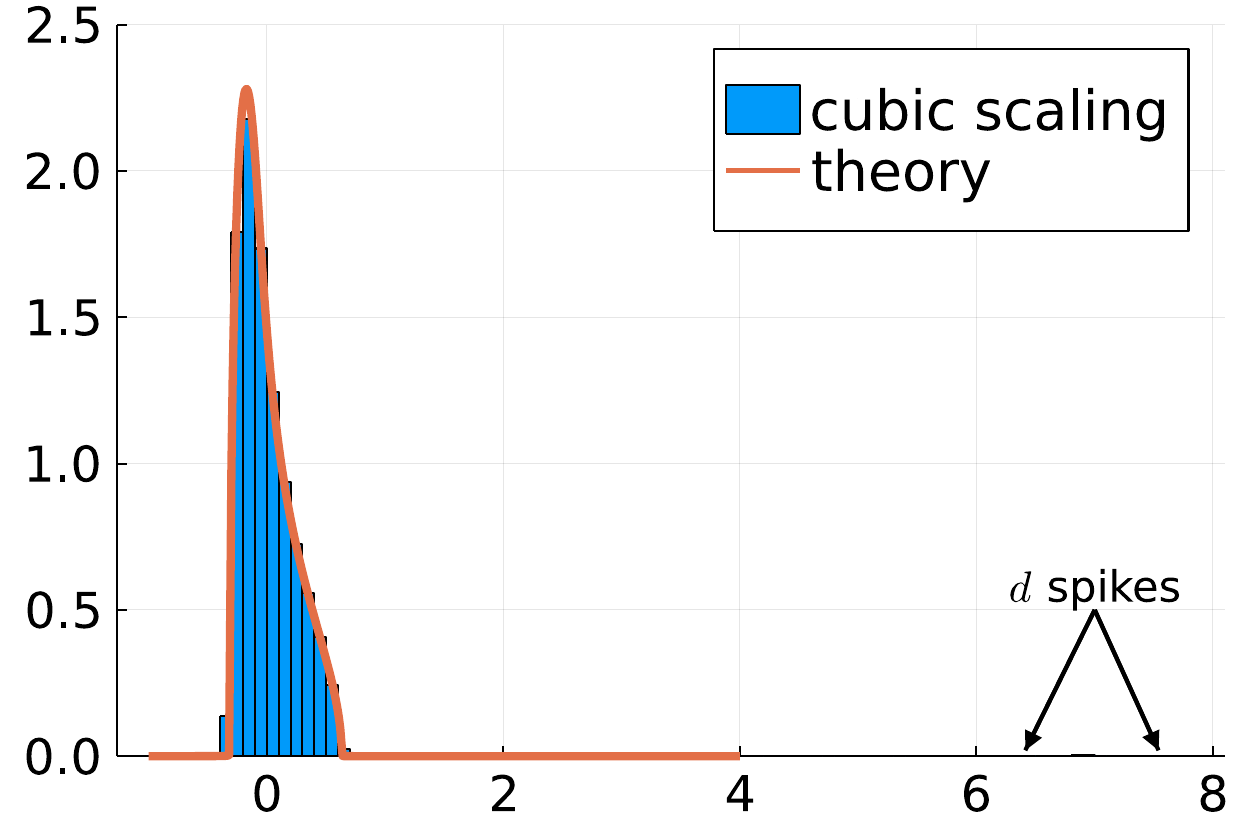}
	\end{subfigure}
	\caption{The ESD of the inner-product kernel matrix \eref{A_f} where $f_d$ is the soft-thresholding function in \eref{thresholding} with $\tau = 1$. In the experiment, we choose $\tau = 1$, $\ratio = 0.15$, and set the dimension to $d = 1.2\times 10^4$, $d = 300$ and $d = 50$ in the linear, quadratic, and cubic scaling regimes, respectively.}\label{fig:thresholding}
\end{figure}

% !TEX root = equivalence.tex

\section{Related Work}
\label{sec:related}

In this section, we discuss several related lines of work in the literature.

\emph{Random kernel matrices}: Among the earlier studies on the spectrum of random kernel matrices, \citet{koltchinskii2000RandomMatrix} considered a setting where the data vectors $\set{\vx_i}$ are sampled from a fixed manifold in $\R^d$. They demonstrated that, when $n \to \infty$ with $d$ fixed, the spectrum of the kernel matrix converges to that of an integral operator on the manifold. The high-dimensional setting, where $n, d \to \infty$ with $n/d \to \ratio \in (0, \infty)$, was first investigated by \citet{elkaroui2010spectrumkernel}. This work examined a scaling in which the matrix elements $A_{ij}$ are $f(\vx_i\tran \vx_j)$. When the data vectors $\set{\vx_i}$ are independently sampled from $\unifsp$, the typical size of $\vx_i\tran\vx_j$ is $\mathcal{O}(1/\sqrt{d})$ for $i \neq j$. As a result, the kernel matrix only depends on the properties of the kernel function $f(x)$ in a local neighborhood near $x = 0$. Indeed, \citet{elkaroui2010spectrumkernel} showed that, under mild local smoothness assumptions on $f(x)$, the kernel matrix is asymptotically equivalent to a simpler matrix obtained via a first-order Taylor expansion. Our work, as well as that of \cite{cheng2013SpectrumRandom}, considers a different scaling: the presence of the $\sqrt{d}$ factor in \eref{A_f} means that the kernel function $f_d$ is applied to values of size $\mathcal{O}(1)$ rather than $\mathcal{O}(1/\sqrt{d})$. Consequently, the spectrum of the matrix in \eref{A_f} depends on the global properties of the kernel function.

The linear asymptotic regime of the model in \eref{A_f} was first investigated by \citet{cheng2013SpectrumRandom}, who established the weak limit of the ESD of the kernel matrix. The idea of ``decorrelation'' by expanding the nonlinear kernel function using an orthogonal polynomial basis was first introduced in that work and plays a significant role in our paper. The results of \cite{cheng2013SpectrumRandom} were obtained for data vectors sampled from isotropic Gaussian and spherical distributions. \citet{do2013spectrumrandom} extended these results to other distributions (such as the Bernoulli case). The spectral norm of the kernel matrix was explored by \citet{fan2017SpectralNorm}, who also made the observation that the limiting distribution obtained in \cite{cheng2013SpectrumRandom} is the free additive convolution of an MP law and a semicircle law.

\emph{Asymptotics with polynomial scaling}: Studies of high-dimensional statistical problems often focus on the linear asymptotic regime. In random matrix theory, the more general polynomial asymptotic regime was investigated by \citet{bloemendal2015IsotropicLocal}, who demonstrated the local MP law for sample covariance matrices $X\tran X$, where $X$ is an $N \times n$ matrix with $\log N \asymp \log n$. In that work, the matrix $X$ is assumed to have independent entries. Although the kernel matrix in our work can also be written in the form of a (generalized) sample covariance matrix [see \eref{Ak_Sk} and \eref{Ak_Sk_full}], the matrix entries in our problem consist of spherical harmonics, which are uncorrelated but dependent random variables. As a result, the technical approach of \cite{bloemendal2015IsotropicLocal} cannot be directly applied here. In the context of kernel methods, exact asymptotics in the polynomial scaling regime were examined in \citet{opper2001Supportvector} and \citet{bordelon2021SpectrumDependent} using nonrigorous statistical physics methods. The heuristics employed by these authors were similar to the scheme outlined in \sref{heuristics}, specifically, treating the uncorrelated spherical harmonics as if they were independent standard normal random variables.

\citet{ghorbani2020Linearizedtwolayers} explored the high-dimensional limit of kernel regression and several related models under a scaling where $d^{\ell+\delta} \le n \le d^{\ell+1-\delta}$ for some (small) $\delta > 0$. In this context, the number of samples is assumed to fall between two consecutive integer powers of $d$. This contrasts with our work, in which we assume $n/d^\ell \to \kappa \in (0, \infty)$. Upon completing this paper, we became aware of an independent work by \citet{misiakiewicz2022Spectruminnerproduct} that examines the limit of kernel ridge regression in the exact polynomial asymptotic regime. One of the main results of that study shows that the limiting spectrum of the kernel matrix follows the MP law when the kernel function $f_d$ is the $\ell$th Gegenbauer polynomial. This finding corresponds to a specific case in our equivalence principle, where the matrix $A$ in \eref{A} comprises solely the component $\co_\ell A_\ell$. In our model, the inclusion of high-order components $\set{A_k}_{k > \ell}$ leads to more general limiting distributions (the free convolution between the MP law and the semicircle law) and additional technical challenges in the proof.

\emph{Non-Hermitian ensembles and random feature models}:  Lastly, we note that it is possible to extend the current study to a non-Hermitian version of  \eref{A_f}.  In this case, one would investigate an $n \times p$ matrix, with entries given by
\[
A_{ij} = \frac{1}{\sqrt n} f_d(\sqrt{d} \vx_i \vy_j),
\]
where $\{\vx_i\}_{i \le n}$ and $\{\vy_j\}_{j \le p}$ represent two collections of vectors in $\R^d$. This type of matrix appears in the random feature model \cite{rahimi2007RandomFeatures,louart2017RandomMatrix,hastie2020SurprisesHighDimensional,penningtonNonlinearRandomMatrix2019}, an interesting theoretical model for large random neural networks. For such non-Hermitian matrices, an asymptotic Gaussian equivalence phenomenon, analogous to the equivalence principle obtained in this study, have been demonstrated under the linear scaling regime, when $n, d$ and $p$ grow to infinity at fixed ratios \cite{mei2019generalizationerror,gerace2021GeneralisationErrorLearning,goldtGaussianEquivalenceGenerative2020,loureiro2022LearningCurvesGeneric,hu2021UniversalityLaws}. We conjecture that similar equivalence principles may exist under the polynomial scaling regime.

% !TEX root = equivalence.tex

\section{Proof of the Main Result}
\label{sec:proof_equivalence}

This section is devoted to the proof of Theorem~\ref{thm:equivalence}. We start by presenting a high-level outline of the proof in \sref{outline}. The technical details are given in Sections~\ref{sec:A_complement} and \ref{sec:A_Schur}, and in the appendix.

\subsection{Outline of the Proof}
\label{sec:outline}

Recall that $n = \ratio d^\ell + o(d^\ell)$ for some fixed $\kappa > 0$ and $\ell \in \N$. Accordingly, we split the matrix $\mA$ in \eref{A} into two parts: the low-order term $\Al \bydef \sum_{0 \le k < \ell} \co_k \mA_k$, and the high-order term $\Ah \bydef \sum_{\ell \le k \le L} \co_k \mA_k$. We start by observing that the contribution of the low-order term to the Stieltjes transform is negligible.

\begin{lemma}\label{lemma:omit_low_order} Let $s(z)$ and $\hat{s}(z)$ denote the Stieltjes transforms of $A$ and $\Ah$, respectively. Under the same settings of Theorem~\ref{thm:equivalence}, we have
\begin{equation}\label{eq:s_diff_LH}
\abs{s(z) - \hat{s}(z)} = \mathcal{O}\Big(\frac{1}{\eta^2 \sqrt{d}}\Big).
\end{equation}
\end{lemma}
\begin{proof}
See Appendix~\ref{appendix:resolvent_identities}.
\end{proof}
%\yml{The larger component $\frac{1}{\eta^2 \sqrt{d}}$ is only due to a constant $\mathcal{O}(d^{-1/2})$ shift  in the eigenvalues. We will eventually simplify the above bound by combining the two terms, if we assume $\eta < 1/\tau$ for some fixed $\tau$.}

In light of Lemma~\ref{lemma:omit_low_order}, we will assume without loss of generality that $\Al = 0$ and $A = \Ah$ in our following discussions. For each $i \in [n]$, we use $A\mi \in \R^{(n-1)\times(n-1)}$ to denote the minor matrix obtained by removing the $i$th column and row of $A$, \emph{i.e.},
\[
A\mi_{ab} = A_{ab}, \qquad a, b \neq i.
\]
Note that, rather than mapping the indices of $A\mi$ to $[n-1]$, we use $a, b \in [n] \setminus \set{i}$ to index the entries of $A\mi$. Let $G(z) = (A - z I)^{-1}$ and $G\mi(z) = (A\mi - z I)^{-1}$ be the resolvents of $A$ and $A\mi$, respectively. 

%To streamline the notation, we will often write $G$ and $G\mi$ for the resolvents by suppressing their dependence on the spectral parameter $z$.

We start our proof by applying Schur's complement formula, which gives us
\begin{equation}\label{eq:Schur}
\frac{1}{G_{ii}(z)} = -z - \sum_{a, b \neq i} A_{ia} G\mi_{ab}(z) A_{ib}.
\end{equation}
Recall that the Stieltjes transform of $A$ can be obtained as $s(z) = \frac{1}{n} \sum_{i} G_{ii}(z)$. The main technical step in our proof is to show that
\begin{equation}\label{eq:H_complement}
Q_i(z) \bydef \sum_{a, b \neq i} A_{ia} G\mi_{ab}(z) A_{ib} - \Big[\frac{\coa s(z)}{1 + \cob s(z)} + \coc s(z)\Big] = \mathcal{O}_\prec(d^{-1/2}),
\end{equation}
where $\coa, \cob, \coc$ are the three constants defined in Theorem~\ref{thm:equivalence}. We will prove this estimate in \sref{A_Schur} (see Proposition~\ref{prop:A_complement}).

Using \eref{H_complement}, we can now split the right-hand side of \eref{Schur} into a leading term and an error term. Multiplying both sides of \eref{Schur} by $G_{ii}(z)$ and averaging over $i$, we get
\begin{equation}\label{eq:s_equation}
s(z)\Big(z + \frac{\coa s(z)}{1 + \cob s(z)} + \coc s(z)\Big) + 1 = \underbrace{- \frac{1}{n} \sum_i Q_i(z) G_{ii}(z)}_{\let\scriptstyle\textstyle(\text{error term})}.
\end{equation}
By \eref{G_inf} in the first step, and by applying \eref{H_complement} and the union bound in the second step, 
\begin{equation}\label{eq:H_error}
\abs{(\text{error term})} \le \frac{1}{\eta}\max_i \,\abs{Q_i(z)} = \frac{1}{\eta} \mathcal{O}_\prec(d^{-1/2}).
\end{equation}
Observe that the left-hand side of \eref{s_equation} has exactly the same functional form as the left-hand side of \eref{m_equation}. The error estimate in \eref{H_error} then implies that the Stieltjes transform $s(z)$ approximately satisfies the nonlinear equation in \eref{m_equation}. By analyzing the uniqueness and stability of the solution $m(z)$ to this nonlinear equation (see Proposition~\ref{prop:stability} in Appendix~\ref{appendix:stability}), we can then conclude that $s(z) \approx m(z)$.

\subsection{Schur Complement: Reparameterization}
\label{sec:A_complement}

We devote this and the next subsections to establishing the estimate in \eref{H_complement}. Observe that, by construction, different coordinates of $A$ are statistically exchangeable. Thus, we only need to show \eref{H_complement} for a single index $i \in [n]$. We choose to do so for $i = n$.

Recall from \eref{A} and \eref{Ak} that
\begin{equation}\label{eq:An_original}
A_{na} = \frac{1}{\sqrt{n}} \sum_{\ell \le k \le L} \co_k q_{k}(\sqrt{d}\, \vx_n\tran \vx_a),
\end{equation}
and the entries of the minor $\mA\mn$ are of the form
\begin{equation}\label{eq:A_minus_n_original}
A\mn_{ab} = \frac{1}{\sqrt{n}} \sum_{\ell \le k \le L} \co_k q_{k}(\sqrt{d}\, \vx_a\tran \vx_b) \charfn_{a \neq b}, \qquad \text{for } a, b \in [n-1],
\end{equation}
where $\charfn_{a \neq b} \bydef 1 - \charfn_{ab}$. As a reminder to the reader: due to Lemma~\ref{lemma:omit_low_order}, we assume that $A$ contains no lower-order components, \emph{i.e.}, $\co_k = 0$ for $0 \le k < \ell$.

The random variables $\{A_{na}\}$ and $\{A\mn_{ab}\}$ are weakly correlated. To make this correlation explicit, we use the following reparameterization. For every $a \in [n-1]$, let 
\begin{equation}
\xi_a \bydef \sqrt{d}\, \vx_n\tran \vx_a \quad\text{and} \quad \widetilde{\vx}_a \bydef (1-\xi_a^2 / d)^{-1/2} R_n\tran \vx_a,
\end{equation}
where $R_n \in \R^{d \times (d-1)}$  is a matrix whose columns are unit-norm vectors orthogonal to $\vx_n$. We can now rewrite \eref{An_original} as
\begin{equation}\label{eq:An}
A_{na} = \frac{1}{\sqrt{n}} \sum_{\ell \le k \le L} \co_k q_{k}(\xi_a).
\end{equation}
By definition, $(R_n R_n\tran + \vx_n \vx_n\tran) =I$, and hence $\vx_a\tran \vx_b = \vx_a\tran (R_n R_n\tran + \vx_n \vx_n\tran) \vx_b$ for $a, b \in [n-1]$. We can then verify that
$\sqrt{d}\, \vx_a\tran \vx_b= r(\xi_a) r(\xi_b) \sqrt{d-1} \widetilde{\vx}_a\tran \widetilde{\vx}_b  + \xi_a \xi_b / \sqrt{d}$, where 
\begin{equation}\label{eq:rxi}
r(\xi) \bydef (1-1/d)^{-1/4} (1- \xi^2/d)^{1/2}.
\end{equation}
This allows us to rewrite \eref{A_minus_n_original} as
\begin{equation}\label{eq:A_minus_n}
A\mn_{ab} = \frac{1}{\sqrt{n}} \sum_{\ell \le k \le L} \co_k q_{k}\Big(r(\xi_a) r(\xi_b) \sqrt{d-1} \widetilde{\vx}_a\tran \widetilde{\vx}_b  + \xi_a \xi_b / \sqrt{d}\Big) \charfn_{a \neq b}.
\end{equation}

The advantage of the new parameterizations in \eref{An} and \eref{A_minus_n} is that the families $\set{\xi_a}$ and $\set{\widetilde{\vx}_a}$ are independent, as we show below.

\begin{lemma}\label{lemma:rotation}
$\set{\xi_a}_{a < n}$ is an i.i.d. family with $\xi_a \sim \taumeasure$, where $\taumeasure$ is the probability measure defined in \eref{tau_measure}. $\set{\widetilde{\vx}_a}_{a < n}$ is an i.i.d. family of random vectors with $\widetilde{\vx}_a \sim \unif{\mathcal{S}^{d-2}}$. Moreover, $\set{\xi_a}$ and $\set{\widetilde{\vx}_a}$ are mutually independent.
\end{lemma}
\begin{proof}
Fix $\vx_n$, and let $E = [\vx_n \, R_n]$. By construction, $E$ is an orthogonal matrix. Define 
\[
\vv_a = E\tran \vx_a, \quad 1 \le a \le n-1.
\]
As $\vx_a \sim_{\text{i.i.d.}}\unifsp$ and $E$ is independent of $\set{\vx_a}_{a < n}$, we can conclude from the orthogonal invariant property of the spherical distribution that $\vv_a \sim_{\text{i.i.d.}} \unifsp$. The statement of the lemma then follows from the observation that
\[
\xi_a = \sqrt{d}\,\ve_1\tran \vv_a \quad \text{and} \quad \widetilde{\vx}_a = \frac{\vv_{a, \setminus 1}}{\Vert{\vv_{a, \setminus 1}}\Vert} , 
\]
where $\vv_{a, \setminus 1}$ is the $(d-1)$-dimensional vector obtained from $\vv_a$ after removing its first element. Then $\xi_a \sim \taumeasure$. Meanwhile, $\widetilde{\vx}_a $ is uniformly distributed over $\mathcal{S}^{d-2}$, and is independent of $\xi_a$. Note that the above derivations are done for fixed $\vx_n$ (i.e. when conditioned on $\vx_n$). However, since the distributions of $\set{\xi_a}$ and $\set{\widetilde{\vx}_a}$ are invariant to the choice of $\vx_n$, we can conclude that $\set{\xi_a}$ and $\set{\widetilde{\vx}_a}$ are indeed independent of $\vx_n$.
\end{proof}

It is easy to check that $\max_a \abs{\xi_a} = \mathcal{O}_\prec(1)$ [see \eref{qk_hpb}], and thus $\xi_a \xi_b / \sqrt{d} = \mathcal{O}_\prec(\frac{1}{\sqrt{d}})$. It is then natural to consider a Taylor expansion of the right-hand side of \eref{A_minus_n} to try to bring the term $\xi_a \xi_b / \sqrt{d}$ out of the polynomial $q_{k}(\cdot)$. In fact, a more careful analysis will give us the following expansion formula, which plays an important role in our proof. First, we need to introduce a set of functions: for each $\Eid \in \N_0$, let
\begin{equation}\label{eq:dxi}
\dxi_{\Eid} = \Bigg\{\sum_{0\le \iq,\jq \le \Eid} c_{\iq\jq}(d) \frac{q_{\iq}(\xi_1) q_{\jq}(\xi_2)}{d^{\max\set{\iq,\jq}/2 +1}}: c_{\iq\jq}(d) = \mathcal{O}_t(1) \text{ for } 0 \le \iq, \jq \le \Eid\Bigg\}.
\end{equation}
In what follows, we will use $\dxi_{\Eid}(\xi_1, \xi_2)$ to denote \emph{any} function in $\dxi_\Eid$. The exact form of $\dxi_{\Eid}(\xi_1, \xi_2)$ can change from one expression to another.

\begin{proposition}\label{prop:expansion}
For each $k = 0, 1, 2, \ldots, L$, we have
\begin{equation}\label{eq:expansion}
q_{k}\Big(r(\xi_1) r(\xi_2) x + \xi_1 \xi_2 /\sqrt{d}\Big) = \sum_{\Eid = 0}^k  \widetilde{q}_{k-\Eid}(x)r^{k-\Eid}(\xi_1) r^{k-\Eid}(\xi_2)\Big[\sqrt{(k)_\Eid} \frac{q_{\Eid}(\xi_1)q_{\Eid}(\xi_2)}{d^{\Eid/2}}+ \dxi_{\Eid}(\xi_1, \xi_2)\Big],
\end{equation}
where $(k)_\Eid = k! / (k-\Eid)!$ denotes the falling factorial, and $\widetilde{q}_k(x)$ is the $k$th Gegenbauer polynomial in dimension $d-1$ as defined in \eref{polynomial_s_d}.
\end{proposition}
\begin{proof}
See Appendix~\ref{appendix:expansion}.
%\hty{We should replace $r(\xi_i) r(\xi_j)$ by a constant of order one to make it clearer.  The proof is basically an expansion of the form 
%\[
%q_{k}\Big( x +y/\sqrt{d}\Big) \sim  \sum_{m = 0}^k  \widetilde{q}_{k-m}(x)   q_{m}(yd^{-1/2}) ,
%\]
%and then rewrite $q_{m}(yd^{-1/2})$ into main term and error terms. 
%}
\end{proof}

For the first  term on the right side of  \eqref{eq:expansion} with $t=0$, we define an $(n-1)\times(n-1)$ matrix $\widetilde{\mA}_k$, whose entries are given by:
\begin{equation}\label{eq:Ak_tilde}
(\widetilde{A}_k)_{ab} = \frac{1}{\sqrt{n}} \widetilde{q}_k(\sqrt{d-1} \,\widetilde{\vx}_a\tran \widetilde{\vx}_b)\,\charfn_{a \neq b}, \quad \text{for }a, b \in [n-1].
\end{equation}
We write their linear combination as
\begin{equation}\label{eq:A_tilde}
\widetilde{\mA} \bydef \sum_{\ell \le k \le L} \co_k \widetilde{\mA}_k.
\end{equation}
For the first  term on the right side of  \eqref{eq:expansion} with $t=k$, we define the vector
\begin{equation}\label{eq:v_xi}
\vv_k(\xi) \bydef (q_k(\xi_1), q_k(\xi_2), \ldots q_k(\xi_{n-1}))\tran.
\end{equation}
%and a diagonal matrix
%\[
%\Lambda_k(\xi) \bydef \diag\set{q_k(\xi_1), q_k(\xi_2), \ldots, q_k(\xi_{n-1})}.
%\]
Using the expansion in \eref{expansion}, we can now write the matrix in \eref{A_minus_n} as
\begin{equation}\label{eq:A_Delta}
\mA\mn = \widetilde{\mA} + \frac{\co_\ell }{\sqrt{n d^\ell / \ell!}}\vv_\ell(\xi) \vv_\ell(\xi)\tran + \sum_{1 \le i \le 5} \Delta_i,
\end{equation}
where $\Delta_1, \ldots, \Delta_5$ are matrices defined as follows: for $a, b \in [n-1]$,
\begin{subequations}
\begin{align}
({\Delta}_1)_{ab} &= \frac{-\co_\ell q_\ell^2(\xi_a)}{\sqrt{n d^\ell / \ell!}}\delta_{ab},\label{eq:Delta1}\\
({\Delta}_2)_{ab} &=  \sum_{\ell \le k \le L} \co_k \sum_{0 \le \Eid \le k}  \sqrt{(k)_\Eid}\cdot(\widetilde{A}_{k-\Eid})_{ab}(r^{k-\Eid}(\xi_a) r^{k-\Eid}(\xi_b)-1) \frac{q_{\Eid}(\xi_a)q_{\Eid}(\xi_b)}{d^{\Eid/2}},\label{eq:Delta2}\\
({\Delta}_3)_{ab} &=  \co_\ell \sum_{1 \le \Eid \le \ell-1} \sqrt{(\ell)_\Eid} \cdot (\widetilde{A}_{\ell-\Eid})_{ab} \frac{q_{\Eid}(\xi_a)q_{\Eid}(\xi_b)}{d^{\Eid/2}},\label{eq:Delta3}\\
({\Delta}_4)_{ab} &=  \sum_{\ell < k \le L} \co_k \sum_{1 \le \Eid \le k} \sqrt{(k)_\Eid}\cdot (\widetilde{A}_{k-\Eid})_{ab}  \frac{q_{\Eid}(\xi_a)q_{\Eid}(\xi_b)}{d^{\Eid/2}},\label{eq:Delta4}\\
({\Delta}_5)_{ab} &=  \sum_{\ell \le k \le L} \co_k \sum_{0 \le \Eid \le k}  (\widetilde{A}_{k-\Eid})_{ab}\cdot r^{k-\Eid}(\xi_a) r^{k-\Eid}(\xi_b) \dxi_{\Eid}(\xi_a, \xi_b).\label{eq:Delta5}
\end{align}
\end{subequations}
Here ${\Delta}_2$ represents the error incurred when replacing $r^{k-\Eid}(\xi_a) r^{k-\Eid}(\xi_b)$ by $1$; the error terms for $k=\ell$ with $t\ge 1$ are collected in ${\Delta}_3$;  the error terms for $\ell < k \le L$ are collected in ${\Delta}_4$; lastly, all the terms involving $\dxi$ are collected in ${\Delta}_5$.

%\yml{We will show that $\Delta_1, \ldots, \Delta_5$ can all be omitted in the resolvent expansion. In particular, $\Delta_1, \Delta_2$ and $\Delta_4$ can be controlled by showing that their operator norms are small. $\Delta_3$ can be controlled by using the fact that $\EE q_a(\xi) q_b(\xi) = 0$ for $a \neq b$. Finally, $\Delta_5$ can be controlled by a mixture of the above arguments: namely, the low-order polynomial terms are small due to $\EE q_a(\xi) q_b(\xi) = 0$ for $a \neq b$ and higher-order terms are controlled by operator norm.}

Later, we will demonstrate that the matrices $\set{\Delta_i}$ can all be considered as small noise terms that become negligible as $d \to \infty$. Moreover, by the construction of $\widetilde{A}_k$ in \eref{Ak_tilde} and by Lemma~\ref{lemma:rotation}, both $\widetilde{\mA}$ and its resolvent $\widetilde{G}(z) \bydef (\widetilde{A} - z I)^{-1}$ are independent of $\set{\xi_a}$. These observations lead us to consider the following expansion.

\begin{lemma}
Let $G\mn(z)$ and $\widetilde{G}(z)$ denote the resolvents of $\mA\mn$ and  $\widetilde{\mA}$, respectively. We have
\begin{equation}\label{eq:Schur_comp}
\begin{aligned}
&\sum_{a, b \neq n} A_{na} G\mn_{ab}(z) A_{nb} = \sum_{\ell \le \iq,\jq \le L} \co_{\iq} \co_{\jq} \chi_{\iq\jq}(z) - \theta(z) \sum_{\ell \le \jq \le L} \co_{\jq} \chi_{\ell \jq}(z)\\
&\qquad\qquad - \frac{1}{\sqrt{n}}\Big(\sum_{\ell \le \iq \le L} \co_{\iq} \vv_{\iq}\tran(\xi) - \theta(z) \vv_\ell\tran(\xi)\Big) \widetilde{G}(z) \Big(\sum_{1 \le c \le 5} \Delta_c\Big) G\mn(z) \Big(\frac{1}{\sqrt{n}} \sum_{\ell \le \jq \le L} \co_{\jq} \vv_{\jq}(\xi)\Big),
\end{aligned}
\end{equation}
where, for $\iq, \jq \in \N$, $\vv_{\iq}(\xi)$ is the vector defined in \eref{v_xi},
\begin{align}
\chi_{\iq\jq}(z) &\bydef \frac{1}{n} \vv_{\iq}\tran(\xi) \widetilde{G}(z) \vv_{\jq}(\xi),\label{eq:chi}\\
\intertext{and}
\theta(z) &\bydef \frac{\co_\ell \sum_{\ell \le \iq \le L} \co_{\iq} \chi_{\iq\ell}}{\sqrt{d^\ell/(\ell! n)} + \co_\ell\chi_{\ell\ell}}.\label{eq:theta}
\end{align}
\end{lemma}
\begin{proof}
We start by recalling the simple resolvent identity in \eref{resolvent_id}. Applying this identity on the two matrices $\mA\mn$ and $\widetilde{\mA} + \frac{\co_\ell \vv_\ell(\xi)\vv_\ell(\xi)\tran}{\sqrt{n d^\ell / \ell!}}$, and by using \eref{A_Delta}, we get
\begin{align}
G\mn(z) &= \Big(\widetilde{\mA} + \frac{\co_\ell \vv_\ell(\xi)\vv_\ell(\xi)\tran}{\sqrt{n d^\ell / \ell!}}-z\mI\Big)^{-1} \Big(\mI - \sum_{1 \le i \le 5} \Delta_i G\mn(z)\Big)\nonumber\\
&= \Big(\widetilde{G}(z) - \frac{\mu_\ell \widetilde{G}(z) \vv_\ell(\xi) \vv_\ell\tran(\xi) \widetilde{G}(z)}{\sqrt{n d^\ell/\ell!} + \co_\ell \vv_\ell\tran(\xi) \widetilde{G}(z) \vv_\ell(\xi)}\Big)\Big(\mI - \sum_{1 \le i \le 5} \Delta_i G\mn(z)\Big),\label{eq:resolvent_exp}
\end{align}
where the second equality follows from the Woodbury matrix identity. 
%\[
%({A} + \tau q q^*)^{-1} = {A}^{-1} - \frac {\tau A^{-1} q \cdot q^* A^{-1}} { 1 + \tau q^* A^{-1}  q}.
%\]
By \eref{An} and \eref{v_xi},
\[
A_{na} = { \frac 1 { \sqrt n}} \sum_{\ell \le k \le L} \co_k [\vv_k(\xi)]_a.
\]
Substituting this into \eref{resolvent_exp} then leads to the desired result.
\end{proof}

\subsection{Schur Complement: the Leading Term}
\label{sec:A_Schur}

Next, we demonstrate that the right-hand side of \eref{Schur_comp} can indeed be separated into a leading term, whose form is provided in \eref{H_complement}, and a small, random error term.

\begin{proposition}\label{prop:A_complement}
Let $A = \sum_{\ell \le k \le L} \co_k A_k$, and $i \in [n]$. Under the same conditions of Theorem~\ref{thm:equivalence}, we have
\begin{equation}\label{eq:A_complement_hpb}
\sum_{a, b \neq i} A_{ia} G\mi_{ab}(z) A_{ib} = \Big[\frac{\coa s(z)}{1 + \cob s(z)} + \coc s(z)\Big] + \mathcal{O}_\prec(d^{-1/2}),
\end{equation}
where $s(z)$ is the Stieltjes transform of the spectrum of $A$.
\end{proposition}

To prepare for the proof of Proposition~\ref{prop:A_complement}, we first establish several intermediate results.

\begin{lemma}\label{lemma:chi_leading}
For $1 \le \iq,\jq \le L$, we have
\begin{equation}\label{eq:chi_leading}
\chi_{\iq\jq}(z) = s(z) \delta_{\iq\jq} + \mathcal{O}_\prec\Big(\frac{1}{\eta^2 \sqrt{d}}\Big).
\end{equation}
\end{lemma}
\begin{proof}
Let $\widetilde{s}(z) \bydef \frac{1}{n} \tr \widetilde{G}$ and $s\mn(z) \bydef \frac{1}{n} \tr G\mn$ be the Stieltjes transforms of $\widetilde{G}$ and $G\mn$, respectively. We will establish \eref{chi_leading} in three steps, by showing (a) $\chi_{\iq\jq}(z) \approx \widetilde{s}(z) \delta_{\iq\jq}$; (b) $\widetilde{s}(z) \approx s\mn(z)$; and (c) $s\mn(z) \approx s(z)$.

For step (a), we first recall the Ward identity \eref{Ward}, which gives us
\[
\frac{1}{n} \sum_{1 \le a, b < n} \abs{\widetilde{G}_{ab}}^2 = \frac{\im(\widetilde{s}(z))}{\eta} \le \frac{1}{\eta^2},
\]
where the inequality is due to \eref{sH_bnd}. By the definitions presented in \eref{v_xi} and \eref{chi}, 
\[
\chi_{\iq\jq}(z) = \frac{1}{\sqrt{n}}\sum_{a,b} [\frac{1}{\sqrt{n}} \widetilde{G}_{ab}(z)]q_{\iq}(\xi_a) q_{\jq}(\xi_b).
\]
Applying \eref{quadratic_hpb} in Proposition~\ref{prop:quadratic_hpb} with $\Psi_c = 1/{\eta}$, we have
\begin{equation}\label{eq:chi_st}
\abs{\chi_{\iq\jq}(z) - \widetilde{s}(z) \delta_{\iq\jq}} \prec \frac{1}{\eta \sqrt{n}}.
\end{equation}

For step (b), we use the expansion formula in \eref{A_Delta}. By \eref{s_diff_rank}, \eref{s_diff}, and the triangular inequality,
\begin{equation}\label{eq:smn_st_1}
\abs{s\mn(z) - \widetilde{s}(z)} \le \frac{C}{n \eta} + \frac{\sum_{i = 1}^5 \norm{\Delta_i}_\mathsf{F}}{\sqrt{n} \eta^2} \le \frac{C}{n \eta} + \frac{\sqrt{n}}{\eta^2} \sum_{i = 1}^5 \norm{\Delta_i}_\infty,
\end{equation}
where $C$ is an absolute constant. Next, we bound the entry-wise $\ell_\infty$ norm of $\set{\Delta_i}$, whose constructions are given in \eref{Delta1} -- \eref{Delta5}. Our derivations will frequently use the following property of stochastic domination, which can be verified by a simple union bound: If $X_1 \prec Y_1$ and $X_2 \prec Y_2$, then 
\begin{equation}\label{eq:dominance_property}
X_1 X_2 \prec Y_1 Y_2 \quad\text{and}\quad X_1 + X_2 \prec Y_1 + Y_2.
\end{equation}
By \eref{qk_hpb}, and the property stated above, we have $\max_a \abs{q_k(\xi_a)}^2 \prec 1$. It follows that $\norm{\Delta_1}_\infty \prec \frac{1}{\sqrt{n d^\ell}}$. Similarly, by using \eref{Ak_inf_hpb}, we can easily verify that
\[
\norm{\Delta_3}_\infty \prec \frac{1}{\sqrt{n d}} \quad\text{and}\quad \norm{\Delta_4}_\infty \prec \frac{1}{\sqrt{n d}}.
\]
To bound $\Delta_5$, we first check from the definition in \eref{dxi} that $\max_{a, b \in [n-1]} \abs{\dxi_t(\xi_a, \xi_b)} \prec 1/d$. Moreover, by its construction in \eref{rxi}, the function $r(\xi) \le 1+\mathcal{O}(1/d)$ for all $\xi$. Combining these estimates then gives us $\norm{\Delta_5}_\infty \prec \frac{1}{d\sqrt{n}}$.

The matrix $\Delta_2$ in \eref{Delta2} requires some additional care, as it contains terms with index $t = 0$. First write
\begin{equation}\label{eq:rab_1}
r^{k-\Eid}(\xi_a) r^{k-\Eid}(\xi_b)-1 = (r^{k-\Eid}(\xi_a) -1)(r^{k-\Eid}(\xi_b) -1) + (r^{k-\Eid}(\xi_a) -1) + (r^{k-\Eid}(\xi_b) -1).
\end{equation}
Applying the high-probability bound \eref{rxi_hpb} and the property in \eref{dominance_property}, we get
\[
\max_{a,b} \abs{r^{k-\Eid}(\xi_a) r^{k-\Eid}(\xi_b)-1} \prec 1/d,
\]
which then implies that $\norm{\Delta_2}_\infty \prec \frac{1}{d\sqrt{n}}$. Substituting these bounds on $\norm{\Delta_c}_\infty$ for $1 \le c \le 5$ into \eref{smn_st_1} then leads to
\begin{equation}\label{eq:smn_st}
\abs{s\mn(z) - \widetilde{s}(z)} \prec \frac{1}{\eta^2 \sqrt{d}}.
\end{equation}
Now we move to step (c), where we compare $s\mn(z)$ with $s(z)$. By the interlacing properties of the eigenvalues of $A\mn$ and $A$, we have the following standard result (see \cite[Lemma 7.5]{erdos2017Dynamicalapproach} for a proof):
\begin{equation}\label{eq:smn_sA}
\abs{s\mn(z) - s(z)} \le \frac{C}{n \eta},
\end{equation} 
where $C$ is an absolute constant. Finally, the statement of the lemma can be obtained by applying the triangular inequality and by using the estimates in \eref{chi_st}, \eref{smn_st}, and \eref{smn_sA}.
\end{proof}

Next, we show that the terms involving $\set{\Delta_c}_{1 \le c \le 5}$ on the right-hand side of \eref{Schur_comp} are small. This is done in the following two lemmas.

\begin{lemma}\label{lemma:Lambda_cross}
Let $\idi, \idj, k$ be integers such that  $0 \le \idi, \idj, k \le L$ and $\idi \neq \idj$. We have
\begin{equation}\label{eq:Lambda_cross}
\norm{\vv_{\idi}\tran(\xi) \widetilde{G}(z) \Lambda_{\idj}(\xi) \widetilde{\mA}_k}  \prec \frac{\sqrt{n}}{\eta},
\end{equation}
where $\Lambda_{\idj}(\xi)$ denotes a diagonal matrix defined as
\begin{equation}\label{eq:Lambda_k}
\Lambda_{\idj}(\xi) \bydef \diag\set{q_{\idj}(\xi_1), q_{\idj}(\xi_2), \ldots, q_{\idj}(\xi_{n-1})}.
\end{equation}
\end{lemma}
\begin{proof}
Write $u \bydef \vv_{\idi}\tran(\xi) \widetilde{G}(z) \Lambda_{\idj}(\xi) \widetilde{\mA}_k$. This is a vector with $n-1$ entries. In our proof, we will show that 
\begin{equation}\label{eq:Lambda_cross_u}
\max_{a \in [n-1]} \, \abs{u_a} \prec {1}/{\eta},
\end{equation}
of which \eref{Lambda_cross} is then an immediate consequence. For each $a \in [n-1]$, define a matrix
\[
C_a \bydef \widetilde{G}(z) \diag\big\{(\widetilde{A}_k)_{1a}, (\widetilde{A}_k)_{2a}, \ldots, (\widetilde{A}_k)_{n-1,a}\big\}.
\]

One can verify that
\[
u_a = \vv_{\idi}\tran(\xi) C_a \vv_{\idj}(\xi). 
\]
%\hty{ there is still a factor $\Lambda_{\idj}(\xi) $, even it is order one} 
We start by proving \eref{Lambda_cross_u} for the special case when $\idi > \idj = 0$. Recall that $\vv_0(\xi) = \vec{1}$. We then have $\norm{C_a \vv_0(\xi)} \le \Vert\widetilde{G}\Vert_\mathsf{op} \norms{\widetilde A_k}_\infty \sqrt{n}\prec 1/\eta$, where the second step uses \eref{G_op} and \eref{Ak_inf_hpb}. Since $C_a$ is independent of $\vv_{\idi}(\xi)$, we can apply \eref{linear_hpb} in Proposition~\ref{prop:quadratic_hpb} with $\Psi_b = 1/\eta$ to get
\begin{equation}\label{eq:ua_i0}
u_a = \vv_{\idi}\tran(\xi) C_a \vv_0(\xi) \prec 1/\eta.
\end{equation}
By the same arguments, $\vv_0\tran(\xi) C_a \vv_\idj(\xi) \prec 1/\eta$ for $\idj > 0$. In what follows, we assume $1 \le \idi, \idj \le L$ with $\idi \neq \idj$. Using the Ward identity \eref{Ward} in the second step, and applying \eref{sH_bnd}, \eref{Ak_inf_hpb} in the third step, we get
\[
\norm{C_a}_\mathsf{F} \le \Vert\widetilde{G}\Vert_\mathsf{F} \norms{\widetilde A_k}_\infty \le \frac{\sqrt{n \cdot \im\,\widetilde{s}(z)}}{\sqrt{\eta}} \norms{\widetilde A_k}_\infty \prec \frac{1}{\eta}.
\]
By appealing to \eref{quadratic_hpb} in Proposition~\ref{prop:quadratic_hpb} with $\Psi_c = 1/\eta$, we have
\begin{equation}\label{eq:ua_ij}
u_a = \vv_{\idi}\tran(\xi) C_a \vv_j(\xi) \prec 1/\eta.
\end{equation}
One can check that the high-probability bounds in \eref{ua_i0} and \eref{ua_ij} are both uniform in $a \in [n-1]$. (To see this, note that we always bound $\max_i \vert(\widetilde A_k)_{ia}\vert$ via $\norms{\widetilde A_k}_\infty$ in the above derivations.) Applying the union bound over $a$ then gives us \eref{Lambda_cross_u}.
\end{proof}

\begin{lemma}
For $\idi, \idj \in \N$ with $\ell \le \idi,\idj \le L$, and for $1 \le c \le 5$, we have
\begin{equation}\label{eq:quadratic_Delta}
\frac{1}{n}\vv_{\idi}\tran(\xi)  \widetilde{G}(z) \Delta_c G\mn(z) \vv_{\idj}(\xi) = \mathcal{O}_\prec\Big(\frac{1}{\eta^2\sqrt{d}}\Big).
\end{equation}
\end{lemma}
\begin{proof}
By using \eref{G_op} in the first step and \eref{qk_hpb} in the second step, we have $\Vert{G\mn \vv_{\idj}(\xi)}\Vert \le \frac{1}{\eta} \norm{\vv_{\idj}(\xi)} \prec \frac{\sqrt{n}}{\eta}$. The same arguments will also give us $\Vert{\vv_{\idi}\tran(\xi) \widetilde{G}(z)}\Vert \prec \sqrt{n}/\eta$. It follows that
\begin{equation}\label{eq:Delta_op_bnd}
\abs{\frac{1}{n}\vv_{\idi}\tran(\xi)  \widetilde{G}(z) \Delta_c G\mn(z) \vv_{\idj}(\xi)} \prec \frac{1}{\eta^2} \norm{\Delta_c}_\mathsf{op}.
\end{equation}
Thus, one way to bound the left-hand side of \eref{quadratic_Delta} is to obtain estimates on the operator norm of $\Delta_c$. We start from $\Delta_1$ in \eref{Delta1}. Since it is a diagonal matrix,
\begin{equation}\label{eq:Delta1_op}
\norm{\Delta_1}_\mathsf{op} \le \frac{1}{\sqrt{nd^\ell}} \max_a\,  q_\ell^2(\xi_a) \prec \frac{1}{\sqrt{n d^\ell}},
\end{equation}
where the second step is due to \eref{qk_hpb}. For $\Delta_2$ in \eref{Delta2}, write $u_k \bydef \max_a \abs{r^k(\xi_a) - 1}$. Using the expansion in \eref{rab_1} and the triangular inequality, we have
\[
\norm{\Delta_2}_\mathsf{op} \le C\sum_{\ell \le k \le L}\sum_{0 \le t \le k} \frac{1}{d^{t/2}}\norms{\widetilde{A}_{k-t}}_\mathsf{op} (u_{k-t}^2 + 2 u_{k-t}) \max_a \abs{q_t(\xi_a)}^2
\]
for some constant $C$ that depends on $\set{\co_k}$. Note that $\widetilde{A}_k$ has exactly the same construction as $A_k$. The only difference is to change the dimension parameter from $d$ to $d-1$. Thus, we can apply Proposition~\ref{prop:Ak_op_bnd} to get $\norms{\widetilde{A}_{k-t}}_\mathsf{op} \le d^{\max\set{(\ell-k+t)/2, 0}}$. Combining this estimate with those in \eref{rxi_hpb} and \eref{qk_hpb} leads to 
\begin{equation}\label{eq:Delta2_op}
\norm{\Delta_2}_\mathsf{op} \prec C\sum_{\ell \le k \le L}\sum_{0 \le t \le k} d^{\max\set{(\ell-k+t)/2, 0}}d^{-t/2}d^{-1} \prec d^{-1},
\end{equation}
where the second step is due to $\max\set{(\ell-k+t)/2, 0} \le t/2$ for $k \ge \ell$. Similarly, for $\Delta_4$ in \eref{Delta4}, we have
\begin{equation}\label{eq:Delta4_op}
\norm{\Delta_4}_\mathsf{op} \prec C\sum_{\ell < k \le L}\sum_{1 \le t \le k} d^{\max\set{(\ell-k+t)/2, 0}}d^{-t/2} \prec d^{-1/2}.
\end{equation}
The second step follows from the fact that, for $k > \ell$ and $t \ge 1$, $\max\{\frac{\ell-k+t}{2}, 0\} \le \max\{\frac{t-1}{2}, 0\} = \frac{t-1}{2}$. By substituting \eref{Delta1_op}, \eref{Delta2_op}, and \eref{Delta4_op} into \eref{Delta_op_bnd}, we verify the statement in \eref{quadratic_Delta} for the cases of $c = 1, 2$ and $4$.

The cases of $\Delta_3$ and $\Delta_5$ require a different approach. We start with $\Delta_3$. Recall the diagonal matrix notation introduced in \eref{Lambda_k}. Write 
\begin{equation}\label{eq:LambdaGv}
\vu_\idj \bydef G\mn(z) \vv_{\idj}(\xi). 
\end{equation}
By \eref{qk_hpb} and \eref{G_op}, we have
\begin{equation}\label{eq:ut_bnd}
\norm{\vu_\idj} \le  \norms{G\mn(z)}_\mathsf{op} \cdot \norm{\vv_{\idj}(\xi)} \prec \frac{\sqrt{n}}{\eta}. 
\end{equation}
and
\begin{equation}\label{eq:ut_Lambda_bnd}
\norm{\Lambda_t(\xi) \vu_\idj} \le \big(\max_a \, \abs{q_t(\xi_a)}\big)\cdot \norms{\vu_\idj} \prec \frac{\sqrt{n}}{\eta}. 
\end{equation}

From the construction of $\Delta_3$ in \eref{Delta3}, one can check that 
\begin{align}
\abs{\frac{1}{n}\vv_{\idi}\tran(\xi)  \widetilde{G}(z) \Delta_3 G\mn(z) \vv_{\idj}(\xi)} &\le C \sum_{1 \le t \le \ell-1} \frac{1}{d^{t/2}n}\abs{\Big(\vv_{\idi}\tran(\xi) \widetilde{G}(z) \Lambda_t(\xi) \widetilde{A}_{\ell-t}\Big) \Lambda_t(\xi) \vu_\idj}\\
&\le C \sum_{1 \le t \le \ell-1} \frac{1}{d^{t/2}n}\norm{\vv_{\idi}\tran(\xi) \widetilde{G}(z) \Lambda_t(\xi) \widetilde{A}_{\ell-t}}\cdot  \norm{\Lambda_t(\xi)\vu_t}\\
&\prec \frac{1}{\eta^2 \sqrt{d}}.
\end{align} 
The second step uses the Cauchy-Schwarz inequality; in the last step, we have applied Lemma~\ref{lemma:Lambda_cross} (note that $\iq \ge \ell > t$), and the estimate in \eref{ut_Lambda_bnd}.

We complete the proof by showing \eref{quadratic_Delta} for $\Delta_5$. By construction, $\Delta_5$ is the linear combination of a finite number of matrices $\big\{\Delta_5^{(k, t)}\big\}$, indexed by $k, t$. The entries of  $\Delta_5^{(k, t)}$ are
\begin{align}
(\Delta_5^{(k, t)})_{ab} &= (\widetilde{A}_{k-t})_{ab} \cdot r^{k-\Eid}(\xi_a) r^{k-\Eid}(\xi_b) \dxi_{\Eid}(\xi_a, \xi_b)\\
&\overset{(a)}{=}(\widetilde{A}_{k-t})_{ab} \big(\dxi_{\Eid + 2\ell}(\xi_a, \xi_b) + f_2(\xi_a, \xi_b)\big)\\
&\overset{(b)}{=} (\widetilde{A}_{k-t})_{ab}\Big(\textstyle\sum_{0 \le i,j \le t+2\ell} \underbrace{c_{ij}(d) d^{-\max\{i, j\}/2-1} q_i(\xi_a) q_j(\xi_b)}_{(\ast)} + f_2(\xi_a, \xi_b)\Big)\\
&\overset{(c)}{=} (\widetilde{A}_{k-t})_{ab} \textstyle\sum_{\substack{0 \le i \le \ell-1\\0\le j \le t+2\ell}}(\ast) + \underbrace{(\widetilde{A}_{k-t})_{ab} \Big(\textstyle\sum_{\substack{\ell \le i \le t+2\ell\\0\le j \le t+2\ell}} (\ast) + f_2(\xi_a, \xi_b)\Big)}_{\eqqcolon R_{ab}}.\label{eq:Delta5_kt}
\end{align}
In step (a), we apply Proposition~\ref{prop:rxi_f1f2} in Appendix~\ref{appendix:concentration} to split $r^{k-\Eid}(\xi_a) r^{k-\Eid}(\xi_b) \dxi_{\Eid}(\xi_a, \xi_b)$ into the sum of two functions, with $f_1(\xi_a, \xi_b) \in \dxi_{t + 2\ell}$ and $f_2(\xi_a, \xi_b)$ satisfying the estimate in \eref{rxi_f2}. Step (b) uses the definition in \eref{dxi}; In step (c), we split the double sum over $0 \le i, j \le t+2\ell$ into two parts, with the index $i$ ranging from $0$ to $\ell-1$ in the first partial sum. Observe that each term in the partial sum in $R_{ab}$ includes a factor $d^{-\max\{i, j\}/2-1} \le d^{-\ell/2-1} = \mathcal{O}(n^{-1/2}d^{-1})$, because $i \ge \ell$. By \eref{Ak_inf_hpb}, \eref{qk_hpb}, and \eref{rxi_f2}, and by the union bound, we can verify that
\begin{equation}\label{eq:R_bnd}
\max_{a,b} \, \abs{R_{ab}} \prec \frac{1}{d n} \quad\text{and thus}\quad \norm{R}_\mathsf{op} \le \norm{R}_\mathsf{F} \prec \frac{1}{d}.
\end{equation}

Now consider the left-hand side of \eref{quadratic_Delta}, with $\Delta_c$ replaced by $\Delta_5^{(k, t)}$. By using \eref{Delta5_kt} and the notation introduced in \eref{Lambda_k} and \eref{LambdaGv}, we have
\begin{align}
&\frac{1}{n}\abs{\vv_{\idi}\tran(\xi)  \widetilde{G}(z) \Delta_5^{(k,t)} G\mn(z) \vv_{\idj}(\xi)} \le \frac{1}{nd} \sum_{\substack{0 \le i \le \ell-1\\0\le j \le t+2\ell}}  \abs{c_{ij}(d)\vv_\idi\tran(\xi) \widetilde{G}(z)\Lambda_i(\xi) \widetilde{A}_{k-t} \Lambda_j(\xi) \vu_\idj} + \frac{1}{n} \abs{\vv_{\idi}\tran(\xi) \widetilde{G}(z)R \vu_\idj}\\
&\qquad\qquad\le \frac{1}{nd} \sum_{\substack{0 \le i \le \ell-1\\0\le j \le t+2\ell}} \abs{c_{ij}(d)}\cdot\norm{\vv_\idi\tran(\xi) \widetilde{G}(z)\Lambda_i(\xi) \widetilde{A}_{k-t}} \cdot \norm{\Lambda_j(\xi) \vu_\idj} + \frac{1}{n} \norms{\vv_\idi\tran(\xi) \widetilde{G}(z)} \norm{R}_\mathsf{op} \norms{\vu_\idj}\label{eq:Delta5_kt_1}\\
&\qquad\qquad \prec \frac{1}{\eta^2 d}.\label{eq:Delta5_kt_2}
\end{align}
In the last step, we bound the right-hand side of \eref{Delta5_kt_1} as follows: for the terms in the sum over $i, j$, we have used the fact that $c_{ij}(d) = \mathcal{O}(1)$ and have applied \eref{ut_Lambda_bnd} and the estimate \eref{Lambda_cross} in Lemma~\ref{lemma:Lambda_cross}. (Note that $\alpha \ge \ell > i$ and thus \eref{Lambda_cross} is applicable.) For the last term on the right-hand side of \eref{Delta5_kt_1}, we use \eref{R_bnd}, \eref{ut_bnd}, and the estimate that $\norms{\vv_\idi\tran(\xi) \widetilde{G}(z)} \prec \sqrt{n} / \eta$.

The large deviation bound in \eref{Delta5_kt_2} holds for all $k, t$ over the range $\ell \le k \le L$ and $0 \le t \le k$. Applying the triangular inequality then yields the estimate in \eref{quadratic_Delta} for $\Delta_5$.
\end{proof}

We may now complete the proof of Proposition~\ref{prop:A_complement}.

\begin{proof}[Proof of Proposition~\ref{prop:A_complement}]
Observe that the different coordinates of $A$ are statistically exchangeable. Thus, it is sufficient to show \eref{A_complement_hpb} for any particular choice of the index $i$. We choose to do so for $i = n$. With the decomposition given in \eref{Schur_comp}, our task boils down to verifying that the right-hand side of \eref{Schur_comp} concentrates around the leading term given in \eref{A_complement_hpb}. To that end, we first show that the term $\theta(z)$ in \eref{theta} is close to
\[
\theta^\ast(z) \bydef \frac{\co_\ell^2 s(z)}{\sqrt{d^\ell/(\ell! n)} + \co_\ell s(z)}.
\]
Indeed, we have
\begin{equation}\label{eq:theta_leading0}
\abs{\theta(z) - \theta^\ast(z)} \le \frac{ \sum_{\ell < \iq \le L} \abs{\co_\ell \co_{\iq}\chi_{\iq\ell}}}{\phantom{\,}\absb{\sqrt{d^\ell/(\ell! n)} + \co_\ell \chi_{\ell \ell}}} + \frac{\co_\ell^2 \sqrt{d^\ell/(\ell! n)} \,\abs{\chi_{\ell\ell}-s(z)}}{\phantom{\,}\absb{\sqrt{d^\ell/(\ell! n)} + \co_\ell \chi_{\ell \ell}(z)}\cdot \absb{\sqrt{d^\ell/(\ell! n)} + \co_\ell s(z)}}.
\end{equation}
The two numerators on the right-hand side are small due to the estimates provided in Lemma~\ref{lemma:chi_leading}. The denominators can be bounded as follows. To lighten the notation, write $C_{d, \ell} \bydef \sqrt{d^\ell/(\ell! n)}$. Recall the definition of $\chi_{\ell\ell}(z)$ in \eref{chi}. By \eref{G_op},
\[
\abs{\co_\ell \chi_{\ell\ell}(z)} \le \frac{\abs{\mu_\ell}\cdot\norm{\vv_\ell(\xi)}^2}{n \eta}.
\]
Consider two cases: (a) If ${\abs{\mu_\ell}\cdot\norm{\vv_\ell(\xi)}^2}/({n \eta}) \le C_{d, \ell} / 2$, we have $\abs{\co_\ell \chi_{\ell\ell}(z)} \le C_{d, \ell} / 2$ and thus
\begin{equation}\label{eq:chi_inv0}
\frac{1}{\absb{C_{d, \ell} + \co_\ell \chi_{\ell \ell}}} \le \frac{2}{C_{d,\ell}}.
\end{equation}
(b) Suppose that ${\abs{\co_\ell}\cdot\norm{\vv_\ell(\xi)}^2}/({n \eta}) > C_{d, \ell} / 2$. Let $\{\lambda_j\}_{j \in [n-1]}$ denote the eigenvalues of $\widetilde{A}$. One can directly verify from the definition in \eref{chi} that
\[
\abs{\co_\ell} \im\,\chi_{\ell\ell}(z) \ge \frac{\abs{\co_\ell} \cdot \norm{\vv_\ell}^2}{n \eta} \min_{j \in [n-1]} \frac{\eta^2}{(\lambda_j - E)^2 + \eta^2} \ge \Big(\frac{C_{d,\ell}}{2}\Big) \frac{\eta^2}{2\norms{\widetilde{A}}_\mathsf{op}^2 + 2E^2 + \eta^2},
\]
where $E = \re(z)$. This then gives us
\begin{equation}\label{eq:chi_inv}
\frac{1}{\absb{C_{d, \ell} + \co_\ell \chi_{\ell \ell}}} \le \frac{2}{C_{d,\ell}} \Big(1 + \frac{2(\norms{\widetilde{A}}_\mathsf{op}^2 + E^2)}{\eta^2}\Big).
\end{equation}
Since the right-hand side of \eref{chi_inv} dominates that of \eref{chi_inv0}, the bound in \eref{chi_inv} in fact is valid for both cases. Note that $\abs{\co_\ell s(z)} \le \abs{\co_\ell}/\eta$. By considering two cases depending on whether $\abs{\co_\ell}/\eta \ge C_{d, \ell}/2$ and by following analogous arguments as above, we can also verify that
\begin{equation}\label{eq:s_inv}
\frac{1}{\absb{C_{d, \ell} + \co_\ell s(z)}} \le \frac{2}{C_{d,\ell}} \Big(1 + \frac{2(\norms{A}_\mathsf{op}^2 + E^2)}{\eta^2}\Big).
\end{equation}
We already have controls on the operator norm of $A$. Indeed, the estimates in \eref{Ak_op_bnd} give us
\[
\norm{A}_\mathsf{op} \le \sum_{k=\ell}^L \abs{\mu_k} \cdot \norm{A_k}_\mathsf{op} \prec 1.
\]
Similarly, $\norms{\widetilde{A}}_\mathsf{op} \prec 1$, as $\widetilde{A}$ is just the $(d-1)$-dimensional version of $A$. Substituting \eref{chi_inv} and \eref{s_inv} into \eref{theta_leading0}, and using the large deviation estimates in \eref{chi_leading}, we get
\begin{equation}\label{eq:theta_leading}
\abs{\theta(z) - \theta^\ast(z)} \prec \frac{1}{\sqrt{d}}.
\end{equation}
Finally, by applying \eref{theta_leading}, \eref{chi_leading}, and \eref{quadratic_Delta} to the right-hand side of \eref{Schur_comp}, and by using the assumption that $n/d^\ell = \ratio + \mathcal{O}(d^{-1/2})$, we reach the statement \eref{A_complement_hpb} of the proposition.
\end{proof}

\subsection{Proof of Theorem~\ref{thm:equivalence}}

The asymptotic limit of the ESD of the $B$ matrix can be obtained by standard methods in random matrix theory. In particular, since $B_\ell$ is a shifted Wishart matrix, its limiting ESD follows the MP law with a ratio parameter equal to $(\ell! \ratio)^{-1}$. Let $s_\ell(z)$ denote the Stieltjes transform of the limiting spectral density of the linearly-scaled version $\co_\ell B_\ell$. By using properties of the MP law (see Appendix~\ref{appendix:MP_law}), and after taking into account the linear transformation (namely, the scaling by $\co_\ell$ and shifting by $(\ell! \ratio)^{-1/2}$) of the eigenvalues, we can characterize $s_\ell(z)$ as the unique solution to
\begin{align} \label{eq:MPell}
\frac 1 {s_\ell (z)}     = -z -  \frac{\coa s_\ell(z)} {1+ \cob s_\ell (z)  },   %+   O(\frac n { N_k}    )  
\end{align} 
where $\coa, \cob$ are the constants defined in Theorem~\ref{thm:equivalence}.

By construction, the law of $B$ is equal to that of $\co_\ell B_\ell + \sqrt{\coc} H$, where $H \in \R^{n \times n}$ is a GOE matrix independent of $B_\ell$. Thus, the Stieltjes transform $s_B(z)$ of its ESD converges almost surely to a limiting function $m(z)$. Moreover, $m(z)$ can be obtained via a free additive convolution between the MP law and the semicircle law, which gives us $m(z) = s_\ell(z+  \coc m(z))$. It then follows from \eref{MPell} that
\begin{align}
\frac 1 {m(z)}= \frac 1 {s_\ell (z+  \coc m(z))}     & = -z-   \coc m(z) +    \frac{\coa s_\ell(z + \coc m(z))} {1+ \cob s_\ell (z+\coc m(z))  } 
\\ 
& =  -z +\frac{\coa m(z)} {1+ \cob m (z)  } -   \coc m(z),  %+   O(\frac n { N_k}    )  
\end{align} 
which is exactly \eqref{eq:m_equation}.

The rest of the proof is devoted to establishing the estimate in \eref{sA_m} for the matrix $A$. To lighten the notation, we drop the subscript in $s_A(z)$ in the following discussions. We start by recalling Lemma~\ref{lemma:omit_low_order}, which shows that removing the low-order components from $A$ will only incur an error of size $\mathcal{O}(1/\sqrt{d})$ in the Stieltjes transform. As this error is of the same size as the right-hand side of \eref{sA_m}, it is sufficient for us to establish \eref{sA_m} for the special case when $A$ contains only the high-order components, \emph{i.e.}, $A = \sum_{k = \ell}^L \co_k A_k$. 

By \eref{s_equation}, \eref{H_error}, \eref{H_complement}, and the high-probability estimate given in Proposition~\ref{prop:A_complement}, we have
\begin{equation}\label{eq:s_equation_approx}
s(z)\Big(z + \frac{\coa s(z)}{1 + \cob s(z)} + \coc s(z)\Big) + 1 = \Delta,
\end{equation}
where the error term $\Delta  \bydef -\frac{1}{n} \sum_i Q_i(z) G_{ii}(z)$ satisfies the estimate
\begin{equation}\label{eq:sA_small_error}
\abs{\Delta} \prec \frac{1}{\eta \sqrt{d}}.
\end{equation}
Thus, $s(z)$ is an approximate solution to \eref{m_equation}. To show $s(z) \approx m(z)$, we will apply the stability bound given in Proposition~\ref{prop:stability}. To that end, we first establish a high-probability upper bound for $1/\im[s(z)]$. Let $\set{\lambda_i}_{i \in [n]}$ denotes the eigenvalues of $A$, and $z = E + \i \eta \in C_+$. We have
\begin{equation}\label{eq:sz_lb0}
\im[s(z)] = \frac{1}{n} \sum_i \frac{\eta}{(\lambda_i - E)^2 + \eta^2} \ge \frac{\eta/2}{\norm{A}_\mathsf{op}^2 +  E^2 + \eta^2}.
\end{equation}
Applying the triangular inequality and Proposition~\ref{prop:Ak_op_bnd}, we have $\norm{A}_\mathsf{op} \le \sum_{k = \ell}^L \abs{\co_k} \cdot \norm{A_k}_\mathsf{op} \prec 1$. Moreover, by the assumption made in Theorem~\ref{thm:equivalence}, $\abs{E} \le \tau^{-1}$ and $\tau \le \eta \le \tau^{-1}$ for some fixed $\tau > 0$. Substituting these bounds into \eref{sz_lb0} then gives us
\begin{equation}\label{eq:sz_lb}
\frac 1{\im[s(z)]} \le \frac{\norm{A}_\mathsf{op}^2 +  E^2 + \eta^2}{\eta/2} \prec 1.
\end{equation}
For any $0 < \varepsilon < 1/2$  and $D > 0$, we have
\begin{align}
\P(\, \abs{s(z) - m(z)} \ge d^{\epsilon -1/2}) &\le \P\big(\abs{\Delta/s(z)} > \eta/2\big)+\P(\, \abs{s(z) - m(z)} \ge d^{\varepsilon -1/2}, \, \abs{\Delta/s(z)} \le \eta/2) \\
&\overset{\text{(a)}}{\le}\P\big(\abs{\Delta/s(z)} \ge \eta/2\big) +  \P\Big(\abs{\Delta/s(z)}\ge (\eta^2/4) d^{\varepsilon - 1/2}\Big)\\
&\overset{\text{(b)}}{\le}2\,  \P\Big(\abs{\Delta/s(z)}\ge (\eta^2/4) d^{\varepsilon - 1/2}\Big)\\
&\le 2\,  \P\Big(\abs{\Delta}(\im[s(z)])^{-1}\ge (\eta^2/4) d^{\varepsilon - 1/2}\Big)\\
&\overset{\text{(c)}}{\le} d^{-D}.
\end{align}
In step (a), the second term on the right-hand side of the inequality follows from Proposition~\ref{prop:stability}, where the error term $\et$ in \eref{m_equation_approx} is $\Delta/s(z)$; Step (b) holds for all sufficiently large $d$; Step (c) follows from the high-probability estimates given in \eref{sA_small_error} and \eref{sz_lb}. As $\varepsilon$ and $D$ can be chosen arbitrarily, we have verified \eref{sA_m}. The almost sure convergence of $s(z)$ to $m(z)$ then follows from \eref{sA_m} and the Borel-Cantelli lemma.
%\include{general_nonlinear}
% !TEX root = equivalence.tex

\section{The Isotropic Gaussian Model}
\label{sec:gaussian}

In this section, we show that the results stated in Theorem~\ref{thm:equivalence} and Theorem~\ref{thm:equivalence_f} still hold if the data vectors are sampled from the i.i.d. isotropic Gaussian distribution instead of the spherical distribution. Specifically, let $\vg_1, \ldots, \vg_n \in \R^d$ be a set of independent vectors drawn from $\mathcal{N}(0,  I_d/d)$. Consider a random matrix $\Gm \in \R^{n \times n}$ with entries
\begin{equation}\label{eq:G_f}
\Gm_{ij} \bydef \begin{cases} \frac{1}{\sqrt{n}} \, f_d(\sqrt{d}\, \vg_i\tran \vg_j) &\text{if } i \ne j\\
0&\text{if } i = j
\end{cases},
\end{equation}
where $f_d: \R \mapsto \R$ is a (nonlinear) kernel function. We show that the weak limit of the empirical spectral distribution of $\Gm$ is characterized by the same equivalence principle given in Theorem~\ref{thm:equivalence} and Theorem~\ref{thm:equivalence_f}.

Our approach is based on a simple coupling method, which allows us to define $\Gm$ in the same probability space as the matrix $A$ in \eref{A_f} and to write $\Gm$ as a perturbation of $A$. To that end, we rewrite $\vg_i$, for each $i \in [n]$, as
\begin{equation}\label{eq:g_decomp}
\vg_i = \rad_i \vx_i,
\end{equation}
where $\rad_i \bydef \norm{\vg_i}$ and $\vx_i \bydef \vg_i / \norm{\vg_i}$. By the rotational invariance of the isotropic Gaussian distribution, $\vx_i \sim \unifsp$, and $\vx_i$ is independent of $\rad_i$. In our subsequent discussions, we will often use the following standard concentration result (see, \emph{e.g.}, \cite[Theorem 3.1.1]{vershynin2018HighdimensionalProbabilityIntroduction}): for any $t > 0$,
\begin{equation}\label{eq:rad_concentrate}
\P(\,\abs{\rad_i - 1} \ge t) \le 2 e^{-c dt^2},
\end{equation}
where $c > 0$ is some absolute constant. In other words, $\sqrt{d}(\rad_i - 1)$ is a sub-Gaussian random variable whose sub-Gaussian norm is upper-bounded by an absolute constant. This then immediately implies that 
\begin{equation}\label{eq:rad_dominance}
\rad_i = 1 + \mathcal{O}_\prec(d^{-1/2}).
\end{equation}

\begin{proposition}\label{prop:sg_polynomial}
Let $A$ and $\Gm$ be the kernel random matrices defined in \eref{A_f} and \eref{G_f}, respectively, where the nonlinear function
\begin{equation}\label{eq:f_gen_polynomial}
f_d(x) = \sum_{0 \le b \le L} c_b x^b
\end{equation}
is an $L$th degree polynomial. Moreover, $A$ and $\Gm$ are constructed in the same probability space by using \eref{g_decomp}. Suppose that $n / d^\ell = \ratio + o(1)$ for some constants $\ratio > 0$ and $\ell \in \N$. Let $s_A(z)$ and $s_{\Gm}(z)$ denote the Stieltjes transforms of the ESDs of $A$ and $\Gm$, respectively. Then for any $z \in \C_+$ with $\im(z) = \eta$, we have
\begin{equation}
\abs{s_A(z) - s_{\Gm}(z)} \le \frac{C_L}{d \eta} + \mathcal{O}_\prec\Big( \frac{\max_{0 \le k \le L} \abs{c_k} }{\eta^2 \sqrt{d}}\Big),
\end{equation}
where $C_L$ is some constant that only depends on $L$.
\end{proposition}
\begin{proof}
Without loss of generality, we assume that $L \ge \ell$. By the definition in \eref{G_f}, for $i \neq j$, 
\begin{align}
\sqrt{n} \Gm_{ij} &= \textstyle\sum_{0 \le b \le L} c_b (\sqrt{d} \vg_i\tran \vg_j)^b\\
&\overset{\eref{g_decomp}}{=} \textstyle\sum_{0 \le b \le L} c_b (\sqrt{d} \vx_i\tran\vx_j)^b \rad_i^b \rad_j^b\\
&= \sqrt{n} A_{ij} + \textstyle\sum_{0 \le b \le L} c_b (\sqrt{d} \vx_i\tran\vx_j)^b \big[(\rad_i^b-1) + (\rad_j^b -1 ) + (\rad_i^b-1)(\rad_j^b-1)\big].\label{eq:GA_diff0}
\end{align}
For any $0 \le b \le L$, we can expand the monomial $x^b$ as a linear combination of Gegenbauer polynomials, \emph{i.e.},
\begin{equation}\label{eq:xcoeff}
x^b = \textstyle\sum_{0 \le k \le b} \xcoeff_{b, k} q_k(x),
\end{equation}
where $\set{\xcoeff_{b, k}}$ are some expansion coefficients. Note that $\set{\xcoeff_{b,k}}$ depend on the dimension $d$ [see, \emph{e.g.}, \eref{Gegenbauer_low_degree}], but we suppress this explicit dependence to streamline the notation. Recall the definition of the matrices $\set{A_k}$ in \eref{Ak}. Using \eref{xcoeff} and after rearranging the terms, we can write \eref{GA_diff0} as
\begin{equation}
\Gm_{ij} = A_{ij} + \underbrace{\textstyle\sum_{0 \le k \le L} (A_k)_{ij} \left(\sum_{k \le b \le L} c_b \xcoeff_{b, k}\big[(\rad_i^b-1) + (\rad_j^b -1 ) + (\rad_i^b-1)(\rad_j^b-1)\big]\right)}_{(\Delta_1)_{ij} + (\Delta_2)_{ij} + (\Delta_3)_{ij}}.
\end{equation}
Here, the difference between $A$ and $\Gm$ is represented by the sum of three $n\times n$ matrices, defined as
\begin{align}
(\Delta_1)_{ij} &= \textstyle\sum_{\ell \le k \le L} (A_k)_{ij} \left(\sum_{k \le b \le L} c_b \xcoeff_{b, k}\big[(\rad_i^b-1) + (\rad_j^b -1 ) + (\rad_i^b-1)(\rad_j^b-1)\big]\right)\\
(\Delta_2)_{ij} &= \textstyle\sum_{0 \le k < \ell} \left(-{\frac{\sqrt N_k}{\sqrt n}} \charfn_{ij}\right) \left(\sum_{k \le b \le L} c_b \xcoeff_{b, k}\big[(\rad_i^b-1) + (\rad_j^b -1 ) + (\rad_i^b-1)(\rad_j^b-1)\big]\right)\\
\intertext{and}
(\Delta_3)_{ij} &= \textstyle\sum_{0 \le k < \ell} \left((A_k)_{ij} + {\frac{\sqrt N_k}{\sqrt n}} \charfn_{ij}\right) \left(\sum_{k \le b \le L} c_b \xcoeff_{b, k}\big[(\rad_i^b-1) + (\rad_j^b -1 ) + (\rad_i^b-1)(\rad_j^b-1)\big]\right),
\end{align}
where $\set{N_k}$ are the constants given in \eref{Nk}. Next, we show that the perturbations by $\Delta_1, \Delta_2, \Delta_3$ only cause negligible changes in the Stieltjes transform of $A$. First, applying the triangular inequality gives us
\begin{align}\label{eq:Delta1_bnd0}
\norm{\Delta_1}_\mathsf{op} \le 2 L^2 \max_{\ell \le k \le L} \norm{A_k}_\mathsf{op} \max_{0 \le b \le L} \abs{c_b} \max_{0 \le k \le b \le L} \abs{\xcoeff_{b, k}} \Big(\max_{\substack{\ell \le b \le L\\ i \in [n]}} \vert{\rad_i^b - 1}\vert + \max_{\substack{\ell \le b \le L\\ i \in [n]}} \vert{\rad_i^b - 1}\vert^2 \Big).
\end{align}
By the estimates of Proposition~\ref{prop:Ak_op_bnd}, we have 
\begin{equation}\label{eq:max_Ak_op_bnd}
\max_{\ell \le k \le L} \norm{A_k}_\mathsf{op} \prec 1.
\end{equation}
To bound the expansion coefficients $\set{\xcoeff_{b, k}}$, we let $\xi$ be a random variable sampled from the distribution $\taumeasure$ in \eref{tau_measure}. Using \eref{xcoeff} and the orthogonality of the Gegenbauer polynomials $\set{q_k(x)}$ with respect to the law of $\xi$, we have
\begin{equation}\label{eq:xcoeff_bnd}
\abs{\xcoeff_{b, k}} = \abs{\EE \xi^b q_k(\xi)} \le (\EE \, \xi^{2b})^{1/2} \le (8b)^{b/2}.
\end{equation}
The second step of the above display uses H\"{o}lder's inequality and the fact that $\EE q_b^2(\xi) = 1$. The last step is due to the moment estimate in \eref{xi_moments}. Next, we bound the remaining terms on the right-hand side of \eref{Delta1_bnd0}. The estimate in \eref{rad_dominance} implies that $\rad_i^b = 1 + \mathcal{O}_\prec(d^{-1/2})$ for all $i \in [n]$ and $0 \le b \le L$. Thus, by the union bound, 
\begin{equation}\label{eq:rad_bnd}
\max_{\substack{\ell \le b \le L\\ i \in [n]}} \vert{\rad_i^b - 1}\vert = \mathcal{O}_\prec(d^{-1/2}).
\end{equation}
Substituting \eref{max_Ak_op_bnd}, \eref{xcoeff_bnd}, and \eref{rad_bnd} into \eref{Delta1_bnd0} then yields
\begin{equation}\label{eq:Delta1_bnd}
\norm{\Delta_1}_\mathsf{op} \prec \frac{\max_{0 \le b \le L} \abs{c_b}}{\sqrt d}.
\end{equation}

The operator norm of $\Delta_2$ can be bounded similarly:
\begin{align}
\norm{\Delta_2}_\mathsf{op} &\le 2 (L+1)^2 \max_{0 \le k < \ell} \frac{\sqrt N_k}{\sqrt n} \max_{0 \le b \le L} \abs{c_b} \max_{0 \le k \le b \le L} \abs{\xcoeff_{b, k}} \Big(\max_{\substack{\ell \le b \le L\\ i \in [n]}} \vert{\rad_i^b - 1}\vert + \max_{\substack{\ell \le b \le L\\ i \in [n]}} \vert{\rad_i^b - 1}\vert^2 \Big)\\
&\prec  \frac{\max_{0 \le b \le L} \abs{c_b}}{d},\label{eq:Delta2_bnd}
\end{align}
where the last step uses the estimate $\sqrt{N_k / n} = \mathcal{O}(d^{-1/2})$ for $0 \le k < \ell$ [recall the definition in \eref{Nk}] and the estimates given in \eref{xcoeff_bnd} and \eref{rad_bnd}.

Now we consider $\Delta_3$. Its operator norm is not small, but it has a low-rank structure. Recall from the decomposition in \eref{Ak_Sk} that $\rank(A_k + \sqrt{\frac{N_k}{n}} I) \le N_k$. It follows that
\begin{equation}\label{eq:Delta3_bnd}
\rank(\Delta_3) \le (L+1) \textstyle\sum_{0 \le k < \ell} N_k = \mathcal{O}_L(d^{\ell - 1}).
\end{equation}
Given the estimates in \eref{Delta1_bnd}, \eref{Delta2_bnd}, and \eref{Delta3_bnd}, we can reach the statement of the proposition by applying the perturbation inequalities \eref{s_diff} and \eref{s_diff_rank} in Lemma~\ref{lemma:sG_perturbation} and the triangular inequality.
\end{proof}

\begin{remark}
As an immediate consequence of Proposition~\ref{prop:sg_polynomial}, we observe that Theorem~\ref{thm:equivalence} remains valid even when data vectors are sampled from the isotropic Gaussian distribution, as opposed to the spherical distribution. More specifically, the estimate provided in \eref{sA_m} continues to hold when we substitute $s_A(z)$ with $s_{\Gm}(z)$, where $\Gm$ represents the kernel matrix defined in \eref{G_f} with $f_d$ corresponding to a linear combination of the first $L$ Gegenbauer polynomials, as expressed in \eref{f_polynomial}.
\end{remark}

Next, we establish a counterpart of Theorem~\ref{thm:equivalence_f} for the isotropic Gaussian case,  under the following assumption on the function $f_d$ in \eref{G_f}.
\begin{assumption}\label{assumption:fg}
Let $g \sim \mathcal{N}(0, 1)$, and let $\set{h_k(x)}_k$ represent the Hermite polynomials defined in \eref{ortho_poly_Hermite}. The function $f_d$ in \eref{G_f} satisfies the following conditions:
\begin{enumerate}
\item[(a)] For each $k \in \N_0$, we have
\begin{equation}\label{eq:cok_limit_b}
\EE [f_d(g) h_k(g)] \xrightarrow{d\to\infty} \co_k,
\end{equation}
for some finite numbers $\set{\co_k}$.

\item[(b)] The sequence $\set{\co_k}$ is square-summable, {i.e.}, 
\begin{equation}\label{eq:cok_summable_b}
\sigma^2 \bydef \sum_{k = 0}^\infty \co_k^2 < \infty.
\end{equation}
Moreover,
\begin{equation}\label{eq:cosigma_b}
\EE [f_d^2(g)] \xrightarrow{d\to\infty} \sigma^2.
\end{equation}

\item[(c)] Let $w(x)$ denote the probability density functions of $\mathcal{N}(0, 1)$, and let $\breve{w}_d(x)$ denote the density function of the random variable $\breve{\xi}_d = \sqrt{d} \vx\tran \vy$, where $\vx, \vy \in \R^d$ are two independent random vectors sampled from $\mathcal{N}(0, I_d/d)$. We have
\begin{equation}\label{eq:weight_w_wdb}
\int_{\R} f_d^2(x) \abs{w(x) - \breve{w}_d(x)} \dif x  \xrightarrow{d\to\infty} 0.
\end{equation}
\end{enumerate}
\end{assumption}

\begin{remark}
The conditions in Assumption~\ref{assumption:fg} closely resemble those in Assumption~\ref{assumption:f}. The only difference lies in the expression given in \eref{weight_w_wdb}, where we replace the probability density $w_d(x)$ from \eref{weight_w_wd} with a new density function, $\breve{w}_d(x)$. Furthermore, when $f_d = f$ is a function that remains independent of the dimension $d$, a simpler set of sufficient conditions can be employed, as stated in the following lemma.
\end{remark}

\begin{lemma}\label{lemma:f_const_condition_b}
A function $f(x)$ meets the conditions in Assumption~\ref{assumption:fg} if (a) $\EE[f^2(g)] < \infty$ with $g \sim \mathcal{N}(0, 1)$, and (b) there exist positive constants $c_1, c_2, c_3$ such that $\abs{f(x)} < c_1 e^{c_2 \abs{x}}$ when $\abs{x} \ge c_3$.
\end{lemma}
\begin{proof}
The conditions in \eref{cok_limit_b}, \eref{cok_summable_b}, and \eref{cosigma_b} directly follow from the assumption that $\EE[f^2(g)] < \infty$ and the completeness of the Hermite polynomials. To verify the condition in \eref{weight_w_wdb}, we can employ the estimate \eref{f2_w_wdb} provided in Lemma~\ref{lemma:w_wd_wdb} from Appendix~\ref{appendix:concentration}.
\end{proof}

\begin{theorem}\label{thm:equivalence_fg}
The statement of Theorem~\ref{thm:equivalence_f} remains valid if we replace the matrix $A$ with the matrix $\Gm$ as given in \eref{G_f}, and replace Assumption~\ref{assumption:f} with Assumption~\ref{assumption:fg}.
\end{theorem}
\begin{proof}
Our proof is a modification of the proof for Theorem~\ref{thm:equivalence_f}, and we provide the details in Appendix~\ref{appendix:proof_fg}.
\end{proof}

In concluding this section, we highlight an important caveat concerning the results presented above. While changing the data distribution from spherical to isotropic Gaussian does not alter the weak limit of the empirical spectral distribution of the kernel matrix, the Gaussian model may indeed introduce additional spike eigenvalues that are absent in the original spherical case. To illustrate this, let us consider $\Gm$ and $A$ as the kernel matrices defined in \eref{G_f} and \eref{A_f}, respectively, with the nonlinear function chosen as $f_d(x) = x^2 - 1$. Furthermore, we assume that $n/d^2 = \ratio + o(1)$ for a constant $\ratio > 0$.

Notice that the nonlinear function in this case is equivalent (up to a constant) to the second-degree Gegenbauer polynomial $q_2(x)$, as seen in \eref{Gegenbauer_low_degree}. Consequently, we can apply Proposition~\ref{prop:Ak_op_bnd}, resulting in:
\begin{equation}\label{eq:A_op_scq2}
\norm{A}_\mathsf{op} \prec 1,
\end{equation}
This implies that, for any $\epsilon > 0$, the spectrum of $A$ will be confined within the interval $[-d^\epsilon, d^\epsilon]$ with high probability as $d \to \infty$. However, this is not the case for $\Gm$. Although its bulk eigenvalues share the same weak limit as those of $A$, the matrix $\Gm$ exhibits two spike eigenvalues of order $\mathcal{O}(\sqrt{d})$.

To see that, we first use the decomposition \eref{g_decomp} to express
\begin{align}
\Gm_{ij} &= \frac{1}{\sqrt n} \big[(\sqrt d \vx_i\tran \vx_j \rad_i \rad_j)^2 - 1 \big] \charfn_{i \neq j}\\
&=A_{ij} \rad_i^2 \rad_j^2 + \frac{1}{\sqrt n}(\rad_i^2 \rad_j^2 - 1) \charfn_{i \neq j}.\label{eq:G_A_perturb0}
\end{align}
By defining $\delta_i \bydef \rad_i^2 - 1$ for $i \in [n]$, we can further expand \eref{G_A_perturb0}:
\begin{equation}\label{eq:G_A_perturb}
\Gm_{ij} = A_{ij} + \underbrace{A_{ij} \delta_i + A_{ij} \delta_j + A_{ij} \delta_i \delta_j -\frac{1}{\sqrt n} (\rad_i^2 \rad_j^2 -1)\charfn_{i = j}}_{\eqqcolon E_{ij}} + \underbrace{\frac{1}{\sqrt n} (\delta_i + \delta_j + \delta_i \delta_j)}_{\eqqcolon F_{ij}},
\end{equation}
where $E$ and $F$ are two $n\times n$ matrices representing the difference between $\Gm$ and $A$. Observe that the contribution from $E$ is negligible. Indeed, from the estimate in \eref{rad_dominance}, we obtain $\delta_i = \mathcal{O}_\prec(d^{-1/2})$, and by applying the union bound, $\max_{i \in [n]} \abs{\delta_i} = \mathcal{O}_\prec(d^{-1/2})$. Employing these estimates and the one in \eref{A_op_scq2}, we have
\begin{align}
\norm{E}_\mathsf{op} \le  \max_{i \in [n]} \{2\abs{\delta_i} + \delta_i^2\} \norm{A}_\mathsf{op} + \frac{1}{\sqrt n} (1 + \max_{i \in [n]} \rad_i^4) \prec \frac{1}{\sqrt d}.
\end{align}
Next, we examine the second perturbation matrix $F$ in \eref{G_A_perturb}. $F$ has rank $2$, and its two nonnegative eigenvalues, denoted by $\lambda^+$ and $\lambda^-$, can be directly computed as
\begin{equation}\label{eq:l1l2_0}
\lambda^+,  \lambda^- = \frac{\sqrt n}{d} \Big(\frac{c_1}{2} + \sqrt{d} c_2  \pm \sqrt{c_1 d} \sqrt{1+ c_1/(4d) + c_2/\sqrt{d}}\Big),
\end{equation}
where
\begin{equation}
c_1 = \frac{1}{n}\sum_{i \in [n]} (\sqrt{d}\delta_i)^2 \quad \text{and}\quad c_2 = \frac{1}{n} \sum_{i \in [n]} (\sqrt{d} \delta_i).
\end{equation}
It is easy to verify that $\EE \delta_i = 0$ and $\EE (\sqrt{d} \delta_i)^2 = 2$. Furthermore, the concentration inequality \eref{rad_concentrate} implies that $\norms{\sqrt{d} \delta_i}_{L^p} \le C_p$ for every $p \in \N$, where $C_p$ is some finite constant that only depend on $p$. Using these estimates, we can apply Lemma~\ref{lemma:linear_quadratic_lp} to get
\begin{equation}\label{eq:estimates_c1c2}
c_1 = 2 + \mathcal{O}_\prec(n^{-1/2}) \quad \text{and} \quad c_2 = \mathcal{O}_\prec(n^{-1/2}).
\end{equation}
After substituting \eref{estimates_c1c2} into \eref{l1l2_0}, the formula can be simplified as
\[
\lambda^+,  \lambda^- = \frac{\sqrt n}{d}(\pm \sqrt{2d} + 1) + \mathcal{O}_\prec(d^{-1/2}).
\]
Since $n/d^2 \to \ratio$, the two nonzero eigenvalues of $F$ are thus close to $\pm \sqrt{2\kappa d}$. Finally, given that $G = A + E + F$ with $\norm{A+E}_\mathsf{op} \prec 1$, standard eigenvalue perturbation arguments allow us to conclude that the matrix $\Gm$ also has two large spike eigenvalues close to $\pm \sqrt{2\kappa d}$.

% !TEX root = equivalence.tex

\section{Summary and Discussions}
\label{sec:summary}

Our main results, Theorem~\ref{thm:equivalence} and Theorem~\ref{thm:equivalence_f}, establish the weak limit of the spectrum for random inner-product kernel matrices in the polynomial scaling regime, where $d, n \to \infty$ such that $n / d^\ell \to \kappa \in (0, \infty)$ for a fixed $\kappa \in \R$ and $\ell \in \N$. The central insight of this work is an asymptotic equivalence principle, which asserts that the limiting spectral distribution of the kernel matrix corresponds to that of a linear combination of a (shifted) Wishart matrix and an independent GOE matrix. Consequently, the limiting distribution is characterized by a free additive convolution between an MP law and a semicircle law.

We have established our results only for data vectors sampled from spherical or isotropic Gaussian distributions. In fact, Theorem~\ref{thm:equivalence} is expected to hold (with certain modifications in error bounds) for more general data distributions exhibiting reasonably fast decay at infinity. It is also worth noting that the error bounds in Theorem~\ref{thm:equivalence} are not optimal. Moreover, we currently constrain the imaginary part of the spectral parameter $z = E + i \eta$ to $\eta > \tau$ for some constant $\tau$. Although our proof approach can accommodate $\eta \ge d^{-c}$ for some small constant $c$ with more careful book-keeping, this is still far from the regime $\eta \sim n^{-1 + \epsilon}$ required to reach individual eigenvalue locations.

Extending Theorem~\ref{thm:equivalence} to encompass the entire range $n^{-1 + \epsilon} \le \eta \le 1$ may prove challenging due to the presence of multiscale structures in the eigenvalues. These structures emerge because the matrix $A$ in \eqref{eq:A} is a sum of component matrices $\set{A_k}$ across different scales. As a result, accurately characterizing individual eigenvalues, including extremal ones, poses an interesting open problem. Related issues, such as eigenvector statistics and eigenvalue universality, also hinge on addressing this matter. The potential for multiscale structures sets the nonlinear model apart from standard mean field models like Wigner matrices or sample covariance matrices \cite{erdos2017Dynamicalapproach}. We aim to tackle some of these questions in future papers.

\appendix

% !TEX root = equivalence.tex

\section{Gegenbauer Polynomials and Spherical Harmonics}
\label{appendix:ortho_poly}

In this appendix, we collect a few useful properties of the (normalized) Gegenbauer polynomials $\{q_k(x)\}_{k \ge 0}$ defined in \eref{ortho_poly}. All of these results are standard, and their proofs can be found in \emph{e.g.}, \cite{dai2013ApproximationTheory, efthimiou2014Sphericalharmonics}.

\emph{Three-term recurrence relation}: For each $k \in \N$, we have
\begin{equation}\label{eq:recurrent}
xq_{k}(x) = a_k q_{k+1}(x) + a_{k-1} q_{k-1}(x),
\end{equation}
where 
\begin{equation}\label{eq:recurrent_a}
a_k = \sqrt{\frac{(k+1)(d+k-2)d}{(d+2k)(d+2k-2)}} = \sqrt{k+1}\, (1 + \mathcal{O}(d^{-1})).
\end{equation}
We will repeatedly apply this recurrence relation in our proof of Proposition~\ref{prop:expansion}. 

\emph{Spherical harmonics and the addition theorem}: We recall that a spherical harmonic of degree $k$ is a harmonic and homogeneous polynomial of degree $k$ in $d$ variables that are restricted to $\mathcal{S}^{d-1}$. The set of all degree $k$ harmonics forms a linear subspace, whose dimension is denoted by $\Nk$. We have $\Nk[0] = 1$, $\Nk[1] = d$, and in general, for $k \in \N$,
\begin{equation}\label{eq:Nk}
N_k \bydef  \frac{2k+d-2}{k} {{k+d-3} \choose {k-1}} = \frac{d^k}{k!}(1+\mathcal{O}(1/d)).
\end{equation}

Using the Gram-Schmidt procedure, we can construct an orthonormal set of spherical harmonics. Let $\har{k}{i}(\vx)$, for $i \in [\Nk]$, denote the $i$th degree-$k$ spherical harmonic in this orthonormal set. We then have
\begin{equation}\label{eq:sh_ortho}
\EE_{\vx \sim \unifsp}\har{k}{a}(\vx) \har{k',b}(\vx) = \charfn_{k k'} \charfn_{ab},
\end{equation}
for any $k, k' \in \N_0$ and any $a \in [\Nk], b \in [\Nk[k']]$.

The Gengenbauer polynomials and spherical harmonics are deeply connected. In particular, we have the following identity, which can be viewed as a high-dimensional generalization of the classical addition theorem \cite{coster1991SimpleProofAddition}: for any $\vx, \vy \in \sph$, 
\begin{equation}\label{eq:linearization}
q_{k}(\sqrt{d} \, \vx\tran \vy) = \frac{1}{\sqrt{\Nk}}\sum_{a \in [\Nk]} \har{k}{a}(\vx) \har{k}{a}(\vy).
\end{equation}
In addition, for the special case of $\vx = \vy$, we have
\begin{equation}\label{eq:har_norm}
q_{k}(\sqrt{d}) = \frac{1}{\sqrt{\Nk}}\sum_{a \in [\Nk]} \har{k}{a}^2(\vx) = \sqrt{\Nk}.
\end{equation}
The identity in \eref{linearization} is useful because it allows us to ``linearize'' the term $q_{k}(\sqrt{d} \, \vx\tran \vy)$ as an inner product of two vectors made of the spherical harmonics. Moreover, \eref{sh_ortho} implies that, if $\vx \sim \unifsp$, the vector of spherical harmonics $(\har{k}{1}(\vx), \ldots, \har{k}{\Nk}(\vx))$ is an isotropic random vector.

Using \eref{linearization} and \eref{har_norm}, we can rewrite the matrix $\mA_k$ defined in \eref{Ak} as
\begin{equation}\label{eq:Ak_factor}
\mA_k = \frac{1}{\sqrt{n \Nk}} \SH_k\tran \SH_k - \sqrt{\frac{\Nk}{n}}\, \mI,
\end{equation}
where $\SH_k$ is a $\Nk \times n$ matrix whose entries are the spherical harmonics, \emph{i.e.}
\begin{equation}\label{eq:Yk_def}
(\SH_k)_{ai} = Y_{k, a}(\vx_i),
\end{equation}
and $\set{\vx_i}_{i \in [n]}$ are the data vectors in the definition in \eref{Ak}.

\section{Proof of Proposition~\ref{prop:expansion}}
\label{appendix:expansion}

We first state several simple properties of the set $\dxi_m$ introduced in \eref{dxi}.

\begin{lemma}\label{lemma:dxi_properties}
For $m \in \N_0$, we have
\begin{subequations}
\begin{equation}\label{eq:dxi_p1}
\dxi_m(\xi_1, \xi_2) \in \dxi_{m+1},
\end{equation}
\begin{equation}\label{eq:dxi_p2}
\frac{\xi_1 \xi_2}{\sqrt{d}} \dxi_m(\xi_1, \xi_2) \in \dxi_{m+1},
\end{equation}
\begin{equation}\label{eq:dxi_p3a}
\frac{\xi_1^2}{d} \dxi_m(\xi_1, \xi_2) \in \dxi_{m+2}, \ \frac{\xi_2^2}{d} \dxi_m(\xi_1, \xi_2) \in \dxi_{m+2}
\end{equation}
and
\begin{equation}\label{eq:dxi_p3}
r^2(\xi_1)r^2(\xi_2) \dxi_m(\xi_1, \xi_2) \in \dxi_{m+2},
\end{equation}
\end{subequations}
where $r(\xi)$ is the function defined in \eref{rxi}.
\end{lemma}
\begin{proof}
The property \eref{dxi_p1} follows immediately from the definition of $\dxi_m$. To show \eref{dxi_p2}, \eref{dxi_p3a}, and \eref{dxi_p3}, we can assume without loss of generality that
\begin{equation}\label{eq:monomial}
\dxi_m(\xi_1, \xi_2) = \frac{q_a(\xi_1) q_b(\xi_2)}{d^{\max\set{a, b}/2+1}}
\end{equation}
for some $a, b$ such that $0 \le a, b \le m$. The more general case, where $\dxi_m(\xi_1, \xi_2)$ is a linear combination of terms like \eref{monomial}, can be handled by using the linearity of $\dxi_{m+1}$ and $\dxi_{m+2}$.

The recurrent relation \eref{recurrent} implies that $\xi_1 q_a(\xi_1) = \sum_{i \in [0, a+1]} \mu_i q_i(\xi_1)$ and $\xi_2 q_b(\xi_2) = \sum_{j \in [0, b+1]} \nu_j q_i(\xi_1)$, where $\mu_i = \mathcal{O}(1), \nu_i = \mathcal{O}(1)$ are some fixed expansion coefficients. It follows that
\[
\frac{\xi_1 \xi_2}{\sqrt{d}} \dxi_m(\xi_1, \xi_2) = \frac{\sum_{i=0}^{a+1} \sum_{j = 0}^{b+1} \mu_i \nu_j q_i(\xi_1) q_j(\xi_2)}{d^{\max\set{a+1, b+1}/2+1}} = \frac{\sum_{i=0}^{a+1} \sum_{j = 0}^{b+1} ({\mu_i \nu_j}/{d
^{\pi_{ij}}}) q_i(\xi_1) q_j(\xi_2)}{d^{\max\set{i, j}/2+1}},
\]
where $\pi_{ij} = \max\set{a+1, b+1}/2 - \max\set{i, j}/2 \ge 0$. Thus, we $\frac{\xi_1 \xi_2}{\sqrt{d}} \dxi_m(\xi_1, \xi_2) \in \dxi_{m+1}$.

The properties stated in \eref{dxi_p3a} can be proved similarly. Consider the function in \eref{monomial}. We have $\xi_1^2 q_a(\xi_1) = \sum_{i = 0}^{a+2} \mu_i q_i(\xi_1)$ for some expansion coefficients $\mu_i = \mathcal{O}(1)$. It follows that
\[
({\xi_1^2}/{d}) \dxi_m(\xi_1, \xi_2) = \frac{\sum_{i=0}^{a+2} \frac{\mu_i}{d
^{\pi_{i}}} q_i(\xi_1) q_b(\xi_2)}{d^{\max\set{i, b}/2+1}},
\]
where $\pi_i =  \max\set{a+2, b+2}/2 - \max\set{i, b}/2 \ge 0$. Since $\mu_i / d^{\pi_i} = \mathcal{O}(1)$, we have $(\xi_1^2/d) \dxi_m(\xi_1, \xi_2) \in \dxi_{m+2}$. That $(\xi_2^2/d) \dxi_m(\xi_1, \xi_2) \in \dxi_{m+2}$ follows from analogous arguments.

Recall the definition of $r(\xi)$ in \eref{rxi}. We have
\begin{align}
r^2(\xi_1)r^2(\xi_2) \dxi_m(\xi_1, \xi_2) &= (1-1/d)^{-1} (1 - \xi_1^2/d)(1-\xi_2^2/d) \dxi_m(\xi_1, \xi_2)\nonumber\\
	&= (1 + \mathcal{O}(1/d)) \Big(1 + (\xi_1 \xi_2 /\sqrt{d})^2/d - \xi_1^2/d - \xi_2^2/d\Big) \dxi_m(\xi_1, \xi_2).  \label{eq:dxi_r}
\end{align}
Note that $\dxi_m(\xi_1, \xi_2) \in \dxi_{m+2}$ by \eref{dxi_p1}. Applying the property in \eref{dxi_p2} twice, we have $(\xi_1 \xi_2 /\sqrt{d})^2 \dxi_m(\xi_1, \xi_2) \in \dxi_{m+2}$. By \eref{dxi_p3a}, the last two terms of \eref{dxi_r} also belong to $\dxi_{m+2}$. By the linearity of $\dxi_{m+2}$, we can conclude that the left-hand side of \eref{dxi_p3} indeed belongs to $\dxi_{m+2}$.
\end{proof}

\begin{lemma}
For each $k \in \N$, we have
\begin{equation}\label{eq:recurrent_xiq}
\frac{\xi_1 \xi_2 q_k(\xi_1) q_k(\xi_2)}{d^{(k+1)/2}} = (k+1)\frac{q_{k+1}(\xi_1) q_{k+1}(\xi_2)}{d^{(k+1)/2}} + \sqrt{k(k+1)}\frac{q_{k-1}(\xi_1)q_{k+1}(\xi_2) + q_{k+1}(\xi_1) q_{k-1}(\xi_2)}{d^{(k+1)/2}} + \dxi_{k+1}(\xi_1, \xi_2),
\end{equation}
and
\begin{equation}\label{eq:recurrent_sqrt}
\frac{r^2(\xi_1) r^2(\xi_2) q_{k-1}(\xi_1) q_{k-1}(\xi_2)}{d^{(k-1)/2}} = \frac{q_{k-1}(\xi_1) q_{k-1}(\xi_2)}{d^{(k-1)/2}} - \sqrt{k(k+1)}\frac{q_{k-1}(\xi_1)q_{k+1}(\xi_2) + q_{k+1}(\xi_1) q_{k-1}(\xi_2)}{d^{(k+1)/2}} + \dxi_{k+1}(\xi_1, \xi_2).
\end{equation}
\end{lemma}
\begin{proof}
To lighten the notation, we write, for any $a, b \in \N_0$,
\begin{equation}\label{eq:shorthand}
q_{a,b} = q_a(\xi_1)q_b(\xi_2) \quad\text{and}\quad r_{1,2} = r(\xi_1)r(\xi_2).
\end{equation}
To show \eref{recurrent_xiq}, we use the recurrent relation in \eref{recurrent}, which gives us $\xi q_{k}(\xi) = a_k q_{k+1}(\xi) + a_{k-1} q_{k-1}(\xi)$. It follows that
\[
\frac{\xi_1 \xi_2 q_{k,k}}{d^{(k+1)/2}} = a_k^2 \frac{q_{k+1,k+1}}{d^{(k+1)/2}} + a_ka_{k-1}\frac{q_{k-1,k+1} + q_{k+1,k-1}}{d^{(k+1)/2}} + a_{k-1}^2 \frac{q_{k-1,k-1}}{d^{(k-1)/2+1}}.
\]
Note that the last term on the right-hand side of above expression is in $\dxi_{k-1}$ and thus in $\dxi_{k+1}$ [by \eref{dxi_p1}]. From \eref{recurrent_a}, we have $a_k^2 = (k+1) (1 + \mathcal{O}(1/d))$ and $a_k a_{k-1} = \sqrt{k(k+1)}  (1 + \mathcal{O}(1/d))$. Thus, $a_k^2 \frac{q_{k+1,k+1}}{d^{(k+1)/2}} = (k+1) \frac{q_{k+1,k+1}}{d^{(k+1)/2}} + \dxi_{k+1}(\xi_1, \xi_2)$, and similarly, $a_k a_{k-1}\frac{q_{k-1,k+1} + q_{k+1,k-1}}{d^{(k+1)/2}} = \sqrt{k(k+1)}\frac{q_{k-1,k+1} + q_{k+1,k-1}}{d^{(k+1)/2}} + \dxi_{k+1}(\xi_1, \xi_2)$. 

Next, we show \eref{recurrent_sqrt}. Recall from the definition in \eref{rxi} that $\sqrt{1-d^{-1}} r^2(\xi) = 1-\xi^2/d$. Thus,
\begin{align}
\sqrt{1-d^{-1}} r^2(\xi) q_{k-1}(\xi) &= q_{k-1}(\xi) - \big[\xi^2 q_{k-1}(\xi)\big] / d\\
	&= q_{k-1}(\xi) - \big[a_{k-1} a_k q_{k+1}(\xi) + (a_{k-1}^2 + a_{k-2}^2) q_{k-1}(\xi) + a_{k-2}a_{k-3} q_{k-2}(\xi)\big] / d,\label{eq:recurrent_sqrt1}
\end{align}
where to reach \eref{recurrent_sqrt1}, we have applied the recurrent formula \eref{recurrent} twice. Here, we have assumed that $k \ge 3$, but the formula is valid for all $k \in \N$, if we define $a_{-1} = a_{-2} = 0$ and $q_{-1}(\xi) = q_{-2}(\xi) = 0$. Using \eref{recurrent_sqrt1}, we can write
\[
\frac{r_{1,2}^2 q_{k-1, k-1}}{d^{(k-1)/2}} =  [d/(d-1)] \frac{q_{k-1, k-1}}{d^{(k-1)/2}} - [d/(d-1)] a_{k-1}a_k \frac{q_{k-1, k+1} + q_{k+1, k-1}}{d^{(k+1)/2}} + \dxi_{k+1}(\xi_1, \xi_2),
\]
which then leads to \eref{recurrent_sqrt}, as $d/(d-1) = 1 + \mathcal{O}(1/d)$ and $[d/(d-1)]a_{k-1} a_k = \sqrt{k(k+1)} + \mathcal{O}(1/d)$.
\end{proof}

Next, we prove Proposition~\ref{prop:expansion} by induction on the polynomial degree $k$.  Throughout the proof, we use the shorthand notation introduced in \eref{shorthand}. Recall that $q_0(x) = \widetilde{q}_0(x) = 1$ and $q_1(x) = \widetilde{q}_1(x) = x$. It is then straightforward to verify the formula \eref{expansion} for $k = 0, 1$. Specifically, 
\begin{equation}\label{eq:q0_e}
q_0(r_{1,2} x + \frac{\xi_1 \xi_2}{\sqrt{d}}) = \widetilde{q}_0(x)
\end{equation}
and
\begin{equation}\label{eq:q1_e}
q_1(r_{1,2} x + \frac{\xi_1 \xi_2}{\sqrt{d}}) = \widetilde{q}_1(x) r_{1,2} + \widetilde{q}_0(x) \frac{q_{1,1}}{\sqrt{d}}.
\end{equation}

Now we carry out the induction. Assume that \eref{expansion} holds for $k$ and $k-1$, with some $k \ge 1$. To prove it for $k + 1$, we apply the recurrence relation in \eref{recurrent}, which gives us
\begin{align}
a_{k} q_{k+1}(r_{1,2} x + \frac{\xi_1 \xi_2}{\sqrt{d}}) &= (r_{1,2} x + \frac{\xi_1 \xi_2}{\sqrt{d}}) q_{k}(r_{1,2} x + \frac{\xi_1 \xi_2}{\sqrt{d}}) - a_{k-1} q_{k-1}(r_{1,2} x + \frac{\xi_1 \xi_2}{\sqrt{d}})\\
&= (r_{1,2} x + \frac{\xi_1 \xi_2}{\sqrt{d}}) \sum_{m = 0}^{k} \qt_{k - m}(x) r_{1,2}^{k - m}\Big[\sqrt{(k)_m} \frac{q_{m,m}}{d^{m/2}}+ \dxi_{m}(\xi_1, \xi_2)\Big]\\
&\qquad\qquad\qquad - a_{k - 1} \sum_{m = 0}^{k-1} \qt_{k -1 - m}(x) r_{1,2}^{k - 1 - m}\Big[\sqrt{(k - 1)_m} \frac{q_{m,m}}{d^{m/2}}+ \dxi_{m}(\xi_1, \xi_2)\Big],\label{eq:qk_e1}
\end{align}
where in reaching the last step we have used \eref{expansion} to expand $q_{k}(r_{1,2} x + \frac{\xi_1 \xi_2}{\sqrt{d}})$ and $q_{k-1}(r_{1,2} x + \frac{\xi_1 \xi_2}{\sqrt{d}})$.

On the right-hand side of \eref{qk_e1} there are factors related to $x$ in the form of $x \qt_{k - m}(x)$. They are polynomials of $x$, and can thus be rewritten as a linear combination of the orthogonal polynomials. To that end, we first recall from \eref{polynomial_s_d} that $\set{\widetilde{q}_k(x)}$ denote the orthogonal polynomials defined for dimension $d-1$. Similar to \eref{recurrent}, they also satisfy a recurrence relation
\begin{equation}\label{eq:recurrent_dt}
x\qt_{k}(x) = \widetilde{a}_k \qt_{k+1}(x) + \widetilde{a}_{k-1} \qt_{k-1}(x),
\end{equation}
where the coefficients are
\begin{equation}\label{eq:recurrent_at}
\widetilde{a}_k = \sqrt{\frac{(k+1)(d+k-3)(d-1)}{(d+2k-1)(d+2k-3)}} = \sqrt{k+1}\, (1 + \mathcal{O}(1/d)).
\end{equation}
Applying this formula, and by setting $\tilde{a}_{-1} = 0$ and $\qt_{-1}(x) = 0$, we can write
\begin{equation}\label{eq:rec_qt}
x \qt_{k - m}(x) = \tilde{a}_{k-m} \qt_{k+1 - m}(x) + \tilde{a}_{k-1 - m} \qt_{k - 1 - m}(x) \qquad \text{for } 0 \le m \le k.
\end{equation}
Replacing all the factors of $x \qt_{k - m}(x)$ in \eref{qk_e1} by the right-hand side of \eref{rec_qt}, we can rewrite \eref{qk_e1} as a linear combination of the orthogonal polynomials $\set{\qt_{k+1-m}(x)}_{m=0}^{k+1}$, \emph{i.e.},
\begin{equation}\label{eq:qk_e2}
\begin{aligned}
q_{k+1}(r_{1,2} x + \frac{\xi_1 \xi_2}{\sqrt{d}}) &=\sum_{m = 0}^{k} \qt_{k + 1 - m}(x) r_{1,2}^{k +1 - m}\frac{\widetilde{a}_{k-m}}{a_{k} }\Big[\sqrt{(k)_m} \frac{q_{m,m}}{d^{m/2}}+ \dxi_{m}(\xi_1, \xi_2)\Big]\\
&\qquad\qquad+\sum_{m = 0}^{k-1} \qt_{k - 1 - m}(x) r_{1,2}^{k +1 - m}\frac{\widetilde{a}_{k-1-m}}{a_{k}}\Big[\sqrt{(k)_m} \frac{q_{m,m}}{d^{m/2}}+ \dxi_{m}(\xi_1, \xi_2)\Big]\\
&\qquad\qquad+\sum_{m = 0}^{k} \qt_{k - m}(x) r_{1,2}^{k - m}\frac{\xi_1 \xi_2}{a_{k} \sqrt{d}}\Big[\sqrt{(k)_m} \frac{q_{m,m}}{d^{m/2}}+ \dxi_{m}(\xi_1, \xi_2)\Big]\\
&\qquad\qquad-\sum_{m = 0}^{k-1} \qt_{k -1 - m}(x) r_{1,2}^{k - 1 - m}\frac{a_{k-1}}{a_{k}}\Big[\sqrt{(k - 1)_m} \frac{q_{m,m}}{d^{m/2}}+ \dxi_{m}(\xi_1, \xi_2)\Big]\\
&= \sum_{m=0}^{k+1} \qt_{k+1-m}(x) C_m(\xi_1,\xi_2),
\end{aligned}
\end{equation}
where $\set{C_m(\xi_1, \xi_2)}_{m=0}^{k+1}$ are some functions that only depend on $\xi_1, \xi_2$ but not on $x$. Next, we identify the exact expressions for $C_m(\xi_1, \xi_2)$. 

By examining \eref{qk_e2}, we can see that
\begin{equation}\label{eq:C0}
C_0(\xi_1, \xi_2) = r_{1,2}^{k+1} \frac{\widetilde{a}_{k}}{a_{k}}[1+ \dxi_{0}(\xi_1, \xi_2)] = r_{1,2}^{k+1}[1 + \dxi_0(\xi_1, \xi_2)],
\end{equation}
where the last equality follows from \eref{recurrent_a} and \eref{recurrent_at}.
\begin{align}
C_1(\xi_1, \xi_2) &= r_{1,2}^k \frac{\widetilde{a}_{k-1}}{a_k}\Big[\sqrt{k}\frac{q_{1,1}}{\sqrt{d}} + \dxi_1(\xi_1, \xi_2)\Big] + r_{1,2}^k \frac{\xi_1 \xi_2}{a_k\sqrt{d}}\Big[1 + \dxi_0(\xi_1, \xi_2)\Big]\\
&= r_{1,2}^k \frac{\sqrt{k}}{\sqrt{k+1}}  \Big[\sqrt{k} \frac{q_{1,1}}{\sqrt{d}} + \dxi_1(\xi_1, \xi_2)\Big] + r_{1,2}^k \frac{1}{\sqrt{k+1}}\Big[\frac{q_{1,1}}{\sqrt{d}} + \dxi_1(\xi_1, \xi_2)\Big] \qquad \text{by } \eref{recurrent_a}, \eref{recurrent_at} \text{ and } \eref{dxi_p2}\\
&=r_{1,2}^k \Big[\sqrt{k+1} \frac{q_{1,1}}{\sqrt{d}} + \dxi_1(\xi_1, \xi_2)\Big].\label{eq:C1}
\end{align}
For each $m$ in the range $2 \le m \le k+1$, we have
\begin{equation}\label{eq:Cm1}
\begin{aligned}
&C_m(\xi_1, \xi_2) = r_{1,2}^{k +1 - m}\underbrace{\Big[\frac{\widetilde{a}_{k-m}}{a_{k} }\sqrt{(k)_m} \frac{q_{m,m}}{d^{m/2}}+ \dxi_{m}(\xi_1, \xi_2)\Big]}_{\mathcal{F}_1} \\
&\qquad+ r_{1,2}^{k+1-m}\underbrace{\Big[\frac{\widetilde{a}_{k+1-m}}{a_{k}}\sqrt{(k)_{m-2}} \big(r_{1,2}^2\frac{q_{m-2,m-2}}{d^{(m-2)/2}}\big)+ r_{1,2}^2\dxi_{m-2}(\xi_1, \xi_2)\Big]}_{\mathcal{F}_2}\\
&\qquad\qquad+r_{1,2}^{k+1-m}\underbrace{\Big[\frac{1}{a_{k}}\sqrt{(k)_{m-1}} \frac{\xi_1 \xi_2 q_{m-1,m-1}}{d^{m/2}}+ \frac{\xi_1 \xi_2}{\sqrt{d}}\dxi_{m-1}(\xi_1, \xi_2)\Big]}_{\mathcal{F}_3} \\
&\qquad- r_{1,2}^{k+1-m}\underbrace{\Big[\frac{a_{k-1}}{a_k} \sqrt{(k-1)_{m-2}} \frac{q_{m-2,m-2}}{d^{(m-2)/2}} + \dxi_{m-2}(\xi_1, \xi_2)\Big]}_{\mathcal{F}_4}.
\end{aligned}
\end{equation}
Using the formulas for $a_k$ and $\widetilde{a}_k$ in \eref{recurrent_a} and \eref{recurrent_at}, we have
\begin{equation}\label{eq:F1}
\mathcal{F}_1 = \sqrt{\frac{(k+1-m) (k)_m}{k+1}} \frac{q_{m,m}}{d^{m/2}} + \dxi_m(\xi_1, \xi_2) = \frac{k+1-m}{k+1} \sqrt{(k+1)_m} \frac{q_{m,m}}{d^{m/2}} + \dxi_m(\xi_1, \xi_2).
\end{equation}
Applying \eref{recurrent_sqrt} and \eref{dxi_p3} gives us
\begin{equation}\label{eq:F2}
\mathcal{F}_2 = \sqrt{\frac{(k)_{m-1}}{k+1}}\Big(\frac{q_{m-2,m-2}}{d^{(m-2)/2}} - \sqrt{m(m-1)} \frac{q_{m-2, m} + q_{m, m-2}}{d^{m/2}}\Big) + \dxi_m(\xi_1, \xi_2).
\end{equation}
By \eref{recurrent_xiq} and \eref{dxi_p2},
\begin{equation}\label{eq:F3}
\mathcal{F}_3 = \sqrt{\frac{(k)_{m-1}}{k+1}}\Big(m\frac{q_{m,m}}{d^{m/2}} + \sqrt{m(m-1)} \frac{q_{m-2, m} + q_{m, m-2}}{d^{m/2}}\Big) + \dxi_m(\xi_1, \xi_2).
\end{equation}
Similarly, one can verify that
\begin{equation}\label{eq:F4}
\mathcal{F}_4 = \sqrt{\frac{(k)_{m-1}}{k+1}}\frac{q_{m-2,m-2}}{d^{(m-2)/2}} + \dxi_m(\xi_1, \xi_2).
\end{equation}
Substituting \eref{F1}, \eref{F2}, \eref{F3}, \eref{F4} into \eref{Cm1}, we have
\begin{equation}\label{eq:Cm}
C_m(\xi_1, \xi_2) = r_{1,2}^{k +1 - m} \Big(\sqrt{(k+1)_m} \frac{q_{m,m}}{d^{m/2}} + \dxi_m(\xi_1, \xi_2)\Big).
\end{equation}
Note that the expressions in \eref{C0}, \eref{C1} and \eref{Cm} exactly match the right-hand side of \eref{expansion} for $k+1$. Thus, by induction, we can conclude that the formula \eref{expansion} holds for all $k$.

% !TEX root = equivalence.tex

\section{Review of the Marchenko-Pastur Law}
\label{appendix:MP_law}

In this appendix, we review some basic properties of the Marchenko-Pastur (MP) law that will be used in our work. Let $X$ be an $M\times n$ matrix whose entries $X_{\mu i}$  are independent complex-valued random variables satisfying
\begin{equation}\label{eq:cond_X}
\EE X_{\mu i} = 0\qquad \EE \abs{X_{\mu i}}^2= \frac{1}{n}.
\end{equation}
Additionally, $\abs{\sqrt{n} X_{\mu i}}$ has a sufficient number of bounded moments. We also assume that $M$ and $n$ satisfy the bounds
\begin{equation} \label{NM gen}
n^{1/C} \le M \le n^C
\end{equation}
for some positive constant $C$, and define the aspect ratio 
\[
 \rp \bydef \rp_n = \frac Mn,
\]
which may depend on $n$. Now consider an $n \times  n$ matrix 
\begin{equation}\label{eq:gMP}
R = X\tran \Sigma X - \frac{\tr \Sigma}n I,
\end{equation}
where $\Sigma$ is a diagonal $M \times M$ matrix with matrix elements 
$\sigma_i$ on the diagonal. Note that the subtraction by $\frac{\tr \Sigma}n I$ in \eref{gMP} makes sure that the diagonal elements of $R$ are approximately equal to 0. We do this to match the construction in \eref{A_f}.

Assuming that the law of the diagonal entries of $\Sigma$ is given by a probability law $\pi_M = M^{-1} \sum \delta_{\sigma_i}  \to \pi$. The MP law asserts that the limiting Stieltjes transform 
of the eigenvalues of $R$, denoted by $m(z)$, satisfies the self-consistent equation
\begin{equation}\label{eq:m_gMP}
\frac 1 m = -z - \rp \int  \frac{x^2 m} {1 + x m } \pi(\dif x).
\end{equation}
See, \emph{e.g.},  eq. (2.10) in \citet{knowles2016AnisotropicLocalLaws}. (Notice  that  our equation is slightly different due to a constant shift of the eigenvalues). 
 In the special case of $\Sigma = t I$ and $R = t X\tran X - t \rp I$, the above formula reduces to
 \begin{equation}\label{eq:m_MP}
\frac 1 m = -z -  \rp\frac { t^2 m}{1 + t m}.
\end{equation}
The limiting eigenvalue density associated with \eref{m_MP} is given by 
\begin{equation}\label{eq:rho_MP}
\varrho_{\rp,t}  (x)  \dif x \bydef  \frac{\sqrt{\big[(2\sqrt{\rp} +x/t - 1)(2\sqrt{\rp} - x/t + 1)\big]_+}}{2\pi \sign(t)[x +t \rp]}\, \dif x + (1-\rp)_+ \, \delta( x+ t\rp) \dif x.
\end{equation}

As a side note, we remark  that a different convention was used in \citet{bloemendalPrincipalComponentsSample2014}, where $Y$ is an $M\times n$ matrix whose entries $Y_{i \mu}$  are independent  random variables satisfying
%\begin{equation} \label{cond on entries of X}
$\EE Y_{i \mu}\;=\;0\,,\; \EE Y_{i \mu}^2\;=\;\frac{1}{\sqrt {n M}}$. 
%\end{equation}
%The density of $Y\tran Y$ is given by 
%\begin{equation} %\label{S_MP}
%\rho(x)  dx  \;=\; \frac{  \sqrt \phi   \sqrt{ [(x-\gamma_-)(\gamma_+-x)]_+}    } {2\pi x}+  (1-\phi )_+ \, \delta(  x  ) \dif x, \quad 
%\gamma_\pm = \phi^{ 1/2}+\phi^{-1/2} \pm 2
%\end{equation}
One can check easily that $\varrho_{tH+ r} (x)  \dif x = \varrho_H \big (\frac { x-r} t \big ) \frac { \dif x}{\abs{t}}$ where $\varrho_H$ denotes the density of a symmetric matrix $H$. 
Hence  $R=  t  [\sqrt \rp  Y\tran Y-   \rp   I]$ and we have  $\varrho_R (x) \dif x= { \varrho_{Y\tran Y}(\frac { x } { t  \sqrt \rp}  + \sqrt \rp ) }/({ \,\abs{t}\sqrt \rp}) \dif x$. 
Thus our formula is identical to the one that appeared in \cite[eq. (3.3)]{bloemendalPrincipalComponentsSample2014}.

\section{Moment Bounds and Concentration Inequalities}
\label{appendix:concentration}

In this appendix, we derive several moment and concentration inequalities that will be used in our proof.

%\subsection{Moment Bounds}
\begin{lemma}
Let $\xi \sim \taumeasure$, and $q_k(x)$ be the $k$th Gegenbauer polynomial for $k \in \N_0$. For each $p \in \N$, we have
\begin{equation}\label{eq:xi_normp}
\normp{\xi}{p} = (\EE\, \abs{\xi}^{p})^{1/p} \le 2\sqrt{p}
\end{equation}
and
\begin{equation}\label{eq:qk_normp}
\normp{q_k(\xi)}{p} \le \qlp{k}{p}
\end{equation}
where $\qlp{k}{p}$ is some constant that only depends on $k$ and $p$.
\end{lemma}
\begin{remark}
Both the probability measure $\taumeasure$ and the Gegenbauer polynomials $q_k(x)$ depend on the dimension $d$. However, the upper bounds in \eref{xi_normp} and \eref{qk_normp} hold uniformly for all $d$.
\end{remark}
\begin{proof}
The probability distribution of $\xi$ is given by
\begin{equation}\label{eq:xi_pdf}
w_d(\xi) = \begin{cases}\frac{\Gamma(d/2)}{\sqrt{d \pi} \, \Gamma((d-1)/2)} (1-\xi^2/d)^{\frac{d-3}{2}} &\text{if } \abs{\xi} \le \sqrt{d}\\
0 &\text{otherwise},
\end{cases}
\end{equation}
where $\Gamma(\cdot)$ is the gamma function. Thus, we have
\begin{align}
\EE\, \abs{\xi}^p &= \frac{2\,\Gamma(d/2)}{\sqrt{d \pi} \, \Gamma((d-1)/2)} \int_0^{\sqrt{d}} \xi^p (1-\xi^2/d)^{\frac{d-3}{2}} \,\mathrm{d}\xi\\
&= \frac{d^{p/2}\,\Gamma(d/2)\,\Gamma((p+1)/2)}{\Gamma((d+p)/2)\,\Gamma(1/2)}\label{eq:xi_moments}\\
&\le (4p)^{p/2},
\end{align}
where the second line uses the relationship between the beta function and the gamma function \cite{artin2015GammaFunction}, and the last line is obtained by applying Stirling's formula for the gamma function \cite{jameson2015SimpleProofStirling}.

Next, we show \eref{qk_normp}. Two special cases are easy: for $k = 0$, we have $\normp{q_0(\xi)}{p} = 1$; for $p = 1$, $\normp{q_k(\xi)}{1} \le \normp{q_k(\xi)}{2} = 1$. Thus, in what follows we assume $k \ge 1$ and $p \ge 2$. Denote by $c_k$ the leading coefficient of the polynomial $q_k(\xi)$. Since $q_k(\xi) - c_k \xi^k$ is a polynomial of degree $k-1$, it can be written as linear combination of the lower order Gegenbauer polynomials, \emph{i.e.},
\[
q_k(\xi) = c_k \xi^k + \sum_{i=0}^{k-1} b_{k, i} q_i(\xi),
\]
where $\set{b_{k,i}}_i$ are the expansion coefficients. By using the orthogonality of the Gegenbauer polynomials (see \eref{ortho_poly}), we have, for $0 \le i \le k-1$,
\[
0 = \EE [q_k(\xi) q_i(\xi)] = c_k \EE [\xi^k q_i(\xi)] + b_{k, i},
\]
and thus
\[
q_k(\xi) = c_k \xi^k - c_k \sum_{i=0}^{k-1} \EE [\xi^k q_i(\xi)] q_i(\xi).
\]
From the triangular inequality,
\begin{equation}\label{eq:qk_normp1}
\normp{q_k(\xi)}{p} \le c_k \normp{\xi^k}{p} + c_k \sum_{i=0}^{k-1} \abs{\EE [\xi^k q_i(\xi)]} \cdot \normp{q_i(\xi)}{p}.
\end{equation}
Applying the Cauchy-Schwarz inequality, we have
\begin{equation}\label{eq:qk_normp2}
\abs{\EE [\xi^k q_i(\xi)]} \le [\EE\, \xi^{2k} \, \EE\, q_i^2(\xi)]^{1/2} \le  (8k)^{k/2},
\end{equation}
where the last inequality is due to \eref{xi_normp} and the fact that $\EE q_i^2(\xi) = 1$. In addition, $\normp{\xi^k}{p} \le (4pk)^{k/2}$ by \eref{xi_normp}. Substituting these two bounds into \eref{qk_normp1} then gives us, for $p \ge 2$,
\begin{equation}\label{eq:qk_normp3}
\normp{q_k(\xi)}{p} \le c_k (4pk)^{k/2}\Big(1 + \sum_{i=0}^{k-1} \normp{q_i(\xi)}{p}\Big).
\end{equation}
We now provide a bound for $c_k$, the leading coefficient of $q_k(\xi)$. Recall from \eref{Gegenbauer_low_degree} that $c_1 = 1$ and $c_2 = (1/\sqrt{2}) \sqrt{(d+2)/(d-1)}$. A general formula for $c_k$ can be obtained by examining the three-term recurrence relation in \eref{recurrent}, which gives us $c_{k+1} = c_k/a_k$ for every $k \in \N$. This implies that $c_k = (a_0 a_1 a_2 \ldots a_{k-1})^{-1}$, and from the explicit formula for $a_k$ in \eref{recurrent_a},
\begin{equation}
c_k = \prod_{i=0}^{k-1} \Big(\frac{1+(2/d)i}{1+i}\Big)^{1/2}\Big(1+\frac{i}{i+d-2}\Big)^{1/2}\le 2^{k/2}
\end{equation}
for all $d \ge 2$. Using this estimate in \eref{qk_normp3} leads to
\[
\normp{q_k(\xi)}{p} \le (8pk)^{k/2}\Big(1 + \sum_{i=0}^{k-1} \normp{q_i(\xi)}{p}\Big), \qquad \text{for all } k \in \N.
\]
Starting from $\normp{q_0(\xi)}{p} = 1$, we can apply the above bound recursively to verify \eref{qk_normp}.                                                    
\end{proof}

In the discussion of the equivalence principle for general nonlinear kernel functions in \sref{general_f} and \sref{gaussian}, we examine three closely-related probability distributions: the standard normal distribution $\mathcal{N}(0, 1)$, the probability measure $\taumeasure$ as defined in \eref{tau_measure}, and the distribution of the random variable 
\begin{equation}\label{eq:bxi}
\breve{\xi}_d \bydef \sqrt{d} \vx\tran \vy,
\end{equation}
where $\vx, \vy \in \R^d$ are two independent random vectors sampled from $\mathcal{N}(0, I_d/d)$. It can be readily verified that, as $d \to \infty$, the latter two distributions converge to $\mathcal{N}(0, 1)$.

\begin{lemma}\label{lemma:w_wd_wdb}
Let $f: \R \mapsto \R$ be a function. If the following conditions hold: (a) $\int f^2(x) w(x) \dif x < \infty$, and (b) there exist positive constants $c_1, c_2, c_3$ such that $\abs{f(x)} < c_1 e^{c_2 \abs{x}}$ when $\abs{x} \ge c_3$, then
\begin{equation}\label{eq:f2_w_wd}
\int_{\R} f^2(x) \abs{w(x) - w_d(x)} \dif x \xrightarrow{d\to\infty} 0
\end{equation}
and
\begin{equation}\label{eq:f2_w_wdb}
\int_{\R} f^2(x) \abs{w(x) - \breve{w}_d(x)} \dif x \xrightarrow{d\to\infty} 0.
\end{equation}
Here, $w(x)$ denotes the probability density function of $\mathcal{N}(0, 1)$, $w_d(x)$ denotes the density function of the probability measure $\taumeasure$ as defined in \eref{tau_measure}, and $\breve{w}_d(x)$ represents the probability density function of the random variable $\breve{\xi}_d$ defined in \eref{bxi}.
\end{lemma}
\begin{proof}
We begin with \eref{f2_w_wd}. Let $M$ be a constant in the interval $[c_3, \sqrt{d/2}]$. We split the integration in \eref{f2_w_wd} into two parts:
\begin{align}
&\int_{\R} f^2(x) \abs{w(x) - w_d(x)} \dif x = \Big(\int_{\abs{x} \le M} + \int_{\abs{x} \ge M}\Big) f^2(x) \abs{w(x) - w_d(x)} \dif x\\
&\qquad\qquad\le \sup_{\abs{x} \le M} \abs{1 - w_d(x)/w(x)} \cdot \int_{\abs{x} \le M} f^2(x) w(x) \dif x + 2 c_1^2 \int_{x \ge M} e^{2 c_2 x} [w(x) + w_d(x)] \dif x.\label{eq:f2_w_wd1}
\end{align}
In reaching the second step, we have used the condition that $\abs{f(x)} < c_1 e^{c_2 \abs{x}}$ for $x \ge M$, and we have also used the property that the densities $w(x)$ and $w_d(x)$ are both even functions.

A closed-form expression of $w_d(x)$ can be found in \eref{xi_pdf}. By applying Stirling's formula for the gamma function \cite{jameson2015SimpleProofStirling} and Taylor's expansion for $\log(1- x^2/d)$, it is straightforward to verify that $w_d(x)/w(x) = [1 + \mathcal{O}(1/d)] e^{3 x^2 /(2d)}$ over the interval $\abs{x} \le M$, where $M$ is some constant satisfying $M < \sqrt{d/2}$. Consequently, we have
\begin{equation}\label{eq:w_wd_ratio}
\sup_{\abs{x} \le M} \abs{1 - w_d(x)/w(x)} \xrightarrow{d\to\infty} 0.
\end{equation}
Once more, using the closed-form expression in \eref{xi_pdf}, we can directly confirm that $w_d(x) \le [e^{3/2} + \mathcal{O}(1/d)] w(x)$ for all $x \in \R$. Therefore, the second term on the right-hand side of \eref{f2_w_wd1} can be bounded as follows:
\begin{equation}\label{eq:w_wd_tail}
\int_{x \ge M} e^{2 c_2 x} [w(x) + w_d(x)] \dif x \le \mathcal{O}(1) \int_{x \ge M} e^{2c_2 x} w(x) \dif x \le \mathcal{O}(1) e^{-M^2/4},
\end{equation}
where the final inequality follows from the Cauchy-Schwartz inequality and standard Gaussian tail bounds. By assumption, $\int f^2(x) w(x) \dif x < \infty$. Upon substituting \eref{w_wd_ratio} and \eref{w_wd_tail} into \eref{f2_w_wd1}, we have 
\begin{equation}
\limsup_{d \to \infty} \int f^2(x) \abs{w(x) - w_d(x)} \dif x \le \mathcal{O}(1) e^{-M^2/4}.
\end{equation}
Since the constant $M$ can be chosen arbitrarily, we have then demonstrated \eref{f2_w_wd}.

Next, we verify the expression in \eref{f2_w_wdb}. Let $\varphi(t) \bydef \EE[e^{\mathrm{i} t \breve{\xi}_d}]$ represent the characteristic function of the random variable $\breve{\xi}_d$ defined in \eref{bxi}. It is straightforward to verify that $\varphi(t) = (\EE[e^{\mathrm{i} t \sqrt{d} x_1 y_1}])^d = (1 + t^2/d)^{-d/2}$. By applying the inversion formula for characteristic functions, we obtain
\begin{align}
&\abs{w(x) - \breve{w}_d(x)} \le \frac{1}{2\pi} \abs{\int_{\R} e^{-\mathrm{i} tx} (e^{-t^2/2} - (1 + t^2/d)^{-d/2}) \dif t}\\
&\quad\le  \frac{1}{2\pi} \int_{\R} \abs{e^{-t^2/2} - (1 + t^2/d)^{-d/2}} \dif t\\
&\quad\le  \frac{1}{2\pi} \int_{\abs{t} \le \sqrt{4\log d}} \abs{(1 + t^2/d)^{-d/2} - e^{-t^2/2}} \dif t+ \frac{1}{\pi}\Big(\int_{\sqrt{4\log d} \le t \le \sqrt{d}} + \int_{t \ge \sqrt{d}}\Big) (1 + t^2/d)^{-d/2} \dif t.\label{eq:w_wdb1}
\end{align}
In reaching the final step, we have used the inequality $\log(1+t^2/d) \le t^2/d$, which then implies that $(1+t^2/d)^{-d/2} \ge e^{-t^2/2}$. It is straightforward to verify via Taylor's expansion that 
\begin{equation}\label{eq:w_wdb_i1}
\sup_{\abs{t} \le \sqrt{4\log d}} \abs{e^{t^2/2}(1 + t^2/d)^{-d/2} - 1}  = \mathcal{O}\Big(\frac{\log^2 d}{d}\Big).
\end{equation}
We then have:
\begin{equation}
\int_{\abs{t} \le \sqrt{4\log d}} \abs{(1 + t^2/d)^{-d/2} - e^{-t^2/2}} \dif t = \mathcal{O}\Big(\frac{\log^2 d}{d}\Big).
\end{equation}
To bound the second integral on the right-hand side of \eref{w_wdb1}, we use the inequality $\log(1 + t^2/d) \ge (\log2) t^2/d$ for $0 \le t \le \sqrt{d}$. It follows that
\begin{equation}\label{eq:w_wdb_i2}
\int_{\sqrt{4\log d} \le t \le \sqrt{d}} (1 + t^2/d)^{-d/2} \dif t \le \int_{\sqrt{4\log d} \le t \le \sqrt{d}} e^{-(\log2)  t^2/2} \dif t = \mathcal{O}(d^{-1}),
\end{equation}
where the final step uses standard Gaussian tail bounds. For the third integral on the right-hand side of \eref{w_wdb1}, we obtain
\begin{equation}\label{eq:w_wdb_i3}
\int_{t \ge \sqrt{d}} (1 + t^2/d)^{-d/2} \dif t = \sqrt{d} \int_1^\infty (1+t^2)^{-d/2} \dif t \le \sqrt{d} \int_1^\infty t^{-d} \dif t = \mathcal{O}(d^{-1/2}).
\end{equation}
By combining the bounds \eref{w_wdb_i1}, \eref{w_wdb_i2}, and \eref{w_wdb_i3} for the three integrals, we can establish that
\begin{equation}\label{eq:w_wdb}
\sup_{x \in \R} \, \abs{w(x) - \breve{w}_d(x)} = \mathcal{O}(d^{-1/2}).
\end{equation}
Let $M_d$ be a positive number that depends on $d$. We divide the integral in \eref{f2_w_wdb} into two parts, following a similar approach as in \eref{f2_w_wd1}. This results in:
\begin{align}
\int_{-\infty}^\infty f^2(x) \abs{w(x) - \breve{w}_d(x)} \dif x &\le \sup_{\abs{x} \le M_d} \abs{1 - \breve{w}_d(x)/w(x)} \cdot \int_{\abs{x} \le M_d} f^2(x) w(x) \dif x\label{eq:w_wdb_1}\\
&\qquad\qquad+ 2 c_1^2 \int_{x \ge M_d} e^{2 c_2 x} w(x) \dif x + 2 c_1^2 \int_{x \ge M_d} e^{2 c_2 x} \breve{w}_d(x) \dif x\\
&\le \mathcal{O}(d^{-1/2})e^{M_d^2/2}  + \mathcal{O}(1) e^{-M_d^2/4} + 2 c_1^2 \int_{x \ge M_d} e^{2 c_2 x} \breve{w}_d(x) \dif x,\label{eq:w_wdb_2}
\end{align}
where in obtaining the last inequality, we have utilized \eref{w_wdb}, the assumption that $\int f^2(x) w(x) \dif x < \infty$, and the estimate provided in \eref{w_wd_tail} to bound the first two terms on the right-hand side of \eref{w_wdb_1}. By applying the Cauchy-Schwarz inequality, we have:
\begin{equation}\label{eq:w_wdb_3}
\int_{x \ge M_d} e^{2 c_2 x} \breve{w}_d(x) \dif x \le \big[\EE e^{4c_2 \breve{\xi}_d} \cdot \P(\breve{\xi}_d \ge M_d)\big]^{1/2} \le \big[\EE e^{4c_2 \breve{\xi}_d}\cdot \EE \breve{\xi}_d^2\big]^{1/2}/M_d = \mathcal{O}(M_d^{-1}),
\end{equation}
with the final step due to the identities $\EE e^{4c_2 \breve{\xi}_d} = (1 - 16c_2^2/d)^{-d/2}$ and $\EE \breve{\xi}_d^2 = 1$. The desired result in \eref{f2_w_wdb} then follows by substituting \eref{w_wdb_3} into \eref{w_wdb_2}, and setting $M_d = \sqrt{\log d}/2$.
\end{proof}

%By symmetry, all the odd moments of $\xi$ are equal to 0. To compute the even moments, we can use the definition of the beta function and get
%\begin{equation}\label{eq:xi_moment}
%\EE \xi^{2k} = \frac{d^k\Gamma(k+1/2)\Gamma(d/2)}{\Gamma(k+d/2)\sqrt{\pi}} = \frac{(2k-1)!!}{\prod_{0 \le j < k} (1+2j / d)} =  (2k-1)!!(1+\mathcal{O}(d^{-1})).
%\end{equation}

In the following discussion, we derive several useful moment and high probability bounds for linear and quadratic functions of independent random variables. We first recall the following estimates, whose proof can be found in \cite[Lemma 7.8, Lemma 7.9]{erdos2017Dynamicalapproach}.
\begin{lemma}\label{lemma:linear_quadratic_lp}
Let $(X_i)_{i \in [n]}$ and $(Y_j)_{j \in [m]}$ be independent families of random variables, and $(b_i)_{i \in [n]}$ and $(c_{ij})_{i \in [n], j \in [m]}$ be deterministic complex-valued coefficients; here $n, m \in \N$. Suppose that all entries of $(X_i)$ and $(Y_j)$ are independent and satisfy
\begin{equation}
\EE X = \EE Y = 0, \qquad \EE \abs{X}^2 = \EE \abs{Y}^2 = 1, \qquad \normp{X}{p} \le \kappa_{p},\qquad \normp{Y}{p} \le \gamma_p,
\end{equation}
for all $p \in \N$ and some constants $\kappa_p, \gamma_p$. Then, we have
\begin{subequations}
\begin{equation}\label{eq:linear_lp}
\normp{\textstyle\sum_i b_i X_i}{p} \le (C p)^{1/2} \kappa_p \Big(\sum_i \abs{b_i}^2\Big)^{1/2}
\end{equation}
and
\begin{equation}\label{eq:quadratic_lp}
\normp{\textstyle\sum_{i,j} c_{ij} X_i Y_j}{p} \le C p \kappa_p \gamma_p \Big(\sum_{i,j} \abs{c_{ij}}^2\Big)^{1/2},
\end{equation}
\end{subequations}
where $C$ is an absolute constant.
\end{lemma}

\begin{proposition}\label{prop:quadratic_qab}
Let $(c_{ij})_{i, j \in [n]}$ be a deterministic matrix with complex-valued coefficients, and $(\xi_i)_{i \in [n]}$ be an independent family of random variables with $\xi_i  \sim \taumeasure$. For every $a, b \in \N$, we have
\begin{equation}\label{eq:quadratic_qab}
\normp{\textstyle \sum_{i, j} c_{ij} q_a(\xi_i) q_b(\xi_j) - (\sum_i c_{ii})\delta_{ab}}{p} \le C(a, b, p) \big(\textstyle\sum_{i,j} \abs{c_{ij}}^2\big)^{1/2},
\end{equation}
where $C(a, b, p)$ is a quantity that depends on $a$, $b$, and $p$, but not on $n$.
\end{proposition}
\begin{proof}
By the triangular inequality, we have
\begin{equation}\label{eq:quadratic_qab_triangle}
\normp{\textstyle \sum_{i, j} c_{ij} q_a(\xi_i) q_b(\xi_j) - (\sum_i c_{ii})\delta_{ab}}{p} \le \normp{\textstyle \sum_{i\neq j} c_{ij} q_a(\xi_i) q_b(\xi_j)}{p} + \normp{\textstyle \sum_{i} c_{ii} \big(q_a(\xi_i) q_b(\xi_i) - \delta_{ab}\big)}{p}.
\end{equation}
To bound the first term on the right-hand side, we use the standard decoupling technique. Observe that, for every $i \neq j$, the following identity holds:
\begin{equation}
1 = \frac{1}{Z_n} \sum_{I \subset [n]} \charfn(i \in I) \charfn(j \in I^c),
\end{equation}
where the sum ranges over all subsets of $[n]$, and $Z_n = 2^{n-2}$. Applying this identity and the triangular inequality allows us to write
\begin{equation}\label{eq:quadratic_qab1}
\normp{\textstyle \sum_{i\neq j} c_{ij} q_a(\xi_i) q_b(\xi_j)}{p} \le \frac{1}{Z_n} \sum_{I \subset [n]} \normp{\textstyle \sum_{i \in I, j \in I^c} c_{ij} q_a(\xi_i) q_b(\xi_j)}{p}.
\end{equation}
Note that the families of random variables $\set{q_a(\xi_i)}_{i \in I}$ and $\set{q_b(\xi_j)}_{j \in I^c}$ have mean-zero independent components with unit variances; they are also mutually independent. We can then apply \eref{quadratic_lp} to bound each term in the sum in \eref{quadratic_qab1}. Moreover, the sum involves a total of $2^n-2$ terms. Thus, 
\begin{equation}\label{eq:quadratic_qab2}
\normp{\textstyle \sum_{i\neq j} c_{ij} q_a(\xi_i) q_b(\xi_j)}{p} \le 4Cp \normp{q_a(\xi)}{p}\normp{q_b(\xi)}{p}\Big(\sum_{i\neq j} \abs{c_{ij}}^2\Big)^{1/2}.
\end{equation}

Now we bound the second term on the right-hand side of \eref{quadratic_qab_triangle}. Write $X_i = q_a(\xi_i) q_b(\xi_i) - \delta_{ab}$ and $\sigma = \normp{X_i}{2}$. Then $(X_i / \sigma)_i$ is a family of independent random variables that satisfy the condition of Lemma~\ref{lemma:linear_quadratic_lp}. Using \eref{linear_lp} gives us
\begin{align}
\normp{\textstyle \sum_{i} c_{ii} \big(q_a(\xi_i) q_b(\xi_i) - \delta_{ab}\big)}{p} &= \sigma \normp{\textstyle \sum_{i} c_{ii} (X_i/\sigma)}{p}\\
&\le (CP)^{1/2} \normp{X_i}{p} \big(\textstyle\sum_{i} \abs{c_{ii}}^2\big)^{1/2}.
\end{align}
By the triangular inequality (in the first line) and the Cauchy-Schwarz inequality (in the third line),
\begin{align}
\normp{X_i}{p} &\le \normp{q_a(\xi_i) q_b(\xi_i)}{p} + 1\\
&= (\EE \abs{q_a(\xi_i)}^p \abs{q_b(\xi_i)}^p)^{1/p} + 1\\
&\le \normp{q_a(\xi_i)}{2p} \normp{q_b(\xi_i)}{2p} + 1.\label{eq:quadratic_qab3}
\end{align}
Substituting \eref{quadratic_qab2} and \eref{quadratic_qab3} into \eref{quadratic_qab_triangle}, and by applying the moment bound in \eref{qk_normp}, we complete the proof of the statement in \eref{quadratic_qab}.
\end{proof}

The moment estimates obtained above can be easily turned into high probability bounds, as follows.
\begin{proposition}\label{eq:pointwise_hpb}
Suppose $n = \mathcal{O}(d^\ell)$ for some fixed $\ell \in \N$. Let $(\xi_i)_{i \in [n]}$ be an independent family of random variables with $\xi_i  \sim \taumeasure$, and let $A_k$ be the matrix defined in \eref{Ak}. For each fixed $k \in \N_0$, we have the bounds
\begin{align}
\max_{i \in [n]}\, \abs{q_k(\xi_i)} &\prec 1,\label{eq:qk_hpb}\\
\norm{A_k}_\infty &\prec \frac{1}{\sqrt{n}},\label{eq:Ak_inf_hpb}\\
\max_{i \in [n]}\, \abs{r^k(\xi_i) - 1} &\prec \frac{1}{d},\label{eq:rxi_hpb}
\end{align}
where $r(\xi)$ is the function defined in \eref{rxi}.
\end{proposition}
\begin{proof}
For any $\epsilon > 0$ and $p \in \N$,
\begin{align}
\P(\textstyle \max_i \, \abs{q_k(\xi_i)} \ge d^\epsilon) &\le n \, \P(\,\abs{q_k(\xi)} \ge d^\epsilon)\\
&\le C d^\ell \, \P(\, \abs{q_k(\xi)}^p \ge d^{\epsilon p})\\
&\le C d^{\ell - \epsilon p} \qlp{k}{p},
\end{align}
where the last step is by the Markov inequality and by the moment bound in \eref{qk_normp}. For any $D > 0$, we can always find $p = p(\epsilon, D)$ such that $C d^{\ell - \epsilon p} \qlp{k}{p} < d^{-D}$ for all sufficiently large $d$. Since $\epsilon$ and $D$ were arbitrary, we have verified \eref{qk_hpb}. To show \eref{Ak_inf_hpb}, observe that the law of $(A_k)_{ij}$, for any $i \neq j$, is equal to that of $q_k(\xi)/\sqrt{n}$ with $\xi \sim \taumeasure$. By \eref{qk_hpb}, we have $\abs{(A_k)_{ij}} \prec 1/\sqrt{n}$, uniformly over $\set{i, j \in [n]: i \neq j}$. Applying the union bound over $\mathcal{O}(d^{2\ell})$ entries of $A_k$ then yields \eref{Ak_inf_hpb}.

Now we prove \eref{rxi_hpb}. The case of $k = 0$ is trivial, so we assume $k \ge 1$. From the definition in \eref{rxi}, $r(\xi) = (1-\xi^2/d)^{1/2} + \mathcal{O}(1/d)$, and thus
\begin{align}
\abs{r(\xi) - 1} &\le \abs{1 - (1- \xi^2/d)^{1/2}} + \mathcal{O}(1/d)\\
&\le \xi^2/d + \mathcal{O}(1/d).
\end{align}
Combining this deterministic bound with the high-probability bound in \eref{qk_hpb}, we have that, for $\xi_i \sim \taumeasure$, 
\begin{equation}\label{eq:rxi_1}
r(\xi_i) - 1 = \mathcal{O}_\prec(1/d).
\end{equation}
For each $k \in \N$,
\[
\abs{r^k(\xi_i) - 1} \le \sum_{t=1}^k {k\choose{t}} \abs{r(\xi_i) - 1}^t.
\]
Applying \eref{rxi_1} and the union bound then allows us to conclude \eref{rxi_hpb}.
\end{proof}

\begin{proposition}\label{prop:rxi_f1f2}
Suppose that $n = \mathcal{O}(d^\ell)$ for some fixed $\ell \in \N$. Let $(\xi_i)_{i \in [n]}$ be an independent family of random variables with $\xi_i  \sim \taumeasure$. For every $k, t \in N_0$, and every function $\dxi_t(\xi_1, \xi_2)$ from the function class defined in \eref{dxi}, we can write
\begin{equation}\label{eq:rxi_f1f2}
r^k(\xi_1) r^k(\xi_2) \dxi_t(\xi_1, \xi_2) = f_1(\xi_1, \xi_2) + f_2(\xi_1, \xi_2),
\end{equation}
where the two functions $f_1, f_2$ are such that
\begin{equation}\label{eq:rxi_f1}
f_1(\xi_1, \xi_2) \in \dxi_{t + 2\ell}
\end{equation}
and
\begin{equation}\label{eq:rxi_f2}
\max_{a,b \in [n]} \abs{f_2(\xi_a, \xi_b)} \prec \frac{1}{d \sqrt{n}}.
\end{equation}
\end{proposition}
%\hty{ This one is basically trivial because $r= (1-\xi^2/d)^{-1/2}$ and $d$ is large. But I don't understand why you need to expand so such a high order. } 
\begin{proof}
The special case when $k = 0$ is trivial, as we can simply choose $f_1(\xi_1, \xi_2) = \dxi_t(\xi_1, \xi_2)$ and $f_2(\xi_1, \xi_2) = 0$. In what follows we assume $k \ge 1$. For any $\nu \in \N$, consider the $\nu$-th order Taylor expansion of the function $(1-x)^{k/2}$ around $x = 0$. This allows us to write
\begin{equation}\label{eq:Taylor}
(1-x)^{k/2} = P_\nu(x) + R_\nu(x),
\end{equation}
where
\begin{equation}\label{eq:Taylor_P}
P_\nu(x) = \sum_{i = 0}^{\nu - 1} \frac{(-1)^i (k/2)_i}{i!} x^i,
\end{equation}
and $R_\nu(x)$ is the remainder term. (In \eref{Taylor_P}, $(k/2)_i = (k/2)(k/2-1) \ldots (k/2-i+1)$ denotes the falling factorial.) For any $0 \le x \le 1$, we can verify from the Lagrange form of the the remainder that
\begin{equation}\label{eq:Taylor_R}
\abs{R_\nu(x)} \le (k/2)_\nu \max\set{(1-x)^{ (k/2)_{\nu+1}}, 1} x^\nu.
\end{equation}
Now recall the definition of $r(\xi)$ in \eref{rxi}. By using \eref{Taylor}, we can indeed decompose $r^k(\xi_1) r^k(\xi_2) \dxi_t(\xi_1, \xi_2)$ into the form of \eref{rxi_f1f2}, where
\[
f_1(\xi_1, \xi_2) = (1-1/d)^{-k/2} P_\nu(\xi_1^2/d) P_\nu(\xi_2^2 / d) \dxi_t(\xi_1, \xi_2)
\]
and
\[
f_2(\xi_1, \xi_2) = (1-1/d)^{-k/2} \big[R_\nu(\xi_1^2/d) P_\nu(\xi_2^2/d) + P_\nu(\xi_1^2/d) R_\nu(\xi_2^2/d) + R_\nu(\xi_1^2/d) R_\nu(\xi_2^2/d)\big] \dxi_t(\xi_1, \xi_2).
\]
Note that $(1-1/d)^{-k/2} = 1 + \mathcal{O}(1/d)$. For an independent family of random variables $\set{\xi_a}_{a \in [n]}$ with $\xi \sim \taumeasure$, we can use \eref{qk_hpb} to verify that $\max_a \abs{P_\nu(\xi_a^2/d)} \prec 1$. Similarly, by the definition in \eref{dxi} and \eref{qk_hpb}, we have $\max_{a,b} \abs{\dxi_t(\xi_a, \xi_b)} \prec 1/d$. It follows from these estimates that
\begin{align}\label{eq:f2_bnd}
\max_{a,b} \, \abs{f_2(\xi_a, \xi_b)} \prec \frac{1}{d}(1 + \max_a \vert{R_\nu(\xi_a^2/d)}\vert) \max_a \vert{R_\nu(\xi_a^2/d)}\vert.
\end{align}
Next, we establish a high-probability upper bound for $\max_a \vert{R_\nu(\xi_a^2/d)}\vert$. For any $a \in [n]$, $\epsilon \in (0, 1)$, and $D > 0$, and for sufficiently large $d$, we have
\begin{align}
\P(\vert{R_\nu(\xi_a^2/d)}\vert \ge d^{\epsilon - \nu}) &\le \P((1-\xi_a^2/d)^{(k/2)_{\nu+1}} \xi_a^{2\nu} \ge d^{\epsilon/2}) + \P(\xi_a^{2\nu} \ge d^{\epsilon/2})\\
&\le \P((1-\xi_a^2/d)^{(k/2)_{\nu+1}} \xi_a^{2\nu} \ge d^{\epsilon/2}, \xi_a^2 < d/2) + \P(\xi_a^{2} \ge d/2) + \P(\xi_a^{2\nu} \ge d^{\epsilon/2})\\
&\le \P(\xi_a^{2\nu} \ge \min\{2^{(k/2)_{\nu+1}}, 1\} d^{\epsilon/2}) + \P(\xi_a^{2} \ge d/2) + \P(\xi_a^{2\nu} \ge d^{\epsilon/2})\\
&\le d^{-D}.
\end{align}
In the first step, we use the deterministic bound in \eref{Taylor_R} and the fact that $(k/2)_\nu < d^{\epsilon/2}$ for all sufficiently large $d$. The last step follows from the property that $\xi_a = \mathcal{O}_\prec(1)$. Since the above inequality is uniform in $a \in [n]$, we have $\max_a \vert{R_\nu(\xi_a^2/d)}\vert \prec d^{-\nu}$. Substituting this bound into \eref{f2_bnd} and choosing $\nu = \lceil \ell/2 \rceil$,   we then reach the statement in \eref{rxi_f2}.

To show \eref{rxi_f1}, we use the definition of the Taylor polynomial in \eref{Taylor_P}, which gives us
\[
f_1(\xi_1, \xi_2) = (1-1/d)^{-k/2} \sum_{0 \le i, j \le \nu-1} c_i c_j (\xi_1^2/d)^i (\xi_2^2/d)^j \dxi_t(\xi_1, \xi_2),
\]
where $\set{c_i}$ are some fixed coefficients. The statement in \eref{rxi_f1} then follows from a repeated application of the property in \eref{dxi_p3a}.
\end{proof}

\begin{proposition}\label{prop:quadratic_hpb}
Let $(b_i)_{i \in [n]}$, $(c_{ij})_{i, j \in [n]}$ and $(\xi_i)_{i \in [n]}$ be three independent families of random variables. Moreover, $(\xi_i)$ is i.i.d., with $\xi_i  \sim \taumeasure$. Suppose that
\begin{equation}\label{eq:dominance_psi}
\big(\textstyle\sum_{i} \abs{b_i}^2\big)^{1/2} \prec \Psi_b \quad\text{and}\quad\big(\textstyle\sum_{i,j} \abs{c_{ij}}^2\big)^{1/2} \prec \Psi_c
\end{equation}
for some random variables $\Psi_b$ and $\Psi_c$. Then, for every fixed $L \in \N$, and any $a, b \in [L]$, we have
\begin{subequations}
\begin{equation}\label{eq:linear_hpb}
\abs{\textstyle \sum_{i} b_{i} q_a(\xi_i) }\prec \Psi_b,
\end{equation}
and
\begin{equation}
\abs{\textstyle \sum_{i, j} c_{ij} q_a(\xi_i) q_b(\xi_j) - (\sum_i c_{ii})\delta_{ab}} \prec \Psi_c.\label{eq:quadratic_hpb}
\end{equation}
\end{subequations}
\end{proposition}
%\yml{Need to fix the notation about a,b}
\begin{proof}
For any $\epsilon > 0$, $D > 0$,
\begin{align}
&\P\Big(\abs{\textstyle \sum_{i, j} c_{ij} q_a(\xi_i) q_b(\xi_j) - (\sum_i c_{ii})\delta_{ab}} \ge d^\epsilon \Psi_c\Big)\\
&\qquad\le \P\Big(\abs{\textstyle \sum_{i, j} c_{ij} q_a(\xi_i) q_b(\xi_j) - (\sum_i c_{ii})\delta_{ab}} \ge d^\epsilon \Psi_c \text{ and } \big(\textstyle\sum_{i,j} \abs{c_{ij}}^2\big)^{1/2} \le d^{\epsilon/2} \Psi_c\Big)\\
&\qquad\qquad\qquad+\P\Big(\big(\textstyle\sum_{i,j} \abs{c_{ij}}^2\big)^{1/2} > d^{\epsilon/2} \Psi_c\Big)\\
&\qquad\le \P\Big(\abs{\textstyle \sum_{i, j} c_{ij} q_a(\xi_i) q_b(\xi_j) - (\sum_i c_{ii})\delta_{ab}} \ge d^{\epsilon/2} \Big(\textstyle\sum_{i,j} \abs{c_{ij}}^2\big)^{1/2}\Big) + d^{-(D+1)}\\
&\qquad\le \big(\frac{C(a, b, p)}{d^{\epsilon/2}}\Big)^p + d^{-(D+1)}.
\end{align}
In the second step, we have used the definition of the stochastic dominance inequality \eref{dominance_psi} with parameters $\epsilon/2$ and $D+1$. The last step follows from the Markov inequality and the moment bound \eref{quadratic_qab}. For any $D > 0$, there is a large enough $p$ such that $\epsilon p/2 \ge D+1$. Thus, the right-hand side can be bounded by $d^{-D}$ for all sufficiently large $d$. This then established the bound in \eref{quadratic_hpb}. 

The proof of \eref{linear_hpb} follows exactly the same arguments. The only difference is that, instead of using the moment bound in \eref{quadratic_qab}, we appeal to the bound in \eref{linear_lp} and the moment estimate of $q_a(\xi)$ given in \eref{qk_normp} when applying the Markov inequality. We omit the details.
\end{proof}

%This last equality shows that the moments of $\xi$ match those of $\mathcal{N}(0, 1)$, up to a small correction factor of $\mathcal{O}(1/d)$. Thus, the law of $\xi$ converges to that of standard normal random variable as $d \to \infty$. 
%
%For large $d$, the law of $\xi$ is close to that of a standard Gaussian. It follows that $\set{q_{k,d}(x)}$ are closely related to the standard Hermite polynomials, denoted by $\set{p_k(x)}$.

% !TEX root = equivalence.tex

\section{The Operator Norms of $\set{A_k}$}
\label{appendix:op_norm}

In this appendix, we present some high probability bounds for the operator norms of the matrices $\set{\mA_k}$ defined in \eref{Ak}. 
%We will handle the low-order cases, for $0 \le k < \ell$, and the higher-order cases, for $k \ge \ell$, separately, by using different proof techniques.

\subsection{The Low-Order Components}
\label{appendix:low_order}

We start with the low-order components, \emph{i.e.}, those $\mA_k$ with $0 \le k < \ell$.

\begin{lemma}\label{lemma:op_spike}
Suppose that $n = \alpha d^\ell + o(d^\ell)$ for some $\ell \in \N$ and $\alpha > 0$. For each $k \in \N_0$ with $k < \ell$, let $\lambda_1(\mA_k) \ge \lambda_2(\mA_k) \ge \ldots \ge \lambda_n(\mA_k)$ denote the eigenvalues of $\mA_k$ in non-increasing order. For all sufficiently large $d$, we have
\begin{equation}\label{eq:Ak_spike}
\max_{i \in [\Nk]} \abs{\lambda_i(\mA_k) - \sqrt{n/\Nk}} = \mathcal{O}_\prec(1)
\end{equation}
and
\begin{equation}\label{eq:Ak_e0}
\lambda_i(\mA_k) = -\sqrt{\Nk/{n}} \qquad \text{for all } i > \Nk.
\end{equation}
\end{lemma}
\begin{remark}
Recall from \eref{Nk} that $\Nk =  {d^k}/{k!}+\mathcal{O}(d^{k-1})$. Thus, the condition $k < \ell$ ensures that $\Nk \ll n$ for all sufficiently large $d$. The lemma asserts that the spectrum of $\mA_k$ can be separated into two distinct clusters: the top $\Nk$ eigenvalues reside within an interval of width $\mathcal{O}_\prec(1)$, centered around $\sqrt{n/\Nk} = \sqrt{ k! \ratio d^{\ell-k}}(1+o(1))$; meanwhile, the remaining $n - \Nk$ eigenvalues are all equal to $\sqrt{\Nk/n} = {o}(1)$.
\end{remark}
\begin{proof}
Consider the factorized representation of $\mA_k$ in \eref{Ak_factor}. Observe that the positive-semidefinite matrix $\mY_k\tran \mY_k$ is rank-deficient. In fact, $\rank(\mY_k\tran \mY_k) = \Nk \ll n$, which then immediately gives us \eref{Ak_e0}.

Next, we study the top $\Nk$ eigenvalues of $\mA_k$. In what follows, we use $\lambda_i(\mM)$ to represent the $i$th largest eigenvalue of any symmetric matrix $\mM$. By \eref{Ak_factor}, and for each $i \in [\Nk]$, we have
\[
\lambda_i(\mA_k) + \sqrt{\Nk/n}= \frac{1}{\sqrt{n \Nk}} \lambda_i(\SH_k\tran \SH_k)  = \frac{1}{\sqrt{n \Nk}} \lambda_i(\SH_k \SH_k\tran).
\]
It follows that
\[
\abs{\lambda_i(\mA_k) - \sqrt{n/\Nk}} \le \frac{1}{\sqrt{n \Nk}}\abs{\lambda_i(\SH_k \SH_k\tran - n \mI)} + \sqrt{\Nk/n}
\]
and thus
\begin{equation}\label{eq:Ak_spike1}
\max_{i \in [\Nk]} \abs{\lambda_i(\mA_k) - \sqrt{n/\Nk}} \le \frac{1}{\sqrt{n \Nk}}\norm{\SH_k \SH_k\tran - n \mI}_\mathsf{op} + \mathcal{O}(d^{-(\ell-k)/2}).
\end{equation}

%\hty{ I guess all we try to say here is that the non-zero eigenvalues of $\mY_k \mY\tran_k$ and $\mY\tran_k \mY_k$ are the same. 

%Then by definition, all non-zero eigenvalues 
%\begin{equation}\label{eq:Ak_factor}
%\mA_k = \frac{1}{\sqrt{n \Nk}} \mY_k\tran \mY_k - \sqrt{{\Nk}/{n}}\, \mI,
%\end{equation}} 
The case of $k = 0$ is straightforward: $\SH_0 = (1, 1, \ldots, 1)$ represents a constant row vector of length $n$, and thus $\SH_0 \SH_0\tran - n = 0$. For $1 \le k < \ell$, we bound the operator norm of $\SH_k \SH_k\tran - n \mI$ by using the matrix Bernstein inequality. To that end, we first define a sequence of matrices $\set{\mZ_i = \vu_i \vu_i\tran - \mI}_{i \in [n]}$ where
\[
\vu_i = (\har{k}{1}(\vx_i), \har{k}{2}(\vx_i), \ldots, \har{k}{\Nk}(\vx_i))
\]
is an $\Nk$-dimensional random vector comprising of spherical harmonics. One can verify from the definition of $\SH_k$ in \eref{Yk_def} that
\begin{equation}\label{eq:sum_Z}
\SH_k \SH\tran_k - n \mI = \sum_{i \in [n]} \mZ_i.
\end{equation}
Since the data vectors $\vx_i \sim_\text{iid} \unifsp$, the matrices $\set{\mZ_i}$ are independent. By \eref{sh_ortho}, $\vu_i$ is an isotropic random vector and thus $\EE \mZ_i = \mat{0}$. Moreover, \eref{har_norm} implies that $\norm{\vu_i} = \sqrt{\Nk}$. It follows that
\begin{equation}\label{eq:def_K_Bernstein}
K \bydef \norm{\mZ_i}_\mathsf{op} = \Nk - 1
\end{equation}
and
\begin{equation}\label{eq:def_sigma_Bernstein}
\sigma^2 \bydef \norm{\textstyle\sum_{i \in [n]} \EE \mZ_i^2}_\mathsf{op} = \norm{\textstyle\sum_{i \in [n]} \EE [(\Nk-2) \vu_i \vu_i\tran + \mI]}_\mathsf{op} = n(\Nk - 1).
\end{equation}
For every $t > 0$, the standard matrix Bernstein inequality \cite{tropp2015IntroductionMatrix} gives us
\begin{align}
\P\Big(\frac{1}{\sqrt{n \Nk}}\norm{\textstyle\sum_{i \in [n]} \mZ_i}_\mathsf{op} \ge t\Big)  & \le 2\Nk \exp\Big(-\frac{n \Nk t^2/2}{\sigma^2 + K \sqrt{n \Nk} (t/3)}\Big)\\
&\le 2 \Nk \exp\Big(-\frac{t^2/2}{1 + \sqrt{\Nk/n} (t/3)}\Big),
\end{align}
where to reach the second inequality we have used the expressions in \eref{def_K_Bernstein} and \eref{def_sigma_Bernstein}. For every $0 < \epsilon < 1/2$, let $t = d^\epsilon$. Since $\Nk = \mathcal{O}(d^k)$ with $k < \ell$, we have $\sqrt{\Nk/n} (t/3) < 1$ for all sufficiently large $d$. It follows that, for every $D > 0$,
\[
\P\Big(\frac{1}{\sqrt{n \Nk}}\norm{\textstyle\sum_{i \in [n]} \mZ_i}_\mathsf{op} \ge t\Big) \le d^{-D}
\]
for sufficiently large $d$, and thus $\frac{1}{\sqrt{n \Nk}}\norm{\textstyle\sum_{i \in [n]} \mZ_i}_\mathsf{op} = \mathcal{O}_\prec(1)$. The statement \eref{Ak_spike} of the lemma then follows from \eref{sum_Z} and \eref{Ak_spike1}.
\end{proof}

\subsection{Moment Calculations for the High-Order Components}

Next, we bound the operator norms of the high-order matrices, \emph{i.e.}, those $\mA_k$ with $k \ge \ell$. The approach taken in our proof of Lemma~\ref{lemma:op_spike}, which is based on the matrix Bernstein inequality, can be modified to show that $\norm{\mA_\ell}_\mathsf{op} = \mathcal{O}_\prec(1)$, but it does not seem to yield useful results for cases when $k > \ell$. In what follows, we control the operator norm of $A_k$ by using the standard moment method. Observe that, for each $\epsilon > 0$ and $p \in \N$, we have
\begin{align}
\P(\,\norm{\mA_k}_\mathsf{op} \ge d^\epsilon) &= \P(\,\norm{\mA_k}_\mathsf{op}^{2p} \ge d^{2\epsilon p})\\
&\le d^{-2\epsilon p} \, \EE \norm{\mA_k}_\mathsf{op}^{2p}\\
&\le d^{-2\epsilon p} \, \EE\tr(\mA_k^{2p}).\label{eq:op_reg_1}
\end{align}
Thus, to control the probability on the left-hand side, it suffices to bound $\EE\tr(\mA_k^{2p})$ for a sufficiently large integer $p$. Bounds for the moments have been previously studied in the literature. \citet{cheng2013SpectrumRandom} provides a bound for the 4th moment $\EE\tr(\mA_k^{4})$. For general $p \in \N$, \citet[Lemma 3]{ghorbani2020Linearizedtwolayers} showed that
\begin{equation}\label{eq:moment_previous}
\EE\tr(\mA_k^{2p}) \le (Cp)^{3p} n(1 + d^{p(k-\ell) + \ell - 2k})
\end{equation}
for some positive constant $C$. By substituting \eref{moment_previous} into \eref{op_reg_1}, we observe that, when $k > \ell$, this moment bound is not sufficient to make \eref{op_reg_1} smaller than $d^{-D}$ for every $D > 0$. As a result, we cannot use \eref{moment_previous} to show $\norm{\mA_k}_\mathsf{op} \prec 1$. In this work, we adapt the moment calculations from \cite{ghorbani2020Linearizedtwolayers} and derive a new bound for $\EE\tr(\mA_k^{2P})$ that is tight up to its leading order (see Lemma~\ref{lemma:trace_bnd}).

As in \cite{ghorbani2020Linearizedtwolayers}, our moment calculations hinge on the following property of Gegenbauer polynomials.
\begin{lemma}
Let $\vz \sim \unifsp$. For each $k \in \N_0$ and any deterministic vectors $\vx, \vy$ in $\sph$, we have
\begin{equation}\label{eq:m_single}
\EE_{\vz} q_k(\sqrt{d} \vx\tran \vz) q_k(\sqrt{d} \vz\tran \vy) = \frac{1}{\sqrt{\Nk}} q_k(\sqrt{d} \vx\tran \vy).
\end{equation}
%As a special case, when $\vx = \vy$,
%\begin{equation}\label{eq:m_double}
%\EE_{\vz} q_k^2(\sqrt{d} \vx\tran \vz) = 1.
%\end{equation}
\end{lemma}
\begin{proof}
Using \eref{linearization}, we can write
\begin{align}
\EE_{\vz} q_k(\sqrt{d} \vx\tran \vz) q_k(\sqrt{d} \vz\tran \vy) &= \frac{1}{\Nk} \sum_{a, b \in [\Nk]} \har{k}{a}(\vx) \har{k}{b}(\vy) \big[\EE \har{k}{a}(\vz)\har{k}{b}(\vz)\big]\\
&= \frac{1}{\Nk} \sum_{a, b \in [\Nk]} \har{k}{a}(\vx) \har{k}{b}(\vy) \delta_{ab} \qquad [\text{by } \eref{sh_ortho}]\\
&= \frac{1}{\Nk} \sum_{a\in [\Nk]} \har{k}{a}(\vx) \har{k}{a}(\vy).
\end{align}%The identity \eref{m_double} follows from \eref{m_single} and \eref{har_norm}.
Applying \eref{linearization} one more time then gives us \eref{m_single}.
\end{proof}

\begin{lemma}\label{lemma:moment_unique}
Let $\vx_1, \ldots, \vx_n \sim_\text{iid} \unifsp$. For $k \in \N_0$ and $i, j \in [n]$, write 
\begin{equation}\label{eq:Q_shorthand}
Q_{ij} \bydef q_k(\sqrt{d} \vx_{i}\tran \vx_{j}).
\end{equation}
Then, for every length $2p$ sequence of indices $\vi = [i_1, i_2, \ldots, i_{2p}] \in [n]^{2p}$, we have
\begin{equation}\label{eq:E_cycle}
\abs{\EE \,Q_{i_1 i_2}Q_{i_2 i_3}\ldots Q_{i_{2p-1}i_{2p}}Q_{i_{2p} i_1}} \le [\EE_{\xi \sim \taumeasure} q_k^{2p}(\xi)] \Nk^{p-\nu(\vi)+1},
\end{equation}
where $\nu(\vi) \bydef \#\set{i_1, i_2, \ldots, i_{2p}}$ is the number of distinct indices in $\vi$, and $\taumeasure$ is the probability measure defined in \eref{tau_measure}.
\end{lemma}
\begin{remark}
This lemma is essentially a restatement of \cite[Lemma 3]{ghorbani2020Linearizedtwolayers}, after one adjusts for the different definitions and scalings used in that work and this paper. For the sake of readability and completeness, we present the results in a form that is more convenient for our subsequent arguments, and provide a slightly simplified proof below.
\end{remark}
\begin{proof}
We can view the product $Q_{i_1 i_2}Q_{i_2 i_3}\ldots Q_{i_{2p-1}i_{2p}}Q_{i_{2p} i_1}$ graphically, as a length $2p$ cycle on the vertex set $[n]$. First, consider two special cases: (I) $\nu(\vi) = 1$, \emph{i.e.}, all the indices are identical. Recall from \eref{har_norm} that
\begin{equation}\label{eq:m_double2}
Q_{ii} = \sqrt{\Nk} \quad \text{for all } i.
\end{equation}
Thus
\begin{equation}\label{eq:m_repeat}
\EE Q_{i_1 i_2}Q_{i_2 i_3}\ldots Q_{i_{2p}, i_1} =  Q_{i_1 i_1}^{2p} = \Nk^p.
\end{equation}
(II) $\nu(\vi) = 2p$, \emph{i.e.}, all the indices are distinct. Mapping the indices to the canonical set $[2p]$, we have
\begin{align}
\EE Q_{12}Q_{23}Q_{34}\ldots Q_{2p, 1} &= \EE_{\vx_1, \vx_3, \ldots, \vx_{2p}}\Big([\EE_{\vx_2} Q_{12} Q_{23}] Q_{34} \ldots Q_{2p, 1}\Big)\\
&= \frac{1}{\sqrt{\Nk}} \EE_{\vx_1, \vx_3, \ldots, \vx_{2p}}(Q_{13}Q_{34} \ldots Q_{2p, 1}),\label{eq:conditioning}
\end{align}
where the last step follows from \eref{m_single}.
The above reduction procedure can be applied repeatedly: at each step, it reduces the length of the cycle by $1$ and produces an extra factor of $1/\sqrt{\Nk}$. Doing this for $2p-1$ times then gives us
\begin{equation}\label{eq:m_unique}
\EE Q_{12}Q_{23}\ldots Q_{2p, 1} = \Nk^{-(2p-1)/2} \EE Q_{11} = \Nk^{1-p},
\end{equation}
where the last equality is due to \eref{m_double2}. Note that, since the Gegenbauer polynomials are normalized [recall \eref{ortho_poly}], we have 
\begin{equation}\label{eq:qk_moment_2p}
\EE q_k^{2p}(\xi) \ge (\EE q_k^2(\xi))^p = 1.
\end{equation}
Thus, with \eref{m_repeat} and \eref{m_unique}, we have verified the bound in \eref{E_cycle} for the cases of $\nu(\vi) = 1$ and $\nu(\vi) = 2p$, respectively.

In fact, the procedures leading to \eref{m_repeat} and \eref{m_unique} are special cases of a general and systematic reduction process. Suppose we are given a length $t$ cycle on indices $\vi = (i_1, i_2, \ldots, i_t)$. To simplify the notation, define
\begin{equation}\label{eq:M_i}
M(\vi) = \EE\, Q_{i_1 i_2} Q_{i_2 i_3} \ldots Q_{i_{t-1} i_t} Q_{i_t i_1}.
\end{equation}
We now introduce two reduction mappings, each of which converts a length $t$ cycle to a length $(t-1)$ cycle.

\emph{Type-A reduction}: Given an index sequence $\vi = (i_1, i_2, \ldots, i_t)$ for $t \ge 2$. Suppose there exists at least one $j \in [t]$ such that $i_j \neq i_a$ for all $a \in [t] \setminus \set{j}$. In other words, the index $i_j$ appears exactly once in the cycle. (It is possible that there are more than one such ``singleton'' indices in the sequence, in which case we will choose any one of them.) We remove $i_j$ from the original sequence $\vi$, and call the resulting length $t-1$ sequence $R_A(\vi)$. More precisely,
\[
R_A(\vi) = (i_1, \ldots, i_{j-1}, i_{j+1}, \ldots, i_t)
\]
where the indices $j-1$ and $j+1$ are interpreted modulo $[t]$. Using the same conditioning technique that leads to \eref{conditioning}, we can easily verify that
\begin{equation}\label{eq:M_A}
M(\vi) = \frac{1}{\sqrt{\Nk}} M(R_A(\vi)).
\end{equation}
Thus, a type-A reduction step $R_A(\vi)$ reduces the cycle length by 1 and contributes a factor of $1/\sqrt{\Nk}$.

\emph{Type-B reduction}: Given an index sequence $\vi = (i_1, i_2, \ldots, i_t)$ for $t \ge 2$. Suppose there exists at least one $j \in [t]$ such that $i_{j+1} = i_j$, where $j+1$ is to be interpreted modulo $[t]$. (If more than one such indices exist, we choose any one of them.) We define
\[
R_B(\vi) = (i_1, \ldots, i_{j-1}, i_{j+1}, \ldots, i_t)
\]
as a length $t-1$ sequence obtained by removing $i_j$ from $\vi$. By \eref{m_double2}, we must have
\begin{equation}\label{eq:M_B}
M(\vi) = \sqrt{\Nk} M(R_B(\vi)).
\end{equation}
Thus, a type-B reduction step $R_B(\vi)$ reduces the cycle length by 1 and contributes a factor of $\sqrt{\Nk}$.

Given an index sequence $\vi = (i_1, i_2, \ldots, i_t)$, we can simplify it by iteratively using the Type-A and Type-B reductions, until it cannot be further simplified. Clearly, this process stops after a finite number of steps. Let $\bar{\vi} = (\bar{i}_1, \bar{i}_2, \ldots, \bar{i}_{{t}(\bar{\vi})})$ denote the final outcome of the iterative reduction process.

\begin{example}
As an illustration, let us consider two concrete index sequences and demonstrate how they can be simplified though the above process.
\begin{enumerate}
\item[(a)] $\vi = (1, 1, 2, 2, 2, 3)$: We can start by using a Type-A step to remove the ``singleton'' 3. Then, the ``redundant'' copies of $1$ and $2$ can be removed by applying the Type-B step three times. This then gives us a shorter sequence $(1, 2)$, in which both $1$ and $2$ are singletons. Removing $1$ by a Type-A step gives us the final output $\bar{\vi} = (2)$.

\item[(b)] $\vi = (1, 2, 1, 2, 2, 3)$: Removing the ``singleton'' 3 (via Type-A reduction) and one redundant copy of $2$ (via Type-B reduction), we get $\bar{\vi} = (1, 2, 1, 2)$, which cannot be further simplified.
\end{enumerate}
\end{example}

By \eref{M_A} and \eref{M_B}, we must have
\begin{equation}\label{eq:M_AB}
M(\vi) = \Nk^{(n_B - n_A)/2} M(\bar{\vi}),
\end{equation}
where $n_A$ (resp. $n_B$) is the total number of Type-A (resp. Type-B) steps used in the reduction process that leads to the final outcome $\bar{\vi}$. Denote by $t(\vi)$ and $\nu(\vi)$ the length and number of unique indices in $\vi$, respectively. We also define $t(\bar{\vi})$ and $\nu(\bar{\vi})$ in the same way for the simplified sequence $\bar{\vi}$. By construction, we must have 
\[
n_A + n_B = t(\vi) - t(\bar{\vi}).
\]
One can also check that
\[
n_A = \nu(\vi) - \nu(\bar{\vi}).
\]
Substituting these two identities into \eref{M_AB} then gives us
\begin{equation}\label{eq:M_AB2}
M(\vi) = \Nk^{(t(\vi) - 2\nu(\vi) + 2 \nu(\bar{\vi}) - t(\bar{\vi}))/2} M(\bar{\vi}),
\end{equation}

Again, by examining the constructions of the Type-A and Type-B reduction steps, we can conclude that there can only be two possibilities for the final output $\bar{\vi}$: 

\emph{Case 1}: $t(\bar{\vi}) = \nu(\bar{\vi}) = 1$, \emph{i.e.}, $\bar{i} = (\bar{i}_1)$ is a length 1 cycle. Recall from \eref{m_double2} that $M(\bar{\vi}) = \sqrt{\Nk}$. It then follows from \eref{M_AB2} that
\[
M(\vi) = \Nk^{(t(\vi) - 2\nu(\vi) + 2)/2}.
\]
By applying this identity to a sequence of length $t(\vi) = 2p$ and recalling \eref{qk_moment_2p}, we reach the statement in \eref{E_cycle}.

\emph{Case 2}: The only other possibility is for $\bar{\vi}$ to contain at least two unique indices, \emph{i.e.}, $\nu(\bar{\vi}) \ge 2$. Moreover, every unique index must appear at least twice in the cycle, as otherwise the sequence $\bar{\vi}$ can be further simplified by a Type-A step. These two conditions imply $t(\bar{\vi}) \ge 2 \nu(\bar{\vi})$, in which case \eref{M_AB} gives us
\begin{equation}\label{eq:M_a}
\abs{M(\vi)} \le \Nk^{(t(\vi) - 2\nu(\vi))/2} \abs{M(\bar{\vi})}.
\end{equation}
Moreover, consecutive indices in $\bar{\vi} = (\bar{i}_1, \bar{i}_2, \ldots, \bar{i}_{t(\bar{\vi})})$ cannot have repetitions, \emph{i.e.}, $\bar{i}_j \neq \bar{i}_{j+1}$ for all $j \in [t(\bar{\vi})]$ with $j +1$ to be interpreted as modulo $t(\bar{\vi})$. (If this were not true, then $\bar{\vi}$ would not be the final output as it could be further simplified by a Type-B step.) This condition allows us to apply H\"{o}lder's inequality to get
\begin{equation}\label{eq:M_b}
\abs{M(\bar{\vi})} = \abs{\EE Q_{\bar{i}_1 \bar{i}_2} Q_{\bar{i}_2 \bar{i}_3} \ldots Q_{\bar{i}_{t(\bar{\vi})} i_1}} \le \EE_{\xi \sim \taumeasure} \abs{q_k(\xi)}^{t(\bar{\vi})}.
\end{equation}
Now consider a sequence $\vi$ with length $t(\vi) = 2p$. Since $t(\bar{\vi}) \le 2p$, we can use H\"{o}lder's inequality to further bound the right-hand side as
\begin{equation}\label{eq:M_c}
\EE_{\xi \sim \taumeasure} \abs{q_k(\xi)}^{t(\bar{\vi})} \le [\EE_{\xi \sim \taumeasure} q_k^{2p}(\xi)]^{t(\bar{\vi}) / (2p)} \le \EE_{\xi \sim \taumeasure} q_k^{2p}(\xi),
\end{equation}
where the last step follows from  \eref{qk_moment_2p}, \emph{i.e.}, $ \EE_{\xi \sim \taumeasure} q_k^{2p} \ge 1$.  By substituting \eref{M_b} and \eref{M_c} into \eref{M_a}, we can verify the bound in \eref{E_cycle}.
\end{proof}

\begin{lemma}\label{lemma:trace_bnd}
For every $k \in \N_0$ and $p \in \N$, we have
\begin{equation}\label{eq:trace_bnd}
\EE\tr(\mA_k^{2p}) \le n (2p)^{2p-2}  \big[p/n + \sum_{i = 0}^{p-1} ({n}/{\Nk})^i\big] \cdot \EE_{\xi \sim \taumeasure} q_k^{2p}(\xi).
\end{equation}
\end{lemma}
\begin{proof}
Using the shorthand notation in \eref{Q_shorthand}, we can write $A_k(i, j) = \frac{1}{\sqrt{n}} Q_{ij} \cdot (1-\delta_{ij})$ for $i, j \in [n]$. Let 
\[
\mathcal{C} \bydef \set{\vi = (i_1, i_2, \ldots, i_{2p}) \in [n]^{2p}: i_j \neq i_{j+1} \text{ for all } j \in [2p-1] \text{ and } i_1 \neq i_{2p}}
\]
denote the set of all length $2p$ cycles in which no two consecutive indices are equal. We then have
\begin{align}
\EE\tr(\mA_k^{2p}) &= \frac{1}{n^p} \sum_{\vi \in \mathcal{C}} \EE\, Q_{i_1 i_2} Q_{i_2 i_3} \ldots Q_{i_{2p-1} i_{2p}}Q_{i_{2p} i_1}\\
&= \frac{1}{n^p} \Big(\sum_{\vi \in \mathcal{C}_1} M(\vi) + \sum_{\vi \in \mathcal{C}_2} M(\vi) + \ldots + \sum_{\vi \in \mathcal{C}_{2p}} M(\vi)\Big),\label{eq:trace_bnd_a}
\end{align}
where, for each $a \in [2p]$,
\[
\mathcal{C}_a \bydef \mathcal{C} \cap \set{\vi \in [n]^{2p}: \nu(\vi) = a}
\]
and $M(\vi) = \EE\, Q_{i_1 i_2} Q_{i_2 i_3} \ldots Q_{i_{2p-1} i_{2p}}Q_{i_{2p} i_1}$. The terms $M(\vi)$ in \eref{trace_bnd_a} can be bounded in two different ways. First, recall that, for any $i_j \neq i_{j+1}$, we have $Q_{i_j i_{j+1}} \sim q_k(\xi)$ with $\xi \sim \taumeasure$. Thus, for every $\vi \in \mathcal{C}$, we can apply H\"{o}lder's inequality to get
\begin{equation}\label{eq:M_bnd_1}
\abs{M(\vi)} \le \EE q_k^{2p}(\xi).
\end{equation}
On the other hand, for each $a \in [2p]$, Lemma~\ref{lemma:moment_unique} gives us
\begin{equation}\label{eq:M_bnd_2}
\abs{M(\vi)} \le \EE q_k^{2p}(\xi) \Nk^{p - a + 1} \quad \text{for all } \vi \in \mathcal{C}_a.
\end{equation}
We use \eref{M_bnd_1} to bound the $M(\vi)$ terms in $\mathcal{C}_1, \mathcal{C}_2, \ldots, \mathcal{C}_{p+1}$ and use \eref{M_bnd_2} for the $M(\vi)$ terms in $\mathcal{C}_{p+2}, \ldots, \mathcal{C}_{2p}$. It then follows from \eref{trace_bnd_a} that
\begin{equation}\label{eq:trace_bnd1}
\abs{\EE\tr(\mA_k^{2p})} \le \frac{1}{n^p}  \Big(\sum_{a = 1}^{p+1} \abs{\mathcal{C}_a} + \sum_{a=p+2}^{2p}\abs{\mathcal{C}_a} \Nk^{p - a + 1}\Big) \EE q_k^{2p}(\xi),
\end{equation}
where $\abs{\mathcal{C}_a}$ denotes the cardinality of $\mathcal{C}_a$. For each $a \in [2p]$, we have
\[
 \abs{\mathcal{C}_a} \le a^{2p-2} n(n-1)\ldots (n-a+1) \le (2p)^{2p-2} n^a.
 \]
 This is a crude bound, but it is sufficient for our purpose. Substituting this bound into \eref{trace_bnd1} give us
\[
\abs{\EE\tr(\mA_k^{2p})} \le (2p)^{2p-2} \Big(\sum_{a = 1}^{p+1} n^{a-p} + n \sum_{a=p+2}^{2p}({n}/{\Nk})^{a - p -1}\Big) \EE q_k^{2p}(\xi),
\]
which then implies \eref{trace_bnd}.
\end{proof}

\subsection{Bounds on the Operator Norms}

\begin{proposition}\label{prop:Ak_op_bnd}
Under the conditions of Theorem~\ref{thm:equivalence}, we have, for every $k \in \N_0$,
%\[
%\norm{\mA_k}_\mathsf{op} = \begin{cases} \mathcal{O}_{\prec}(1) &\text{if } k \ge \ell\\
%\frac{\sqrt n}{\sqrt{\Nk}} + \mathcal{O}_{\prec}(1)  &\text{if } 0 \le k < \ell,
%\end{cases}
%\]
\begin{equation}\label{eq:Ak_op_bnd}
\norm{\mA_k}_\mathsf{op} \prec d^{\max\{\ell - k, 0\}/2}.
\end{equation}
\end{proposition}
\begin{proof}
The statement of the proposition means that $\norm{A_k}_\mathsf{op} \prec 1$ for all the high-order components with $k \ge \ell$. In contrast, for the low-order components, $\norm{A_k}_\mathsf{op} \prec d^{(\ell-k)/2}$ when $0 \le k < \ell$.

To show \eref{Ak_op_bnd} for the case of $k \ge \ell$, suppose that we are given an arbitrary $\epsilon > 0$. Substituting \eref{trace_bnd} into \eref{op_reg_1} yields
\begin{align}
\P(\,\norm{\mA_k}_\mathsf{op} \ge d^\epsilon) &\le d^{-2\epsilon p} \, n (2p)^{2p-2}  \big[p/n + \sum_{i = 0}^{p-1} ({n}/{\Nk})^i\big] \cdot \EE_{\xi \sim \taumeasure} q_k^{2p}(\xi)
\end{align}
Recall that $n = \alpha d^\ell + o(d^\ell)$ and $k \ge \ell$. In addition, we have $\Nk = \frac{d^k}{k!}(1+\mathcal{O}(1/d))$. Thus, for sufficiently large $d$, we have, by Markov's inequality, 
\[
\P(\,\norm{\mA_k}_\mathsf{op} \ge d^\epsilon) \le C(\alpha, k, p) d^{\ell - 2\epsilon p},
\]
where $C(\alpha, k, p)$ is some constant that depends on $\alpha, k, p$ but not on $d$. For any $D > 0$, choose a $p$ such that $\ell - 2 \epsilon p < -D - 1$. It follows that
\[
\P(\,\norm{\mA_k}_\mathsf{op} \ge d^\epsilon) \le d^{-D}
\]
for all sufficiently large $d$.

Next, we consider the low-order components, with $0 \le k < \ell$. Lemma~\ref{lemma:op_spike} implies that
\begin{equation}
\norm{A}_\mathsf{op} \le \sqrt{n/N_k} + \mathcal{O}_\prec(1) \prec d^{(\ell-k)/2},
\end{equation}
where the second step uses the fact that $n = \alpha n^\ell + o(n^\ell)$ and $\Nk = d^k/k!(1+\mathcal{O}(1/d))$.
\end{proof}

% !TEX root = equivalence.tex

\section{Perturbations of Resolvents and Stieltjes Transforms}
\label{appendix:resolvent_identities}

In this appendix, we recall some standard properties of the resolvent and the Stieltjes transform of Hermitian matrices. Let $H$ be an $n \times n$ Hermitian matrix. Its resolvent is denoted by $G(z) \bydef (H - zI)^{-1}$, where $z \in \C$ is the spectral parameter with $\eta = \im[z] > 0$. From the trivial inequality
\[
\abs{\frac{1}{\lambda - z}} \le \frac{1}{\eta}\qquad \text{for all } \lambda \in \R,
\]
one can easily verify the following bounds on the operator and the entry-wise $\ell_\infty$ norms of $G(z)$:
\begin{align}
\norm{G(z)}_\mathsf{op} &\le \frac{1}{\eta}\label{eq:G_op}\\
\norm{G(z)}_\infty &\le \frac{1}{\eta}.\label{eq:G_inf}
\end{align}
In our proof, we will also need the \emph{Ward identity} \cite[Lemma 8.3]{erdos2017Dynamicalapproach}:
\begin{equation}\label{eq:Ward}
\sum_{ij} \abs{G_{ij}(z)}^2 = \frac{\im(\tr G(z))}{\eta}.
\end{equation}

Let $s(z)$ be the Stieltjes transform of the empirical spectral distribution of $H$. Since $s(z) = \frac{1}{n} \tr G(z)$, \eref{G_inf} immediately implies the pointwise bound
\begin{equation}\label{eq:sH_bnd}
\abs{s(z)} \le \frac{1}{\eta}.
\end{equation}
One can also easily check that
\begin{equation}\label{eq:dsdz}
\abs{s(z_1) - s(z_2)} \le \frac{\abs{z_1 - z_2}}{\im(z_1) \im(z_2)}.
\end{equation}

Both the resolvent and the Stieltjes transform are stable with respect to matrix perturbations. In our proof of Theorem~\ref{thm:equivalence}, we will use the following standard perturbation estimates.
\begin{lemma}\label{lemma:sG_perturbation}
Let $H_1, H_2$ be $n \times n$ Hermitian matrices, and $G_1(z), G_2(z)$ their resolvents. Then,
\begin{equation}\label{eq:resolvent_id}
G_1(z) = G_2(z) - G_2(z) (H_1 - H_2) G_1(z).
\end{equation}
Moreover, for the Stieltjes transforms of $H_1$ and $H_2$, we have
\begin{align}
\abs{s_1(z) - s_2(z)} &\le \frac{\norm{H_1 - H_2}_\mathsf{F}}{\sqrt{n}\, \eta^2} \le \frac{\norm{H_1 - H_2}_\mathsf{op}}{\eta^2}\label{eq:s_diff}\\
\intertext{and}
\abs{s_1(z) - s_2(z)} &\le \frac{C \rank(H_1 - H_2)}{n \cdot \eta},\label{eq:s_diff_rank}
\end{align}
where $C$ is an absolute constant.
\end{lemma}
\begin{proof}
The formula in \eref{resolvent_id} can be easily verified by using the identities $I = (H_1 - z I) G_1 = (H_2-zI + H_1 - H2) G_1$ and $G_2 (H_2 - zI) = I$. Applying \eref{resolvent_id} gives us
\begin{align}
\abs{s_1(z) - s_2(z)} &= \frac{1}{n} \abs{\tr [G_2(z) (H_1 - H_2) G_1(z)]}\\
&=\frac{1}{n} \abs{\tr [G_1(z)G_2(z) (H_1 - H_2)]}\\
&\le \frac{1}{n} \norm{G_1(z) G_2(z)}_\mathsf{F} \norm{H_1 - H_2}_\mathsf{F},\label{eq:s_diff0}
\end{align}
where the last step uses the Cauchy-Schwarz inequality. By using \eref{G_op} in the last step, we have
\[
\norm{G_1(z) G_2(z)}_\mathsf{F} \le \sqrt{n} \norm{G_1(z) G_2(z)}_\mathsf{op} \le  \sqrt{n} \norm{G_1(z)}_\mathsf{op} \norm{G_1(z)}_\mathsf{op} \le \frac{\sqrt{n}}{\eta^2}.
\]
Substituting this bound into \eref{s_diff0} then leads to the first inequality in \eref{s_diff}. The second inequality in \eref{s_diff} then follows immediately from the fact that $\norm{H_1 - H_2}_\mathsf{F} \le \sqrt{n} \norm{H_1 - H_2}_\mathsf{op}$.

Next, we show \eref{s_diff_rank}. First consider the special case when $\rank(H_1 - H_2) = 1$. As the eigenvalues are invariant under unitary transforms, we can assume without loss of generality that
\begin{equation}\label{eq:H1H2_rankone}
H_1 - H_2 = a \ve_1 \ve_1\tran,
\end{equation}
for some $a \neq 0$. Let $H_1^{[1]}$ (resp. $H_2^{[1]}$) denote the $(n-1)\times(n-1)$ minor matrix obtained by removing the first column and row of $H_1$ (resp. $H_2$). By \eref{H1H2_rankone}, $H_1^{[1]} = H_2^{[1]}$. Let $s^{[1]}(z)$ denote the Stieltjes transform of the eigenvalues of $H_1^{[1]}$. We have
\begin{equation}\label{eq:s_diff_rank_one}
\abs{s_1(z) - s_2(z)} \le \abs{s_1(z) - s^{[1]}(z)} + \abs{s^{[1]}(z) - s_2(z)} \le \frac{C}{n \cdot \eta},
\end{equation}
where the second step uses \cite[Lemma 7.5]{erdos2017Dynamicalapproach}, and $C$ is an absolute constant. For the case when $\rank(H_1 - H_2) > 1$, we can always write $s_1(z) - s_2(z)$ as a telescoping sum of rank-one perturbations. The inequality in \eref{s_diff_rank} can then be obtained by applying the triangular inequality and by using the result for the rank-one case.
\end{proof}

\begin{proof}[Proof of Lemma~\ref{lemma:omit_low_order}]
By using the factorized representation of $A_k$ in \eref{Ak_factor}, we can write the low-order term as
\begin{equation}\label{eq:A_L_fact}
\Al = \sum_{0 \le k < \ell} \frac{\co_k}{\sqrt{n \Nk}} \SH_k\tran \SH_k - \sigma_d I,
\end{equation}
where $\SH_k$ is an $\Nk \times n$ matrix, with $\Nk = \mathcal{O}(d^k)$, and $\sigma_d \bydef \sum_{0 \le k < \ell} \co_k \sqrt{\Nk / n}$. Since $n \asymp d^\ell$, we have $\sigma_d = \mathcal{O}(d^{-1/2})$ and that
\begin{equation}\label{eq:AL_rank}
\rank(\Al + \sigma_d I) \le \sum_{0 \le k < \ell} \Nk = \mathcal{O}(d^{\ell-1}).
\end{equation}
Write $\breve{A} \bydef A + \sigma_d I = (\Al + \sigma_d I) + \Ah$, whose Stieltjes transform is denoted by $\breve{s}(z)$. By the triangular inequality,
\begin{align}
\abs{s(z) - \hat{s}(z)} &\le \abs{s(z) - \breve{s}(z)} + \abs{\breve{s}(z) - \hat{s}(z)}\\
&\le \frac{\sigma_d}{\eta^2} + \frac{C \rank(\breve{A} - \Ah)}{n \cdot \eta},\label{eq:s_diff_LH1}
\end{align}
where the last step follows from the perturbation inequalities in \eref{s_diff} and \eref{s_diff_rank}. By substituting the estimates $\sigma_d = \mathcal{O}(d^{-1/2})$ and \eref{AL_rank} into \eref{s_diff_LH1}, we complete the proof.
\end{proof}

We conclude this appendix by establishing a general comparison inequality that will be utilized in our proofs for both Theorem~\ref{thm:equivalence_f} and Theorem~\ref{thm:equivalence_fg}. Given that the former focuses on data vectors sampled from the spherical distribution, while the latter pertains to isotropic Gaussian data vectors, we present a general model in the following lemma that can accommodate both scenarios.

\begin{lemma}\label{lemma:s1s2_general}
Let $\vx_1, \ldots, \vx_n \in \R^d$ be a set of independent and identically distributed vectors sampled from a probability distribution $\pi$. Consider a random matrix $A \in \R^{n \times n}$ with entries
\begin{equation}\label{eq:A_f_general}
A_{ij} \bydef \begin{cases} \frac{1}{\sqrt{n}} \, f_d(\sqrt{d}\, \vx_i\tran \vx_j) &\text{if } i \ne j\\
0&\text{if } i = j
\end{cases},
\end{equation}
where $f_d: \R \mapsto \R$ is a function. Let $\hat A$ be another matrix constructed in the same way as in \eref{A_f_general}, but with $f_d$ replaced by a different function $\hat f_d$. For $z \in C_+$ with $\eta = \im[z]$, let $s(z)$ and $\hat{s}(z)$ denote the Stieltjes transforms of $A$ and $\hat A$, respectively. Then, we have
\begin{equation}\label{eq:s1s2_general}
\abs{s(z) - \hat s(z)} \le \frac{1}{\eta^2} \big[\EE_{\vx, \vy} \big(f_d(\sqrt{d} \, \vx\tran \vy) - \hat f_d(\sqrt{d}\, \vx\tran \vy)\big)^2\big]^{1/2} + \mathcal{O}_\prec\Big(\frac{1}{\eta\sqrt{n}}\Big),
\end{equation}
where $\vx, \vy \in \R^d$ are two independent vectors sampled from the distribution $\pi$.
\end{lemma}
\begin{proof}
By the triangular inequality,
\begin{align}
\abs{s(z) - \hat s(z)} &\le \abs{\EE s(z) - \EE \hat s(z)} + \abs{s(z) - \EE s(z)} + \abs{\hat s(z) - \EE \hat s(z)}\\
&\le \EE \abs{s(z) - \hat s(z)} + \abs{s(z) - \EE s(z)} + \abs{\hat s(z) - \EE \hat s(z)}.\label{eq:s1s2_triangle}
\end{align}
Next, we bound each term on the right-hand side. For the first term, we apply the perturbation estimate \eref{s_diff} in Lemma~\ref{lemma:sG_perturbation} and H\"{o}lder's inequality, which result in the following expression:
\begin{align}
\EE \abs{s(z) - \hat s(z)} &\le \frac{1}{\sqrt{n} \eta^2} \EE \norms{A - \hat A}_\mathrm{F}\\
&\le \frac{1}{\sqrt{n} \eta^2} (\EE \norms{A - \hat A}_\mathrm{F}^2)^{1/2}\\
&\le \frac{1}{\eta^2} \big[\EE_{\vx, \vy} \big(f_d(\sqrt{d} \vx\tran \vy) - \hat f_d(\sqrt{d} \vx\tran \vy)\big)^2\big]^{1/2}.\label{eq:s1s2_moment}
\end{align}
To reach the last step, we have used the fact that, for any $i, j \in [n]$ with $i \neq j$, the probability distribution of $\sqrt{n} (A_{ij} - \hat A_{ij})$ equals to that of $f_d(\sqrt{d}\, \vx\tran \vy) - \hat f_d(\sqrt{d}\, \vx\tran \vy)$, for two independent vectors $\vx, \vy$ sampled from the distribution $\pi$. 

The second and the third terms on the right-hand side of \eref{s1s2_triangle} can be bounded by standard concentration arguments in random matrix theory. One approach is to view $s(z)$ and $\hat s(z)$, for fixed $z$, as functions of the independent vectors $\set{\vx_i}_{i \in [n]}$. It is straightforward to verify that these functions have bounded variations. This then allows one to apply McDiarmid's inequality to bound the deviation of $s(z)$ and $\hat s(z)$ from their corresponding expectations. Another approach, as used in \cite{cheng2013SpectrumRandom}, is to apply the Burkholder's inequality \cite[Lemma 2.12]{baiSpectralAnalysisLarge2010} to directly bound the moments of the difference between $s(z), \hat s(z)$ and their corresponding expectations. By \cite[Lemma 2.4]{cheng2013SpectrumRandom}, for any $p \in \N$,
\[
\EE \abs{s(z) - \EE s(z)}^{2p} \le \frac{C_p}{\eta^{2p}n^p },
\]
where $C_p$ is a constant that depends on $p$. For any $\varepsilon > 0$, applying Markov's inequality gives us
\begin{align}
\P\big(\abs{s(z) - \EE s(z)} \ge \frac{d^{\varepsilon}}{\eta \sqrt{n}}\big) = \P\big(\abs{s(z) - \EE s(z)}^{2p} \ge \frac{d^{2\varepsilon p}}{\eta^{2p} n^p}\big) \le C_p d^{-2\varepsilon p}.
\end{align}
For any $D > 0$, we can always find a $p$ such that the right-hand side of the above inequality is less than $d^{-D}$ for all sufficiently large $d$. It follows that 
\begin{equation}\label{eq:sz_concentration}
s(z) = \EE s(z) + \mathcal{O}_\prec\Big(\frac{1}{\eta \sqrt{n}}\Big).
\end{equation}
Using the same steps, we can also show:
\begin{equation}\label{eq:hat_sz_concentration}
\hat s(z) = \EE \hat s(z) + \mathcal{O}_\prec\Big(\frac{1}{\eta \sqrt{n}}\Big).
\end{equation}
By substituting \eref{s1s2_moment}, \eref{sz_concentration}, and \eref{hat_sz_concentration} into \eref{s1s2_triangle}, we complete the proof.
\end{proof}

% !TEX root = equivalence.tex

\section{Stability Analysis of the Self-Consistent Equation}
\label{appendix:stability}

In this appendix, we present a stability analysis of the self-consistent equation in \eref{m_equation} satisfied by the limiting Stieltjes transform. We start by rewriting \eref{m_equation} as
\begin{equation}\label{eq:m_equation_frac}
\frac 1 {m} = -z -  \frac {\coa m}{1 + \cob m} - \coc m.
\end{equation}
Now suppose $s \in C_+$ is an approximate solution to \eref{m_equation_frac} such that
\begin{equation}\label{eq:m_equation_approx}
\frac 1 {s} = -z -  \frac {\coa s}{1 + \cob s} - \coc s + \et,
\end{equation}
with some error term $\et$. We show in the following proposition that $s \approx m$ as long as the error $\et$ is small.

\begin{proposition}\label{prop:stability}
For any $z \in C_+$, there exists a unique solution $m \in C_+$ to \eref{m_equation_frac}. Moreover, let $s \in C_+$ be an approximate solution to \eref{m_equation} in the sense of \eref{m_equation_approx}, with the error term satisfying
\begin{equation}\label{eq:fp_small_error}
\abs{\et} \le \frac\eta 2, \qquad \eta = \im[z].
\end{equation}
We have
\begin{equation}\label{eq:m_equation_stability}
\abs{s - m} \le \frac{4\abs{\et}}{\eta^2}.
\end{equation}
\end{proposition}
\begin{proof}
The existence and uniqueness of the solution $m \in \C_+$ is a well-known result. One way to establish this is to observe that $\eref{m_equation_frac}$ is just the free additive convolution of the Marchenko-Pastur law and the semicircle law. See, \emph{e.g.}, \cite{biane1997freeconvolution} for more details. An alternative approach, as used in \cite[Lemma A.1]{cheng2013SpectrumRandom}, is to rewrite \eref{m_equation_frac} as a third-degree polynomial in $m$ and to analyze the roots of that polynomial.

Next, we establish the stability bound in \eref{m_equation_stability}. Recall the family of MP density functions given in \eref{rho_MP}, with two shape parameters $\rp$ and $t$. In what follows, we set $\rp = \coa/\cob^2$ and $t = \cob$, and write $\varrho_\mathrm{MP} \bydef \varrho_{\rp, t}$. The Stieltjes transform of $\varrho_\mathrm{MP}$ is specified as the unique solution to \eref{m_MP}. 
  Substituting our particular choices of $\rp$ and $t$ into \eref{m_MP}, and by the uniqueness of the solution, we obtain the following characterization: for any $z, \Gamma \in C_+$, 
\begin{equation}\label{eq:Gamma_rho}
\frac{1}{\Gamma} = -z - \frac{\coa \Gamma}{1 + \cob \Gamma} \qquad \Leftrightarrow \qquad \Gamma = \int \frac{\varrho_\mathrm{MP}(\dif x)}{x - z}.
\end{equation}
Using \eref{Gamma_rho} and \eref{m_equation_frac}, we can now write the exact solution $m$ as
\begin{equation}\label{eq:m_MP1}
m =  \int \frac{\varrho_\mathrm{MP}(\dif x) }{  x  - z -  \coc m }.
\end{equation}
By the assumption in \eref{fp_small_error}, $\im[z +\coc s - \omega] \ge \eta/2$. Thus, it follows from \eref{Gamma_rho} and \eref{m_equation_approx} that
\begin{equation}\label{eq:s_MP}
s = \int \frac{\varrho_\mathrm{MP}(\dif x) }{  x  - z -  \coc s + \omega }.
\end{equation}
Computing the imaginary part of the equation for $m$, we get 
\begin{equation}\label{eq:im_mz}
\im[m] =  \int   \frac{ \eta + \coc \im[ m]}{  \abs{x  - z -  \coc m}^2 } \,\varrho_\mathrm{MP}(\dif x),
\end{equation}
which also gives us the trivial bound $\im[m] \le 1/\eta$. Similarly, 
\begin{align}
\im[s] &=  \int   \frac{ \eta + \coc \im[s] - \im[\et]}{  \abs{x  - z -  \coc s + \et}^2 }\, \varrho_\mathrm{MP}(\dif x)\label{eq:im_s0}\\
&\ge \int   \frac{ \eta/2 + \coc \im[s]}{  \abs{x  - z -  \coc s + \et}^2 }\, \varrho_\mathrm{MP}(\dif x),\label{eq:im_s}
\end{align}
where the second step uses \eref{fp_small_error}. Note that \eref{im_s0} and \eref{fp_small_error} also imply that $\im[s] \le 2/\eta$. Taking the difference of \eref{m_MP1} and \eref{s_MP} gives us
 \begin{align} \label{eq:imms_0}
  (s-m) \Big [ 1- \int \frac{\coc \, \varrho_\mathrm{MP}(\dif x) }{  (x  - z -  \coc m)  (x  - z -  \coc s + \et)} \Big ]=  \int \frac{ -\et \,\varrho_\mathrm{MP}(\dif x)  }{  (x  - z -  \coc m)  (x  - z - \coc s + \et)}.
\end{align} 
By the Cauchy-Schwarz inequality in the first step, and by \eref{im_mz}, \eref{im_s} in the second step, 
\begin{align}
\abs{  \int \frac{\varrho_\mathrm{MP}(\dif x) }{  (x  - z -  \coc m)  (x  - z -  \coc s + \et)} } &  \le \Big [  \int \frac{\varrho_\mathrm{MP}(\dif x) }{  \abs{x  - z -  \coc m}^2} 
 \int \frac{\varrho_\mathrm{MP}(\dif x) }{  \abs{x  - z -  \coc s+ \et }^2} \Big ]^{1/2} \\
 & \le  \Big [ \frac { \im [m]} { \eta + \coc \im [m]}  \cdot \frac {\im[s]} { \eta/2 + \coc \im[s]} \Big ]^{1/2}\\
 &\le \frac{1}{\eta^2/4 + \coc}, \label{eq:imms_1}
\end{align} 
where the last step uses the bound that $\max\set{\im[m], \im[s]} \le 2/\eta$. This allows us to get
\begin{equation} \label{eq:imms_2}
\abs{ 1- \int \frac{\coc \, \varrho_\mathrm{MP}(\dif x) }{  (x  - z -  m)  (x  - z -  s + \et)}}  \ge 
 1-    \frac { \coc } {\eta^2/4 + \coc} = \frac{\eta^2}{\eta^2 + 4 \coc}.
\end{equation} 
Substituting \eref{imms_1} and \eref{imms_2} into \eref{imms_0}, we then reach the statement \eref{m_equation_stability} of the proposition.
\end{proof}

%Suppose that $z, s \in \C_+$ satisfies \eref{m_equation} approximately, \emph{i.e.},
%\begin{equation}\label{eq:m_equation_approx}
%\Big(z + \frac{\coa s}{1 + \cob s} + \coc s\Big) + \frac 1 s  = \et,
%\end{equation}
%for some $\et$.
%Suppose that $\Im z + \coc s - \et > 0$. 
%Then by definition, we have 
%\begin{equation}\label{eq:fp}
%s(z) =  \int \Big[ x- (z + \coc s - \et)   \Big]^{-1} \dif\rho_\mathrm{MP}(x),
%\end{equation}
%where $\rho_\mathrm{MP}$ is the MP law associated with the equation $(z + \frac{\coa s}{1 + \cob s} ) + \frac 1 s=0$. 
%From now on, we assume that 
%\begin{equation}\label{eq:fp}
%\eta= \Im z \ge 4  |\et|.
%\end{equation}
%
%Similarly to the previous argument, we have 
% \begin{align} \label{eqn:mfcdef}
%  \int   \frac{ 1}{  |x  - z -  s(z)+ \et|^2 } \varrho_\mathrm{MP}(\dif x) \le  \frac {|\Im \et| +  \Im s (z) } { \eta + \Im s (z) - \Im \et} \le 1
%\end{align} 
%Since 
% \begin{align} %\label{eqn:mfcdef}
%\Big |  \int \frac{  \varrho_\mathrm{MP}(\dif x)  }{  (x  - z -  m)  (x  - z -  s + \et)} \Big | \le  \int \frac{ \varrho_\mathrm{MP}(\dif x)  }{ 2 |x  - z -  m|^2 } +   \int \frac{ \varrho_\mathrm{MP}(\dif x)  }{ 2 |x  - z -  s + \et|^2 } 
%\le 1, 
%\end{align} 
%we have 
% \begin{align} \label{eqn:mfcdef}
%  |s-m| \le  \frac { \et} { \frac { \eta} {2( \eta + \Im m)} +  \frac { \eta- 2 |\Im \et| } { 2(\eta + \Im s (z) - \im \et)}} 
%  \le \et M \eta^{1}.
%\end{align}
%

% !TEX root = equivalence.tex

\section{Proof of Theorem~\ref{thm:equivalence_fg}}
\label{appendix:proof_fg}

Recall the condition \eref{cok_summable_b} in Assumption~\ref{assumption:fg}. For any fixed $c > 0$, we can always choose an integer $L \ge \ell + 1$ such that
\begin{equation}\label{eq:co_conv_g}
\sigma^2 - (\eta^4 c^2/96) \le \textstyle\sum_{0 \le k \le L - 1} \co_k^2 \le \sigma^2.
\end{equation}
Now, we define two polynomial functions:
\begin{equation}\label{eq:finite_h_approx}
\bar f_{d}(x) = \textstyle\sum_{0 \le k \le L-1} \co_k h_k(x) + \hat \co_{L} h_{L}(x)\end{equation}
and
\begin{equation}\label{eq:finite_g_approx}
\hat f_d(x) = \textstyle\sum_{0 \le k \le L-1} \co_k q_k(x) + \hat \co_{L} q_{L}(x).
\end{equation}
In the above expressions, $\hat \co_{L} \bydef (\sigma^2 - \sum_{0 \le k \le L-1} \co_k^2)^{1/2}$, with $\set{\mu_k}$ and $\sigma^2$ being the constants introduced in Assumption~\ref{assumption:fg}. Note that \eref{finite_h_approx} and \eref{finite_g_approx} differ in the polynomials they use, with the former employing the Hermite polynomials $\set{h_k(x)}_k$, while the latter utilizing the Gegenbauer polynomials $\set{q_k(x)}_k$.

Throughout our subsequent discussions, we will adopt a convenient shorthand notation: for any two functions $f_1(x)$ and $f_2(x)$, we will use $\inprod{f_1, f_2} \bydef \EE[f_1(g) f_2(g)]$ and $\norm{f_1} \bydef [\EE f^2_1(g)]^{1/2}$, where $g \sim \mathcal{N}(0, 1)$. Utilizing the estimate provided in \eref{qk_hk}, it is straightforward to confirm that:
\begin{equation}\label{eq:fh_f_diff}
\norms{\bar f_d - \hat f_d} = \mathcal{O}_L(d^{-1}).
\end{equation}
By using the property that the Hermite polynomials are orthonormal [see \eref{ortho_poly_Hermite}], we obtain
\begin{align}
\norms{f_d - \bar f_d}^2 &= \textstyle\sum_{0 \le k \le L-1} (\inprod{f_d, h_k} - \co_k)^2 + (\inprod{f_d, h_L} - \hat \co_L)^2 + (\,\norm{f_d}^2 - \sum_{0 \le k \le L} \inprod{f_d, h_k}^2)\\
&\le (\co_L - \hat \co_L)^2 + (\sigma^2 - \textstyle\sum_{0 \le k \le L} \co_k^2) + \eta^4 c^2/96,\label{eq:fghb_0}
\end{align}
where the last inequality holds for all sufficiently large $d$, due to the conditions \eref{cok_limit_b} and \eref{cosigma_b} given in Assumption~\ref{assumption:fg}. By \eref{co_conv_g}, $\hat \co_L^2 \le \eta^4 c^2/96$ and $\co_L^2 \le \eta^4 c^2/96$. We can then further bound the right-hand side of \eref{fghb_0} as follows:
\begin{equation}\label{eq:f_approx_g}
\norms{f_d - \bar f_d}^2 \le 2 \co_L^2 + 2 \hat \co_L^2 + (\sigma^2 - \textstyle\sum_{0 \le k \le L} \co_k^2) + \eta^4 c^2/96 \le \eta^4 c^2/16.
\end{equation}
Applying the triangular inequality and the estimate presented in \eref{fh_f_diff}, we obtain:
\begin{equation}\label{eq:fg_fdfh}
\norms{f_d - \hat f_d}^2 = \int_{\R} [f_d(x) - \hat f_d(x)]^2 w(x) \dif x  \le \eta^4 c^2 / 8,
\end{equation}
for all sufficient large $d$, where the function $w(x)$ in the integral denotes the probability density function of $\mathcal{N}(0, 1)$. We aim to show that the integral in \eref{fg_fdfh} remains small if we replace $w(x)$ by $\breve{w}_d(x)$. The latter represents the density function of the random variable $\breve{\xi}_d$ defined in \eref{bxi}. To that end, we first express
\begin{equation}\label{eq:w_wdb_small}
\abs{\int_{\R} [f_d(x) - \hat f_d(x)]^2 [w(x) - \breve{w}_d(x)] \dif x} \le 2 \int f_d^2(x) \abs{w(x) - \breve{w}_d(x)} \dif x + 2 \int \hat f_d^2(x) \abs{w(x) - \breve{w}_d(x)} \dif x.
\end{equation}
Due to the condition in \eref{weight_w_wdb}, the first term on the right-hand side of \eref{w_wdb_small} converges to $0$ as $d \to \infty$. For the second term, we observe that $\hat f_d(x)$ is a polynomial with bounded coefficients, and each monomial is a function that satisfies the conditions in Lemma~\ref{lemma:w_wd_wdb} found in Appendix~\ref{appendix:concentration}. Consequently, we can employ \eref{f2_w_wdb} to demonstrate that the second term on the right-hand side of \eref{w_wdb_small} also converges to $0$ as $d \to \infty$. By combining \eref{fg_fdfh} with \eref{w_wdb_small}, we can confirm that
\begin{equation}\label{eq:f_fh_b}
\int_{\R} [f_d(x) - \hat f_d(x)]^2 w_d(x) \dif x  \le \eta^4 c^2 / 4
\end{equation}
for all sufficiently large $d$.

Let $\hat \Gm$ (resp. $\hat A$) be a matrix constructed according to \eref{G_f} [resp. \eref{A_f}] but with the function $f_d$ replaced by $\hat f_d$. The Stieltjes transforms of these matrices are denoted by $s_{\hat \Gm}$ and $s_{\hat A}$, respectively. Due to the decomposition \eref{g_decomp}, the three matrices $\Gm, \hat \Gm$ and $\hat A$ are constructed in the same probability space. Let $m(z) \in \C_+$ be the solution to \eref{m_equation_f}. Using the triangular inequality, we have
\begin{equation}\label{eq:fg_triangle}
\abs{s_{\Gm}(z) - m(z)} \le \abs{s_{\Gm}(z) - s_{\hat \Gm}(z)} + \abs{s_{\hat \Gm}(z) - s_{\hat A}(z)} +\abs{s_{\hat A}(z) - m(z)}.
\end{equation}
Next, we proceed to bound each term on the right-hand side of the above inequality. Let $\breve{\xi}_d$ be the random variable defined in \eref{bxi}. We observe that the off-diagonal elements of $J$ have the same (marginal) distribution as that of $f_d(\breve{\xi}_d)$, and the off-diagonal elements of $\hat J$ have the same (marginal) distribution as that of $\hat f_d(\breve{\xi}_d)$. By applying Lemma~\ref{lemma:s1s2_general} in Appendix~\ref{appendix:resolvent_identities}, we get
\begin{align}
\abs{s_{\Gm}(z) - s_{\hat \Gm}(z)} &\le \frac{1}{\eta^2} [\EE (f_d(\breve{\xi}_d) - \hat f_d(\breve{\xi}_d))^2]^{1/2} + \mathcal{O}_\prec\Big(\frac{1}{\eta\sqrt{n}}\Big)\\
&\le \frac{c}{2} + \mathcal{O}_\prec\Big(\frac{1}{\eta\sqrt{n}}\Big),\label{eq:sgm1}
\end{align}
where the last step follows from \eref{f_fh_b}. To bound the second term on the right-hand side of \eref{fg_triangle}, we note that $\hat f_d$ is a degree-$L$ polynomial. If we write it in terms of the monomial basis, \emph{i.e.}, $\hat f_d = \sum_{0 \le k \le L} c_k x^k$, then the monomial coefficients can be bounded by
\begin{equation}\label{eq:ck_bnd}
\max_{0 \le k \le L} \abs{c_k} \le \max\set{\abs{\co_0}, \abs{\co_1}, \ldots, \abs{\co_{L-1}}, \hat \co_L} \mathcal{O}_L(1) \le \mathcal{O}_L(1),
\end{equation}
where the first inequality is a consequence of the asymptotic consistency of the Gegenbauer polynomials with the Hermite polynomials [see \eref{qk_hk}], and the second inequality follows from the condition \eref{cok_summable_b} in Assumption~\ref{assumption:fg}. By employing \eref{ck_bnd} and invoking Proposition~\ref{prop:sg_polynomial}, we obtain
\begin{equation}\label{eq:sgm2}
\abs{s_{\hat \Gm}(z) - s_{\hat A}(z)} \le \frac{C_L}{d \eta} + \mathcal{O}_\prec\Big( \frac{\max_{0 \le k \le L} \abs{c_k} }{\eta^2 \sqrt{d}}\Big) \prec \frac{1}{\sqrt{d}}.
\end{equation}
Finally, as $\hat f_d$ is a linear combination of $L+1$ Gegenbauer polynomials, we can apply Theorem~\ref{thm:equivalence} to obtain
\begin{equation}\label{eq:sgm3}
\abs{s_{\hat A}(z) - m(z)} \prec \frac{1}{\sqrt{d}}.
\end{equation}
Substituting \eref{sgm1}, \eref{sgm2}, and \eref{sgm3} into \eref{fg_triangle} then gives us
\begin{equation}
\abs{s_\Gm(z) - m(z)} \le c/2 +  \mathcal{O}_\prec\Big(\frac{1}{\sqrt{d}}\Big).
\end{equation}
Thus, for any $D > 0$,
\[
\P(\,\abs{s_{\Gm}(z) - m(z)} > c) < d^{-D}
\]
for all sufficiently large $d$. Choosing any fixed $D > 1$, we can apply the Borel-Cantelli lemma to conclude that $s_{\Gm}(z)$ converges to $m(z)$ almost surely.

\bibliographystyle{abbrvnat}
\bibliography{refs}

\end{document}